\documentclass[10pt,a4paper,twoside]{amsart}

\usepackage[utf8]{inputenc}
\usepackage{lmodern}
\usepackage{textcomp}
\usepackage[T1]{fontenc}
\usepackage{microtype}

\usepackage[
    top=30truemm,
    bottom=30truemm,
    left=30truemm,
    right=30truemm
]{geometry}

\usepackage{xcolor}
\definecolor{dgreen}{rgb}{0.00, 0.40, 0.35}
\definecolor{dblue}{rgb}{0.10, 0.30, 0.65}

\usepackage{graphicx}
\usepackage{tikz}
\usepackage{tikz-cd}

\usepackage{mathtools,amsmath,amssymb,mathrsfs,bm}
\usepackage{amsthm}

\usepackage{multirow,arydshln}
\usepackage{url}
\usepackage{comment}

\usepackage{hyperref}
\hypersetup{
    colorlinks=true,
    citecolor=dgreen,
    linkcolor=dblue,
    urlcolor=black,
}

\usepackage{aliascnt}
\theoremstyle{definition}

\newtheorem{definition}{Definition}[section]

\newaliascnt{theorem}{definition}
\newtheorem{theorem}[theorem]{Theorem}
\aliascntresetthe{theorem}

\newaliascnt{proposition}{definition}
\newtheorem{proposition}[proposition]{Proposition}
\aliascntresetthe{proposition}

\newaliascnt{lemma}{definition}
\newtheorem{lemma}[lemma]{Lemma}
\aliascntresetthe{lemma}

\newaliascnt{fact}{definition}

\aliascntresetthe{fact}

\newaliascnt{remark}{definition}
\newtheorem{remark}[remark]{Remark}
\aliascntresetthe{remark}

\newaliascnt{corollary}{definition}
\newtheorem{corollary}[corollary]{Corollary}
\aliascntresetthe{corollary}

\newaliascnt{example}{definition}
\newtheorem{example}[example]{Example}
\aliascntresetthe{example}

\newtheorem*{conjecture}{Conjecture}

\newaliascnt{convention}{definition}

\aliascntresetthe{convention}

\newaliascnt{observation}{definition}

\aliascntresetthe{observation}

\newaliascnt{problem}{definition}

\aliascntresetthe{problem}

\newaliascnt{setup}{definition}

\aliascntresetthe{setup}

\newaliascnt{formula}{definition}

\aliascntresetthe{theorem}

\newaliascnt{claim}{definition}
\newtheorem{claim}[claim]{Claim}
\aliascntresetthe{claim}

\newtheorem{maintheorem}{Main Theorem}

\usepackage[capitalise,nameinlink,noabbrev]{cleveref}
\crefname{theorem}{Theorem}{Theorems}
\crefname{proposition}{Proposition}{Propositions}
\crefname{lemma}{Lemma}{Lemmata}
\crefname{corollary}{Corollary}{Corollaries}
\crefname{definition}{Definition}{Definitions}
\crefname{remark}{Remark}{Remarks}
\crefname{example}{Example}{Examples}
\crefname{conjecture}{Conjecture}{Conjectures}
\crefname{convention}{Convention}{Conventions}
\crefname{formula}{Formula}{Formulas}
\crefname{claim}{Claim}{Claims}
\crefname{maintheorem}{Main Theorem}{Main Theorems}
\crefname{section}{\S}{\S\S}
\crefformat{section}{#2\S#1#3}
\crefformat{enumi}{#2(#1)#3}

\newcommand{\xto}[1]{\xrightarrow{#1}}
\newcommand{\imm}{\looparrowright}
\newcommand{\into}{\hookrightarrow}

\renewcommand{\le}{\leqslant}
\renewcommand{\ge}{\geqslant}

\renewcommand{\epsilon}{\varepsilon}
\renewcommand{\phi}{\varphi}

\renewcommand{\tilde}{\widetilde}

\newcommand{\A}{\mathcal{A}}
\newcommand{\K}{\mathcal{K}}
\newcommand{\G}{\mathcal{G}}

\newcommand{\Z}{\mathbb Z}

\newcommand{\R}{\mathbb R}
\newcommand{\C}{\mathbb C}

\newcommand{\I}{\textrm{I}}
\newcommand{\II}{\textrm{I\hspace{-1pt}I}}
\newcommand{\III}{\textrm{I\hspace{-1pt}I\hspace{-1pt}I}}

\newcommand{\pt}{\mathrm{pt.}}
\newcommand{\GL}{\mathrm{GL}}

\newcommand{\SO}{\mathrm{SO}}
\renewcommand{\O}{\mathrm{O}}

\newcommand{\U}{\mathrm{U}}

\newcommand{\Imm}{\mathrm{Imm}}
\newcommand{\Hom}{\mathrm{Hom}}

\newcommand{\Ker}{\operatorname{\mathrm{Ker}}}

\renewcommand{\Im}{\operatorname{\mathrm{Im}}}

\newcommand{\Tor}{\operatorname{\mathrm{Tor}}}

\newcommand{\Int}{\operatorname{\mathrm{Int}}}
\newcommand{\id}{\operatorname{\mathrm{id}}}
\newcommand{\pr}{\operatorname{\mathrm{pr}}}
\newcommand{\SI}{\mathrm{SI}}
\newcommand{\Tp}{\mathrm{Tp}}
\newcommand{\FR}{\mathrm{FR}}

\newcommand{\Co}{\mathrm{Co}}

\newcommand{\B}{\mathrm{B}}
\newcommand{\E}{\mathrm{E}}
\newcommand{\BO}{\mathrm{BO}}

\newcommand{\BSO}{\mathrm{BSO}}

\newcommand{\BU}{\mathrm{BU}}

\newcommand{\Diff}{\mathrm{Diff}}

\newcommand{\near}{\mathrm{near}}
\newcommand{\far}{\mathrm{far}}

\newcommand{\prot}{\mathrm{prot}}

\newcommand{\ind}{\operatorname{\mathrm{ind}}}
\newcommand{\sgn}{\operatorname{\mathrm{sgn}}}

\newcommand{\ol}{\overline}
\newcommand{\lk}{\mathrm{lk}}
\newcommand{\rot}{\mathrm{rot}}
\newcommand{\rel}{\textrm{rel}}
\newcommand{\connhom}{\delta^*}
\newcommand{\spacesymbol}{\, \text{\textvisiblespace} \,}

\makeatletter
\@namedef{subjclassname@2020}{\textup{2020} Mathematics Subject Classification}
\makeatother

\title[Thom polynomials relative to prescribed maps around the boundary]{Thom polynomials relative to prescribed maps \\ around the boundary} 
\author{Masato Tanabe}
\address[M.~Tanabe]{RIKEN Center for Interdisciplinary Theoretical and Mathematical Sciences (iTHEMS), RIKEN, Wako 351-0198, Japan}
\email{masato.tanabe@riken.jp}

\begin{document}

\begin{abstract}
    Thom polynomials are universal cohomological obstructions to the appearance of singularities of given types in differentiable maps.
    As an application, various invariants of immersions have been expressed in terms of singularities of their extension maps, known as singular Seifert surfaces.
    To place these results in a unified framework, we aim in this paper to establish a relative version of Thom polynomial theory.
    Our results consist of four parts.
    (1) We introduce Thom polynomials {\it relative to} prescribed maps around the boundary (or a closed codimension-zero submanifold) that avoid singularities of given types.
    (2) We show a structure theorem for Thom polynomials relative to framable immersions. It expresses them as the sum of the term obtained by substituting Kervaire's relative characteristic classes into the absolute Thom polynomial and a universal correction term.
    (3) We determine correction terms in several cases, not only reinterpreting earlier works as instances of relative Thom polynomials but also recovering some of them.
    Most earlier formulas are summarized as the vanishing of correction terms.
    (4) We give suggestive evidence for the relative Thom polynomials of multi-singularity types $A_0^k$ with an application.
\end{abstract}

\maketitle
\thispagestyle{empty}

\tableofcontents

\section{Introduction}

\subsection{Background}

Thom polynomials are universal cohomological obstructions to the appearance of singularities of given types in differentiable maps.
Since the pioneering work of Thom~\cite{Tho55, Tho57}, the theory has provided various formulas for invariants of manifolds in terms of counting singularities of maps on them.
The theory has also been extensively developed in deep connection with enumerative problems arising in complex/algebraic geometry and representation theory.
Recently, introductory expositions on the theory were published by Rim\'anyi and Ohmoto~\cite{Rim24, Ohm26}; we refer the reader to these and the references therein for further background.

Here, we briefly recall the classical Thom polynomial theory, to which we will refer as the {\it absolute} version (\cref{section:ATP}).
Let $\eta$ be a singularity type for maps from $m$-manifolds to $n$-manifolds.
The {\it Thom polynomial} of $\eta$ is the unique polynomial $\Tp(\eta) = \Tp(\eta)(cl_i, cl'_j)$ in universal characteristic classes $cl_i, cl'_j$ ($1 \le i \le m$, $1 \le j \le n$) which satisfies the following universality:
for any compact $m$-manifold $M$, any $n$-manifold $N$, and any generic map $f \colon M \to N$, 
the substitution $cl_i \mapsto cl_i(M), cl'_j \mapsto f^* cl_j(N)$ into $\Tp(\eta)$ gives the Poincar\'e dual to $\ol{\eta(f)} \subset M$, the closure of the $\eta$-singularity locus of $f$.
More precisely, $\Tp(\eta)$ is the equivariant fundamental class of a cycle corresponding to $\eta$ which is invariant under the action associated with the classification of singularities:
\[\Tp(\eta) = \Tp(\eta)(cl_i, cl'_j) \in H^*(\B \G; R) \cong R[cl_i, cl'_j].\]
Indeed, the characteristic classes $cl_i, cl'_j$ and coefficients $R$ depend on the geometric category of singularities as shown in \cref{table:geometric-context}; 
the substitution is the pull-back of $\Tp(\eta)$ along the classifying map of the $K$-jet bundle $J^K(M, N)$ together with the graph map of $f$ ($K$ is a sufficiently large integer).
A large collection of computed polynomials can be found on Rim\'anyi's website~\cite{TPP}.

\begin{table*}[ht]
    \centering
    \renewcommand{\arraystretch}{1.2} 
    \setlength{\tabcolsep}{6pt} 
    \caption{Basic data for Thom polynomials depending on categories}
    \label{table:geometric-context}
    \begin{tabular}{|c||c|c|c|}
        \hline
        category & classifying space $\B \G$ & variables $cl_i, cl'_j$ & coefficients $R$ \\ \hline
        real $C^\infty$ unoriented & $\BO(m) \times \BO(n)$ & Stiefel--Whitney classes & $\Z_2$ \\ \hline
        real $C^\infty$ oriented & $\BO(m) \times \BO(n)$ & Pontryagin classes and $2$-torsion & $\Z$ \\ \hline
        complex analytic &$\BU(m) \times \BU(n)$ & Chern classes & $\Z$ \\ \hline
    \end{tabular}
\end{table*}

The main subject in this paper is counting singularities {\it relative to prescribed conditions}.
It has already appeared in several contexts, such as pure generalizations of results in absolute case and the extendability problem of maps \cite{Mor29,Lev95,Sae20}.
In particular, our motivation arises from studies on immersions.
In~\cite{ESz03}, Ekholm--Sz\H{u}cs introduced the notion of a {\it singular Seifert surface}, an extension of a given immersion to a null-cobordism of the given source manifold, which is possibly singular but non-singular around the boundary; they established a formula for Smale's complete invariant~\cite{Sma59} of sphere immersions of certain dimensions, in terms of counting singularities on a singular Seifert surface. 
This extends Hughes--Melvin's formula~\cite{HM85}, which requires the existence of a non-singular extension for a given immersion.
Furthermore, their viewpoint and method have subsequently been applied to various classes of immersions with concrete applications~\cite{SSzT02, ST02, Juh05, ESz06, Tak04, Tak06, Tak07a, Tak07b, ET11, Tak12, NP15, Kin15, PS24, GP24, PT25, Tan25} (see also~\cite{Pin18} for a survey).
Takase and Ohmoto referred to the formulas as `relative Thom polynomials'~\cite{Tak07a, Tak12, Ohm26}.
However, their precise definition has not yet been given and the results remain scattered.
In this paper, to place them in a unified framework, we aim to establish a relative version of Thom polynomial theory.

\subsection{Main result}

Our result consists of four parts: 
defining relative Thom polynomials, 
showing their structure theorem, 
determining concrete structures in several cases, 
and showing suggestive evidence of relative Thom polynomials for multi-singularities.

\subsubsection{Definition}

In cases in which we are mainly interested, an immersion $\iota \colon V \imm \R^n$ is prescribed.
We extend it slightly in a normal direction and obtain an immersion $\phi \colon V \times [0, \epsilon] \imm \R^n$. We also extend $\phi$ to a possibly singular map $f \colon M \to \R^n$ on a null-cobordism $M$ of $V \times \{0\}$, and count its singularities.
Here, the extension from $\iota$ to $\phi$ is necessary to adjust the dimensions of source manifolds.
Keeping this situation in mind, we begin with the following general setup.

Let $\eta$ be a singularity type for maps from $m$-manifolds to $n$-manifolds, and fix the following as a {\it prescribed datum}:
\begin{itemize}
    \item compact $m$-manifold $M_0$; 
    \item an $n$-manifold $N_0$;
    \item a map $\phi \colon M_0 \to N_0$ having no $\eta$-singularities
\end{itemize}
{\it We embed $M_0$ into $\B \G$ via the graph map of $\phi$ together with the classifying map of $J^K(M_0, N_0)$.}
Then we define the Thom polynomial of $\eta$ {\it relative to $\phi$} as a relative cohomology class
\[
    \Tp(\eta | \phi) \in H^*(\B \G, M_0; R)
\]
so that for any {\it extension datum} of $\phi$, i.e.,
\begin{itemize}
    \item any $m$-manifold $M$ containing $M_0$ as a closed submanifold\footnote{In this paper, a submanifold is said to be {\it closed} if it is a closed subset of the ambient manifold.} such that $\partial M \subset \partial M_0$\footnote{This condition is required to apply the Alexander duality below.}; 
    \item any $n$-manifold $N$ containing $N_0$ as a closed submanifold;
    \item any extension $f \colon M \to N$ of $\phi$ which is generic with respect to $\eta$,
\end{itemize}
the pull-back of $\Tp(\eta | \phi)$ along the classifying map of $J^K(M, N)$ together with the graph map of $f$ gives the Alexander dual to $\ol{\eta(f)} \subset M - M_0$ (\cref{definition:RTP-sing,theorem:rel-locus}; see also \cref{subsection:RTP-obst} for Steenrod's obstruction theoretic description).
Of course, the restriction of $\Tp(\eta | \phi)$ to $H^*(\B \G; R)$ is just $\Tp(\eta)$.
However, $\Tp(\eta | \phi)$ differs crucially from $\Tp(\eta)$ in that it may not form a polynomial because of the non-equivariance of $M_0 \subset \B \G$. Hence, the pull-back may not be a simple substitution of characteristic classes.

\subsubsection{Structure theorem}

The next problem is to describe the structure of relative Thom polynomials.
For this purpose, we employ Kervaire's relative characteristic classes, which are relative versions of the usual characteristic classes for vector bundles (see \cref{section:RCC}; they can also be regarded as relative Thom polynomials).
We introduce the notion of {\it framable prescribed datum}, from which the relative characteristic classes are induced in a quite natural way (\cref{definition:framable-prescribed-data}). In particular, the map $\phi$ is assumed to be non-singular.
Once such a datum is given, we can substitute the induced classes into the absolute Thom polynomial and also compare the resulting relative class to the relative Thom polynomial.
The main result is stated as follows and has three versions, as in the absolute version.
Let $\epsilon^k$ denote the trivial vector bundle of rank $k$ and $\gamma^l$ the tautological bundle over $\BO(l)$ or $\BU(l)$.
Assume that $m \le n$.

\begin{maintheorem}[real $C^\infty$ unoriented version; see \cref{maintheorem:BO-Z2-restated}]\label{maintheorem:BO-Z2}
{\it
    For any $\K$-singularity type $\eta \subset J^K(m, n)$ of codimension $q$ and any framable prescribed datum $\phi \colon M_0 \imm N_0$, there is a unique class $\alpha(\eta | \phi) \in H^{q-1}(M_0; \Z_2)$ satisfying the following: for any frame $\Theta$ on $N_0$,
    \[\Tp(\eta | \phi) = \Tp(\eta) \Big( (\pr_1)^* w_i(\epsilon^{n-m} \oplus \gamma^m | \theta), (\pr_2)^* w_j(\gamma^n | \Theta) \Big) + \connhom \alpha(\eta | \phi)\]
    in $H^q(\BO(m) \times \BO(n), M_0; \Z_2)$, 
    where $\connhom$ is the connecting homomorphism.
}
\end{maintheorem}

\begin{maintheorem}[real $C^\infty$ oriented version; see \cref{maintheorem:BO-Z-restated} and also \cref{remark:2-torsion}]\label{maintheorem:BO-Z}
{\it
    For any cooriented $\K$-singularity type $\eta \subset J^K(m, n)$ of codimension $q$ and any framable prescribed datum $\phi \colon M_0 \imm N_0$, there is a unique class $\alpha(\eta | \phi) \in H^{q-1}(M_0; \Z)$ modulo $2$-torsion satisfying the following: for any frame $\Theta$ on $N_0$,
    \[\Tp(\eta | \phi) \equiv \Tp(\eta) \Big( (\pr_1)^* p_i(\epsilon^{n-m} \oplus \gamma^m | \theta), (\pr_2)^* p_j(\gamma^n | \Theta) \Big) + \connhom \alpha(\eta | \phi)\]
    in $H^q(\BO(m) \times \BO(n), M_0; \Z)$ modulo $2$-torsion.
}
\end{maintheorem}

\begin{maintheorem}[complex analytic version; see \cref{maintheorem:BU-Z-restated}]\label{maintheorem:BU-Z}
{\it
    For any complex $\K$-singularity type $\eta \subset J^K(m, n)$ of codimension $q$ and any framable complex prescribed datum $\phi \colon M_0 \imm N_0$, there is a unique class $\alpha(\eta | \phi) \in H^{2q-1}(M_0; \Z)$ satisfying the following: for any complex frame $\Theta$ on $N_0$,
    \[\Tp(\eta | \phi) = \Tp(\eta) \Big( (\pr_1)^* c_i(\epsilon^{n-m} \oplus \gamma^m | \theta), (\pr_2)^* c_j(\gamma^n | \Theta) \Big) + \connhom \alpha(\eta | \phi)\]
    in $H^{2q}(\BU(m) \times \BU(n), M_0; \Z)$.
}
\end{maintheorem}

As a corollary, for any extension datum $f \colon M \to N$ of $\phi$, the Alexander dual to $\ol{\eta(f)} \subset M - M_0$ is expressed as the sum of the substitution part and the class $\alpha(\eta | \phi)$ (\cref{theorem:exp-imm}).
We call $\alpha(\eta | \phi)$ the {\it correction term of $\Tp(\eta | \phi)$}.
This is independent of the choice of an extension datum and of a frame $\Theta$ on $N_0$. (Hence, so is the substitution part.)
Furthermore, $\alpha(\eta | \phi)$ forms a regular homotopy invariant of $\phi$ (\cref{corollary:correction-reg-htpy-inv}).
We also recover and generalize Takase's formula for self-intersection of singular loci~\cite{Tak12} (\cref{remark:alt-proof-8-dim}).

The proofs of the Main Theorems are completely parallel in the three categories.
Each of them is divided into two steps.
First, we find a splitting of the cohomology long exact sequence for the pair $(\B \G, M_0)$ along the relative characteristic classes.
Second, we compare correction terms with respect to two frames on $N_0$.
The $\K$-invariance of $\eta$ is crucial to the second step.
Our argument also partially applies to the case where $m \ge n$ (\cref{theorem:exp-subm}).

\subsubsection{Determining correction terms}

It is important to determine correction terms for the structure of relative Thom polynomials.
First, we consider the complex $\K$-singularity type $A_1$ in dimensions $(m, n) = (1, 1)$ and show that the correction term $\alpha(A_1 | \phi)$ for a certain $\phi$ vanishes 
(\cref{formula:A1-cpx-m=1}). It recovers the relative version of the Riemann--Hurwitz formula.

We also return to earlier works on singular Seifert surfaces, and apply our framework to them.
We determine correction terms for framable prescribed data of the form $\phi \colon V \times [0, \epsilon] \to \R^n$ ($V$ is a null-cobordant $(m - 1)$-manifold) with the $\K$-singularity types listed in \cref{table:list-sing} (see \cref{subsection:list-sing-types} for the definitions of types). This includes also corresponding works in the case where $m \ge n$.
Then we observe that the correction terms vanish in many cases, and that these vanishings are equivalent to the formulas obtained by the corresponding earlier works, despite differences in their nature and proofs.
For example, in Cases (i) and (ii), the vanishing is shown within our framework (\cref{formula:A1-real,formula:A1-cpx}).
The key is to compare correction terms for two prescribed data (\cref{theorem:correction-compare}).
It partially recovers and generalizes N\'emethi--Pint\'er's formulas for the (complex) Smale invariant of immersions associated with map-germs~\cite{NP15} (\cref{theorem:NP15-real,theorem:NP15}).
The correction terms for Cases (iii)--(ix) are determined by reinterpreting corresponding earlier formulas as those of relative Thom polynomials.
In particular, the vanishing in Case (viii) yields a formula for Takase's cobordism invariant of immersions in terms of classical invariants such as the $\mu$-invariant and Hirzebruch's signature defect (\cref{formula:Sigma2-dim4}).
Only in Case (vii), the vanishing remains unknown. However, the vanishing also yields a similar formula for Ekholm--Sz\H{u}cs--Takase--Saeki's regular homotopy invariant of immersions (\cref{conj-formula:A2-SSzT02}).

\begin{table}[ht]
    \centering
    \renewcommand{\arraystretch}{1.2} 
    \setlength{\tabcolsep}{6pt} 
    \caption{Correction terms for framable prescribed data $\phi \colon V \times [0, \epsilon] \to \R^n$}
    \label{table:list-sing}
    \makebox[\textwidth][c]{
    \begin{tabular}{|c||l|l|l|l||l|l|l|}
        \hline
        Case & type $\eta$ & dimensions $(m, n)$ & bdry.\ $V$ & $\alpha(\eta | \phi) \stackrel{?}{=} 0$ & coeff. & see & relevance \\ \hline
        (i) & $A_1$ & $m \le n$ & sphere & Yes & $\Z_2$ & \cref{formula:A1-real} & \cite{NP15} \\ 
        (ii) &  & $m \le n$ (cpx.\ an.) & sphere & Yes & $\Z$ & \cref{formula:A1-cpx} & \cite{NP15} \\ 
        (iii) &  & $m \ge n$ & any & Yes & $\Z_2$ & \cref{formula:A1-nega} & \cite{Sae20} \\
        (iv) &  & $m \ge n = 1$ & any & Yes & $\Z$ & \cref{formula:A1-fxn} & \cite{Mor29} \\ \hline
        (v) & $A_2$ & $(2k, 2)$, $k \ge 1$ & any & Yes & $\Z_2$ & \cref{formula:A2-subm} & \cite{Lev95} \\
        (vi) &  & $(4k, 6k-1)$, $k \ge 1$ & sphere & Yes & $\Z$ & \cref{formula:A2-ESz03} & \cite{ESz03} \\ 
        (vii) &  & $(4, 5)$ & any & ? & $\Z$ & \cref{formula:A2-SSzT02} & \cite{SSzT02} \\ \hline
        (viii) & $\Sigma^2$ & $(4, 4)$ & any & Yes & $\Z$ & \cref{formula:Sigma2-dim4} & \cite{Tak07a,ET11} \\ \hline
        (ix) & $\Sigma_\FR$ & $(8, 8)$ & any & Yes & $\Z$ & \cref{formula:8-dim} & \cite{Tak12} \\ \hline
    \end{tabular}
    }
\end{table}

These results lead us to the following expectation.

\begin{conjecture}
{\it
    For any $\K$-singularity type $\eta$ of codimension $m$, any null-cobordant $(m - 1)$-manifold $V$, and any framable prescribed datum of the form $\phi \colon V \times [0, \epsilon] \to \R^n$,
    \[\alpha(\eta | \phi) = 0.\]
}
\end{conjecture}

\subsubsection{Observation for multi-singularities}

Multi-singularities of maps have also been deeply studied in the complex algebraic/analytic category~\cite{Ron73, Ron84, Kle87, Rim02, Kaz03, Kaz06, Ohm24}.
Although few are known in the real $C^\infty$ category~\cite{Her75, Ron80, Szu86, ESz02}, real multi-singularities on singular Seifert surfaces play important roles in~\cite{ESz03, SSzT02, Tak04, Tak06}.
In particular, Ekholm--Sz\H{u}cs~\cite{ESz03} essentially showed a relative version of Herbert's multiple-point formula, which enumerates $A_0^k$-points of immersions.
This is suggestive evidence for relative Thom polynomials of multi-singularities.
We give the precise statement and proof of a relative version of Herbert's multiple-point formula (\cref{section:rel-multi}).
We also apply it to generalize Ekholm--Sz\H{u}cs' formula for Ekholm's linking invariant of generic immersions~\cite{Ekh01a} (\cref{section:linking}).
Note that the linking invariant has recently been applied to the study of complex singularities and high-codimensional knots~\cite{PS24, PT25, GT25}.

\vspace{10pt}

It would be interesting to find structure theorems and determine correction terms for other singularity types and broader contexts in Thom polynomial theory (cf.~\cite{Kaz06, Rim24, Ohm24}).
This direction would illuminate the aspect of Thom polynomials as topological obstructions.

\subsection*{Organization}

The rest of this paper is organized as follows.
In \cref{section:ATP}, we recall Thom polynomials for singularities of maps.
In \cref{section:RTP}, we define relative Thom polynomials.
In \cref{section:RCC}, we recall Kervaire's relative characteristic classes.
In \cref{section:main}, we show the structure theorem of relative Thom polynomials and show several properties of correction terms.
In \cref{section:classif-imm}, we recall classification of immersions and fix some conventions.
\cref{section:case-A1,section:case-general-posi,section:case-general-nega} are devoted to determining correction terms for the $\K$-singularity types listed in \cref{table:list-sing}.
In \cref{section:multi-case}, we discuss suggestive evidence of relative Thom polynomials for multi-singularities.
The reader can read \cref{section:multi-case} independently of the other sections.

\subsection*{Conventions}

Throughout this paper, all manifolds and maps are assumed to be of class $C^\infty$ unless specifically noted.
The boundary of an oriented manifold is oriented by the outward normal first convention.
For a locally finite cycle $C$ of codimension $q$ in a space $X$, let $[[C]] \in H^q(X, X - C)$ denote the {\it refined fundamental class} of $C$ (the Thom class); for a subset $A \subset X$ such that $A \cap C = \varnothing$, let $[C] \in H^q(X, A)$ denote the restriction of $[[C]]$ (the Poincar\'e or Alexander dual).
If $C$ consists of finite points, let $\# C$ denote its algebraic number.
For a group $G$, let $\{G\}$ denote the $G$-local system naturally determined by the context.

\subsection*{Acknowledgements}

The author is sincerely grateful to Toru Ohmoto, Masamichi Takase, Naohiko Kasuya, L\'aszl\'o Feh\'er, Gerg\H{o} Pint\'er, Andr\'as Sz\H{u}cs, Tam\'as Terpai, Andr\'as Cs\'epai, and Rich\'ard Rim\'anyi for stimulating discussions during the author's visit and for invaluable comments on earlier versions of this article.
Notably, Toru Ohmoto guided him toward the topic of this paper, and has given him constant encouragement and insightful advice. 
L\'aszl\'o Feh\'er kindly pointed out issues of sign conventions and also suggested to him descriptions of relative Thom polynomials and relative characteristic classes which are clearer than those in earlier versions.
This work was supported by JSPS KAKENHI 25KJ0480.

\subsection*{Use of AI}

The author used ChatGPT 5.5 Thinking/Pro to help identify typographical errors and minor inconsistencies. The author takes full responsibility for the mathematical content in this paper, including all statements and proofs.

\section{Absolute Thom polynomials}\label{section:ATP}

We briefly recall notions and facts on singularity theory of differentiable maps and Thom polynomials.
See~\cite{MN20, Rim24, Ohm26} for details.

\subsection{Conventions}\label{subsection:sing-types}

Let $m$, $n$, and $K$ be positive integers.
Throughout this paper, let $\G$ denote one of Mather's groups
\begin{align*}
    \A &= \Diff(\R^m, 0) \times \Diff(\R^n, 0), \\
    \K &= \{H(x, y) = (H_0(x), H_1(x, y)) \in \Diff(\R^m \times \R^n, (0, 0)) \mid H_1(x, 0) = 0\}.
\end{align*}
These groups act on the space of map-germs $f \colon (\R^m, 0) \to (\R^n, 0)$ by
\begin{align*}
    (\sigma, \tau) \cdot f &= \tau \circ f \circ \sigma^{-1} && ((\sigma, \tau) \in \A), \\
    H \cdot f &= H_1 \circ (\id_{(\R^m, 0)}, f) \circ (H_0)^{-1} && (H = (H_0, H_1) \in \K)
\end{align*}
and on the {\it $K$-jet space} $J^K(m, n)$ in the same way.
Two (jets of) map-germs are said to be {\it $\G$-equivalent} if they lie in the same $\G$-orbit.

\begin{definition}\label{definition:sing-type}
    A $\G$-{\it singularity type of codimension $q$} is a $\G$-invariant and Whitney stratified cycle\footnote{See~\cite{Gor81} for the (co)homology theory of Whitney stratified cycles.}\footnote{By abuse of notation, we often identify a cycle with its support.} $\eta \subset J^K(m, n)$ with $\Z_2$ coefficients which carries a refined fundamental class, i.e., 
    \[H^q(J^K(m, n), J^K(m, n) - \ol{\eta}; \Z_2) \cong \Z_2.\]
\end{definition}

Throughout this paper, we will consider only $\G$-singularity types $\eta \subset J^K(m, n)$ such that every jet in $\eta$ is represented by some $K$-$\G$-determined map-germ.
We will also call an $\A$- or $\K$-singularity type a singularity type for short.

For a singularity type $\eta \subset J^K(m, n)$ and a map $f \colon M \to N$ from an $m$-manifold $M$ to an $n$-manifold $N$, we will use the following terminologies.

\begin{itemize}
    \item Let $J^K(M, N)$ denote the {\it $K$-jet bundle} over $M \times N$, with fiber $J^K(m, n)$ and structure group $\A$.
    \item Let $\eta(M, N)$ denote the subbundle of $J^K(M, N)$ with fiber $\eta$. Note that $\ol{\eta(M, N)}$ admits a stratification induced by that of $\ol{\eta}$.
    \item The map $f$ is said to be {\it generic with respect to $\eta$} if its $K$-jet extension $j^K f \colon M \to J^K(M, N)$ is transverse to every stratum of $\ol{\eta(M, N)}$.
    \item The subset $\eta(f) = (j^K f)^{-1}(\eta(M, N)) \subset M$ is called the {\it $\eta$-locus} of $f$. If $f$ is generic with respect to $\eta$, then the closure $\ol{\eta(f)}$ forms a Whitney stratified cycle of codimension $q$ with $\Z_2$ coefficients.
\end{itemize}

\subsection{Thom polynomials}

We begin with a general setup.
For a Lie group $G$, let $\B G$ denote the classifying space of principal $G$-bundles.
For a $G$-space $X$, let $\B_G X$ denote the Borel construction $\E G \times_G X$.

Let $G$ be a Lie group and $X$ an affine $G$-space.
Let $\eta \subset X$ be a $G$-invariant and Whitney stratified cycle of codimension $q$.
Then it carries its $G$-equivariant refined fundamental class
\[[[\B_G \ol{\eta}]] \in H_G^q(X, X - \ol{\eta}; \Z_2) \cong H^q(\B_G X, \B_G(X - \ol{\eta}); \Z_2)\]
and its restriction (the fundamental class)
\[[\B_G \ol{\eta}] \in H_G^q(X; \Z_2) \cong H^q(\B_G X; \Z_2).\]

\begin{definition}\label{definition:ATP-general}
    The {\it Thom polynomial} of $\eta$ is the class
    \[\Tp(\eta) = (\pi^*)^{-1} [\B_G \ol{\eta}] \in H^q_G(\pt; \Z_2) \cong H^q(\B G; \Z_2),\]
    where $\pi \colon \B_G X \to \B G$ is the projection.
\end{definition}

Our main interest is in the case $G = \G$ and $X = J^K(m, n)$.
In this case, the homotopy equivalence $\G \simeq J^1\G \simeq \O(m) \times \O(n)$ yields the isomorphism
\begin{align*}
    H^*(\B \G; \Z_2) 
    &\cong H^*(\BO(m) \times \BO(n); \Z_2) \\
    &= \Z_2[w_1, \dots, w_m, w'_1, \dots, w'_n],
\end{align*}
where $w_i, w'_j$ are universal Stiefel--Whitney classes of degree $i, j$, respectively. 
Thus, $\Tp(\eta)$ forms a homogeneous polynomial of degree $q$ and we will use the following notations:
\[\Tp(\eta) = \Tp(\eta)(w_i, w'_j).\]

The following characterizes Thom polynomials.

\begin{theorem}[{\cite{Tho55, HK56}}]\label{theorem:ATP}
{\it 
Let $\eta \subset J^K(m, n)$ be a singularity type of codimension $q$.
Then $\Tp(\eta) \in \Z_2[w_i, w'_j]$ is a unique homogeneous polynomial of degree $q$ satisfying the following: 
for any compact $m$-manifold $M$, $n$-manifold $N$, and map $f \colon M \to N$ generic with respect to $\eta$, 
\[ \left[ \ol{\eta(f)} \right] = \Tp(\eta)(w_i(M), f^*w_j(N)) \in H^q(M; \Z_2).\]
}
\end{theorem}

Its proof is well-known, but we write it down as it will be helpful for later arguments.

\begin{proof}
    We consider the diagram
    \[
    \begin{tikzcd}[column sep=2em]
        & {J^K(M, N)} \arrow[d] \arrow[r] & {\B_\G J^K(m, n)} \arrow[d, "\pi"] \\
        M \arrow[r, "{(\mathrm{id}, f)}"'] \arrow[ru, "j^K f"] & M \times N \arrow[r] & \B \G,
    \end{tikzcd}
    \]
    where vertical maps are projections and horizontal maps are classifying maps. This induces the diagram
    \[
    \begin{tikzcd}[column sep=2em]
        & {H^q(J^K(M, N); \Z_2)} \arrow[ld, "(j^K f)^*"'] & {H^q(\B_\G J^K(m, n); \Z_2)} \arrow[l] \\
        H^q(M; \Z_2) & H^q(M \times N; \Z_2) \arrow[u] \arrow[l, "{(\mathrm{id}, f)^*}"] & H^q(\B \G; \Z_2). \arrow[u, "\cong", "\pi^*"'] \arrow[l]
    \end{tikzcd}
    \]
    The pull-back of $[\B_\G \ol{\eta}] \in H^q(\B_\G J^K(m, n); \Z_2)$ to $H^q(M; \Z_2)$ through upper side arrows coincides with $[\ol{\eta(f)}]$ by the genericity of $f$.
    The pull-back of $\Tp(\eta) \in H^q(\B \G; \Z_2)$ to $H^q(M; \Z_2)$ through lower side arrows coincides with $\Tp(\eta)(w_i(M), f^*w_j(N))$ since the classifying map $M \times N \to \B \G$ of $J^K(M, N)$ is homotopic to that of $TM \times TN$.
    The commutativity of the diagram completes the proof.
\end{proof}

\begin{remark}[real $C^\infty$ oriented version]\label{remark:TP-coori}
    If a type $\eta$ is coorientable in $J^K(m, n)$, a choice of a coorientation defines the Thom polynomial $\Tp(\eta)$ in the ring
    \begin{align*}
        H^*(\B \G; \Z) 
        &\cong H^*(\BO(m) \times \BO(n); \Z) \\
        &= \Z \left[p_1, \dots, p_{\lfloor m/2 \rfloor}, p'_1, \dots, p'_{\lfloor n/2 \rfloor} \right] \oplus \Im \beta,
    \end{align*}
    and satisfies the same universality as in \cref{theorem:ATP}.
    Here,
    \[\beta \colon H^{*-1}(\BO(m) \times \BO(n); \Z_2) \to H^*(\BO(m) \times \BO(n); \Z)\] 
    denotes the Bockstein homomorphism.
    Hence, $\Im \beta$ is precisely the $2$-torsion part of $H^*(\BO(m) \times \BO(n); \Z)$, which is generated by monomials of the form $\beta(w_1^{i_1} \cdots w_m^{i_m} (w'_1)^{j_1} \cdots (w'_n)^{j_n})$ (or elements whose mod $2$ reductions are monomials of $w_i$ and $w'_j$).
    However, in this paper, we will often ignore these $2$-torsion elements and express the Thom polynomial of a cooriented singularity type $\eta$ as $\Tp(\eta)(p_i, p_j')$. See also \cref{remark:2-torsion}.
\end{remark}

\begin{remark}[complex analytic version]\label{remark:cpx-Tp}
    The notions introduced above are also properly defined in the complex analytic category, and we have 
    \[\G \simeq J^1\G \simeq \U(m) \times \U(n).\]
    Therefore, the Thom polynomial for a complex singularity type is defined in the ring
    \begin{align*}
        H^*(\B \G; \Z) 
        &\cong H^*(\BU(m) \times \BU(n); \Z) \\
        &= \Z[c_1, \dots, c_m, c'_1, \dots, c'_n],
    \end{align*}
    and satisfies the universality for holomorphic maps between complex manifolds.
\end{remark}

Now, we briefly recall a special property of Thom polynomials for $\K$-singularity types.
The following is also similar for the real $C^\infty$ oriented version and the complex analytic version.
Let $\tilde{w}_k$ denote the quotient $k$-th Stiefel--Whitney classes, namely, 
\[1 + \tilde{w}_1 + \tilde{w}_2 + \cdots = \frac{1 + w'_1 + w'_2 + \cdots}{1 + w_1 + w_2 + \cdots}.\]
Then the substitution $w_i \mapsto w_i(M)$ and $w'_j \mapsto f^*w_j(N)$ induces $\tilde{w}_k \mapsto w_k(f) = w_k(f^*TN - TM)$, where $f^*TN - TM$ is a virtual bundle.

\begin{theorem}[\cite{Tho55, Ohm94, FR04}]\label{theorem:ATP-for-K}
{\it
    For any $\K$-singularity type $\eta$, its Thom polynomial $\Tp(\eta)$ is also a polynomial in $\tilde{w}_i$'s.
    In particular, if $\eta$ is of codimension $q$, then the coefficient of $w_q$ in $\Tp(\eta)$ is the negative of that of $w'_q$.
}
\end{theorem}

\begin{remark}
    Although \cref{theorem:ATP-for-K} is enough for our argments in this paper, Thom polynomials for $\K$-singularity types also have the so-called stability property.
    See \cite{Ohm94, FR04} for details.
\end{remark}

Due to the above facts, the literature usually expresses $\Tp(\eta)$ by $\tilde{w}_i$'s. Furthermore, it also might fit relative Thom polyomials (see \cref{remark:relative-char-class-for-K}).
However, we express $\Tp(\eta)$ by $w_i$, $w'_j$ in this paper for a clearer presentation of the Main Theorem.

\subsection{Examples of singularity types}\label{subsection:list-sing-types}

Here, we list the singularity types we will be concerned with in this paper. All of them are stable $\K$-singularity types.
The explicit forms of their Thom polynomials will be recalled from \cref{section:case-A1} onward.

\begin{enumerate}
    \item $A_\mu \subset J^K(m, n)$, the type defined by the local algebra $\R[[x]] / (x^{\mu + 1})$. This is of codimension $\mu (|m - n| + 1)$. In particular, the closure $\ol{A_1}$ consists of all jets of singular map-germs. Furthermore, the type $A_2 \subset J^K(m, n)$ is coorientable if $m < n$ and $n - m + 1 \equiv 0 \pmod 2$ \cite[Proposition 4.1]{And82}, \cite[Theorem 6]{Rim00}.
    \item $\Sigma^i \subset J^K(m, n)$, the type consisting of jets $J^K f(0)$ such that $\dim \Ker df(0) = i$. This is of codimension $i (|m - n| + i)$. Furthermore, the type $\Sigma^i \subset J^K(m, n)$ is coorientable if both $n - m$ and $i$ are even \cite{Ron71, And82, Rim00}.
    \item $\Sigma_\FR \subset J^K(m, m)$ (notice that the source and target dimensions coincide)\footnote{The naming `$\Sigma_\FR$' is due to the discoverer Feh\'er--Rim\'anyi~\cite{FR02} and Takase's naming~\cite{Tak12}.}, the type defined by a linear combination of specific two local algebras. This is of codimension $8$ and coorientable \cite{FR02}.
\end{enumerate}

\begin{remark}
    The type $A_\mu$ coincides with the higher order Thom--Boardman singularity type $\Sigma^{1, \dots, 1, 0}$ ($1$ repeats $\mu$ times), and it holds that $\ol{A_\mu} = \ol{\Sigma^{1, \dots, 1, 0}} = \ol{\Sigma^{1, \dots, 1}}$.
    In this paper, we adopt the notation $A_\mu$.
\end{remark}

\section{Definition of relative Thom polynomials}\label{section:RTP}

In this section, we introduce the notion of relative Thom polynomials.

\subsection{Relative Thom polynomials as fundamental classes}\label{subsection:RTP-fund}

Again, let $G$ be a Lie group and $X$ an affine $G$-space.
Let $\eta \subset X$ be a $G$-invariant and Whitney stratified cycle of codimension $q$ with $\Z_2$ coefficients (the construction is similar for $\Z$ coefficients if $\eta$ is coorientable).
In addition, fix a subcomplex $S \subset \B G$ with respect to some triangulation, and a section $\sigma \colon S \to \B_G X$ which does not meet $\B_G \ol{\eta}$.
Therefore, the refined fundamental class $[[\B_G \ol{\eta}]]$ induces the restriction
\[[\B_G \ol{\eta}] \in H^q(\B_G X, \sigma(S); \Z_2).\]

\begin{definition}\label{definition:RTP-general}
    The {\it Thom polynomial of $\eta$ relative to $\sigma$} is the class
    \[\Tp(\eta | \sigma) = (\pi^*)^{-1} [\B_G \ol{\eta}] \in H^q(\B G, S; \Z_2),\]
    where $\pi \colon (\B_G X, \sigma(S)) \to (\B G, S)$ is the projection.
\end{definition}

We refer to Thom polynomials in \cref{definition:ATP-general} as {\it absolute}, and to the above class as {\it relative}.

\begin{remark}
    The ring $H^*(\B G, S; \Z_2)$ need not be a polynomial ring since $\sigma(S)$ need not be $G$-invariant.
    Hence, the naming `polynomial' is just suggestive.
\end{remark}

The following is by definition.

\begin{proposition}\label{theorem:rel-to-abs}
{\it
    Let $j \colon (\B G, \varnothing) \to (\B G, S)$ denote the natural inclusion. Then
    \[j^*\Tp(\eta | \sigma) = \Tp(\eta) \in H^q(\B G; \Z_2).\]
}
\end{proposition}

Now, we consider the case where $G = \G$ and $X = J^K(m, n)$.

\begin{definition}\label{definition:prescribed-data}
    Let $\eta \subset J^K(m, n)$ be a singularity type.
    A {\it prescribed datum} for $\eta$ is a map $\phi \colon M_0 \to N_0$, where
    \begin{itemize}
        \item $M_0$ is a compact $m$-manifold; 
        \item $N_0$ is an $n$-manifold;
        \item $\phi \colon M_0 \to N_0$ is generic with respect to $\eta$ and satisfies that $\eta(\phi) = \varnothing$.
    \end{itemize}
\end{definition}

\begin{definition}\label{definition:extension-data}
    Let $\phi \colon M_0 \to N_0$ be a prescribed datum for a singularity type $\eta \subset J^K(m, n)$.
    Then, an {\it extension datum} of $\phi$ is a map $f \colon M \to N$, where
    \begin{itemize}
        \item $M$ is an $m$-manifold containing $M_0$ as a closed submanifold such that $\partial M \subset \partial M_0$;
        \item $N$ is an $n$-manifold containing $N_0$ as a closed submanifold;
        \item $f \colon M \to N$ is an extension of $\phi$ which is generic with respect to $\eta$.
    \end{itemize}
\end{definition}

\begin{example}[singular Seifert surfaces: case 1]\label{example:sing-Seif-surf-1}
    Let $V$ be a closed null-cobordant $(m - 1)$-manifold and $\iota \colon V \imm \R^n$ an immersion with trivial normal bundle.
    Choosing a normal $1$-frame (nowhere vanishing normal vector field), $\iota$ is extended to an immersion $\phi \colon V \times [0, \epsilon] \imm \R^n$.    
    The map $\phi$ forms a prescribed datum for any singularity type.
    Its extension datum $f \colon M \to \R^n$, where $M$ is a null-cobordism of $V$, is called a {\it singular Seifert surface} in earlier works, e.g.,~\cite{ESz03,SSzT02,Tak07a,ET11,Tak12}.
\end{example}

\begin{example}[singular Seifert surfaces: case 2]\label{example:sing-Seif-surf-2}
    Let $V$ be a closed null-cobordant $(m - 1)$-manifold and $\iota \colon V \imm \R^{n - 1}$ an immersion.
    The map $\iota \times \id_{[0, \epsilon]} \colon V \times [0, \epsilon] \imm \R^{n - 1} \times [0, \epsilon]$ forms a prescribed datum for any singularity type.
    Its extension datum $f \colon M \to \R^n_+$, where $M$ is a null-cobordism of $V$, is called a {\it singular Seifert surface} as well, or a {\it singular slice manifold} in~\cite{GP24}.
\end{example}

We will meet their special cases in \cref{section:case-A1,section:case-general-posi}.

\begin{example}[extension problem of maps]\label{example:ext-prob}
    Let $M$ be a closed $m$-manifold and $N$ an $n$-manifold with $m \ge n$.
    Let $S \subset M$ be a closed submanifold and $\phi \colon \nu(S) \to N$ a submersion on its tubular neighborhood.
    The map $\phi$ forms a prescribed datum for any singularity type $\eta$.
    Saeki~\cite{Sae20} studied a special case of the extension problem of $\phi$ to a map $f \colon M \to N$ with a condition on their singular loci.
    We will see it in \cref{section:case-general-nega}.
\end{example}

To the end of this subsection, we fix a prescribed datum $\phi \colon M_0 \to N_0$ for a singularity type $\eta \subset J^K(m, n)$.
Then we have the diagram

\vspace{10pt}

\[
\begin{tikzcd}[column sep=2em]
& {J^K(M_0, N_0)} \arrow[d] \arrow[r] & {\B_\G J^K(m, n)} \arrow[d] \\
M_0 \arrow[r, "{(\mathrm{id}, \phi)}"'] \arrow[ru, "j^K \phi"] & M_0 \times N_0 \arrow[r] & \B \G,
\end{tikzcd}
\]

\vspace{10pt}

\noindent{where} the right square is the pull-back and the classifying map is chosen to be an embedding (this exists and is unique up to isotopy).
For simplicity, {\it the embedded image $M_0 \subset \B \G$ will be denoted by $M_0$ as well}.
Notice that $\B_\G \ol{\eta} \cap j^K \phi(M_0) = \varnothing$ by assumption.

\begin{definition}\label{definition:RTP-sing}
    We define the {\it Thom polynomial of $\eta$ relative to $\phi$} to be the Thom polynomial of $\eta$ relative to $j^K \phi$:
    \[\Tp(\eta | \phi) = \Tp(\eta | j^K \phi) \in H^q(\B \G, M_0; \Z_2)\]
\end{definition}

\begin{lemma}\label{theorem:rel-locus}
{\it 
    For any extension datum $f \colon M \to N$ of $\phi$, the class
    \[ \left[ \ol{\eta(f)} \right] \in H^q(M, M_0; \Z_2)\]
    coincides with 
    the pull-back of $\Tp(\eta | \phi) \in H^q(\B \G, M_0; \Z_2)$
    under the lower horizontal maps in

    \vspace{10pt}

    \[
    \begin{tikzcd}[column sep=2em]
    & {H^q(J^K(M, N), j^K \phi(M_0); \Z_2)} \arrow[ld, "(J^K f)^*"'] & {H^q(\B_\G J^K(m, n), j^K \phi(M_0); \Z_2)} \arrow[l] \\
    H^q(M, M_0; \Z_2) & H^q(M \times N, (\id \times \phi)(M_0); \Z_2) \arrow[u] \arrow[l, "{(\mathrm{id}, f)^*}"] & H^q(\B \G, M_0; \Z_2). \arrow[u] \arrow[l]
    \end{tikzcd}
    \]
}
\end{lemma}

\begin{remark}
    If $\eta \subset J^K(m, n)$ is coorientable, the relative Thom polynomial of $\eta$ is defined in $H^q(\BO(m) \times \BO(n), M_0; \Z)$.
\end{remark}

We also consider the complex analytic version and repeat the definitions of analogous notions.
Let $J^K(m, n)$ denote the jet space for holomorphic germs.

\begin{definition}\label{definition:prescribed-data-BU}
    Let $\eta \subset J^K(m, n)$ be a complex singularity type.
    A {\it complex prescribed datum} for $\eta$ is a holomorphic map $\phi \colon M_0 \to N_0$, where
    \begin{itemize}
        \item $M_0$ is a compact complex $m$-manifold;
        \item $N_0$ is a complex $n$-manifold;
        \item $\phi \colon M_0 \to N_0$ is generic with respect to $\eta$ and satisfies that $\eta(\phi) = \varnothing$.
    \end{itemize}
\end{definition}

\begin{definition}\label{definition:extension-data-BU}
    Let $\phi \colon M_0 \to N_0$ be a complex prescribed datum for a complex singularity type $\eta \subset J^K(m, n)$.
    Then, its {\it complex extension datum} of $\phi$ is a holomorphic map $f \colon M \to N$, where
    \begin{itemize}
        \item $M$ is a compact complex $m$-manifold containing $M_0$ as a closed submanifold such that $\partial M \subset \partial M_0$;
        \item $N$ is a complex $n$-manifold containing $N_0$ as a closed submanifold;
        \item $f \colon M \to N$ is an extension of $\phi$ which is generic with respect to $\eta$.
    \end{itemize}
\end{definition}

Then the relative Thom polynomial $\Tp(\eta | \phi) \in H^{2q}(\BU(m) \times \BU(n), M_0; \Z)$ is defined.

\subsection{Relative Thom polynomials as obstruction classes}\label{subsection:RTP-obst}

There is another description of $\Tp(\eta | \phi)$ which is much closer to Steenrod's obstruction theory~\cite{Ste51}.
This point for the absolute case can be found in Feh\'er--Rim\'anyi~\cite[\S2]{FR04}.
For simplicity, let $\pi_1$ denote the {\it abelianization} of fundamental group.

\subsubsection{Quick review on Steenrod's obstruction theory}

Let $E$ be a fiber bundle over a CW complex $B$ with fiber $F$ and structure group $G$.
Assume that, for some integer $q \ge 1$, the following condition holds:
\begin{itemize}
    \item if $q \ge 2$, then $F$ is $(q-2)$-connected;
    \item if $q = 1$, then $\pi_0(F)$ is a group and $G$ acts on $\pi_0(F)$ by group automorphisms.
\end{itemize}
Then, given a section $\sigma$ of $E$ over a closed subcomplex $S \subset B$, we have the {\it primary obstruction class} of $E$ relative to $\sigma$:
\[\mathfrak{o}_q(E | \sigma) \in H^q(B, S; \{\pi_{q-1}(F)\}),\]
that is the precise obstruction to extending $\sigma$ to the union of $S$ and the $q$-skeleton of $B$ (see \cite[\S\S16.5, 29.4]{Ste51} for details).
Furthermore, given another section $\tau$ over $S$, we have the {\it primary difference class} from $\tau$ to $\sigma$:
\[d(\sigma, \tau) \in H^{q-1}(S; \{\pi_{q-1}(F)\}),\]
that is the precise obstruction to homotoping $\tau$ to $\sigma$ over the $(q-1)$-skeleton of $S$.
It holds that
\[\mathfrak{o}_q(E | \sigma) - \mathfrak{o}_q(E | \tau) = \connhom d(\sigma, \tau)\]
in $H^q(B, S; \{\pi_{q-1}(F)\})$, where $\connhom$ denotes the connecting homomorphism.

\subsubsection{Another description of relative Thom polynomials}

Let $G$ be a Lie group and $X$ an affine $G$-space.
Let $\eta \subset X$ be a $G$-invariant and Whitney stratified cycle of codimension $q$.
Assume that the section $\sigma$ of $\B_G (X - \ol{\eta})$ over a closed subcomplex $S \subset \B G$ is given.
If $q \ge 2$, then 
$F = X - \ol{\eta}$ is $(q-2)$-connected by dimensional reason, 
and the class 
\[\mathfrak{o}_q(\B_G (X - \ol{\eta}) | \sigma) \in H^q(\B G, S; \{\pi_{q-1}(X - \ol{\eta})\})\]
is defined (if $q = 1$, we also require the above conditions).
To compare this class to the Thom polynomial, 
we consider the mod $2$ reduction of the local system $\{\pi_{q-1}(X - \ol{\eta})\}$ as follows.
Since the complement $X - \ol{\eta}$ is $(q-2)$-connected, by the Hurewicz theorem and the Alexander duality, we have
\[
    \pi_{q-1}(X - \ol{\eta}) \cong H_{q-1}(X - \ol{\eta}; \Z) \cong H^0(\ol{\eta}; or),
\]
where $or$ denotes the orientation sheaf on $\ol{\eta}$. 
Then we have the mod $2$ reduction 
\[\pi_{q-1}(X - \ol{\eta}) \otimes \Z_2 \cong H^0(\ol{\eta}; or) \otimes \Z_2 \cong \Z_2\]
(recall that a singularity type carries a fundamental class).
We have the following.

\begin{proposition}\label{theorem:another}
{\it
    It holds that
    \[\Tp(\eta | \sigma) = \mathfrak{o}_q(\B_G (X - \ol{\eta}) | \sigma) \bmod 2 \in H^q(\B G, S; \Z_2).\]
}
\end{proposition}

\begin{proof}
    Extend $\sigma$ to the union of $S$ and the $(q-1)$-skeleton of $\B G$.
    Then the class $\mathfrak{o}_q(\B_G (X - \ol{\eta})) \bmod 2$ is the cochain such that for each $q$-cell $e$ of $\B G$ it attains the linking number of $s(\partial e)$ and $\ol{\eta}$ mod $2$ (here, $s|_{\partial e}$ is regarded as a map to the fiber $X - \ol{\eta}$ via some trivialization on $e$).
    Furthermore, extend $\sigma$ to whole $\B G$ as a section of $\B_G X$ so that it is transverse to $\ol{\eta}$, and let $\tilde{\sigma}$ denote the resulting section.
    Then we have
    \[\mathfrak{o}_q(\B_G (X - \ol{\eta}) | \sigma) \bmod 2 = [\tilde{\sigma}^{-1}(\B_G \ol{\eta})] \in H^q(\B G, S; \Z_2).\]
    The right-hand side can also be computed as follows:
    \begin{align*}
        [\tilde{\sigma}^{-1}(\B_G \ol{\eta})] 
        &= \tilde{\sigma}^* [\B_G \ol{\eta}] \\
        &= (\pi^*)^{-1} [\B_G \ol{\eta}] \\
        &= \Tp(\eta | \sigma),
    \end{align*}
    where $\pi \colon \B_G X \to \B G$ is the projection.
\end{proof}

The following is immediate. 

\begin{proposition}\label{theorem:diff-RTP}
{\it
    Let $\tau \colon S \to \B_G X$ be another section which does not meet $\B_G \ol{\eta}$. Then
    \[\Tp(\eta | \sigma) - \Tp(\eta | \tau) = \connhom d_\eta(\sigma, \tau) \in H^q(\B G, S; \Z_2),\]
    where $d_\eta(\sigma, \tau) \in H^{q-1}(S; \Z_2)$ is the difference class from $\tau$ to $\sigma$ as sections of $\B_G (X - \ol{\eta})$.
}
\end{proposition}

In the case where $G = \G$ and $X = J^K(m, n)$, we have the following.
Namely, for every singularity type, its relative Thom polynomial admits the description as in \cref{theorem:another}.

\begin{proposition}\label{theorem:obstruction-class-OK}
{\it
    Let $\eta \subset J^K(m, n)$ be a singularity type of codimension $q \ge 1$.
    Let $\phi \colon M_0 \to N_0$ be a prescribed datum for $\eta$.
    Then the obstruction class
    $\mathfrak{o}_q(\B_G (X - \ol{\eta}) | j^K \phi) \in H^q(\B G, M_0; \{\pi_{q-1}(X - \ol{\eta})\})$
    is always defined.
}
\end{proposition}

\begin{proof}
    There is nothing to prove if $q \ge 2$.
    The remaining case is only $A_1 \subset J^K(m, m)$, whose codimension is $1$.
    However, we have 
    \[
        J^K(m, m) - \ol{A_1} 
        \simeq J^1(m, m) - \ol{A_1} 
        \simeq \O(m)
    \]
    and $\pi_0(\O(m)) \cong \Z_2$. The compatibility of the action of $\G$ also holds.
\end{proof}

\begin{remark}
    The arguments in this subsection are similar for $\Z$ coefficients if $\eta$ is coorientable.
    Also in the complex analytic category, every singularity type $\eta \subset J^K(m, n)$ of complex codimension $q \ge 1$ defines the obstruction class with $\Z$ coefficients.
    Indeed, the real codimension of $\eta$ is at least $2$.
\end{remark}

\section{Relative characteristic classes}\label{section:RCC}

In this section, we recall Kervaire's relative characteristic classes.
See~\cite{Ker57} for details.
We also relate them to relative Thom polynomials.
Throughout this section, let $B$ be a simplicial complex and $S \subset B$ a subcomplex.

\begin{definition}\label{definition:RCC-Euler}
    Let $E$ be a real oriented vector bundle over $B$ of rank $m$.
    Assume that a $1$-frame (nowhere-vanishing section) $\theta$ of $E$ is given over $S$.
    Then the primary obstruction to extending $\theta$ to $B$ is called {\it Euler class relative to $\theta$}, and denoted by
    \[e(E | \theta) \in H^m(B, S; \{\pi_{m - 1}(\R^m - \{0\})\}) \cong H^m(B, S; \Z).\]
\end{definition}

To consider other types, recall the following two spaces.
The Stiefel manifold $V_{m, r}$ is the space of $r$-frames in $\R^m$.
This is homotopy equivalent to $\O(m) / \O(m - r)$, $(m - r - 1)$-connected, and 
\[\pi_{m - r}(V_{m, r}) \cong
\begin{cases}
    \Z & (m - r > 0 \text{ even, or } r = 1) \\
    \Z_2 & (\text{otherwise}).
\end{cases}\]
The complex Stiefel manifold $W_{m, r}$ is the space of $r$-frames in $\C^m$.
This is homotopy equivalent to $\U(m) / \U(m - r)$, $(2m - 2r)$-connected, and 
\[\pi_{2m - 2r + 1}(W_{m, r}) \cong \Z.\]

\begin{definition}\label{definition:RCC-SW}
    Let $E$ be a real vector bundle over $B$ of rank $m$.
    Assume that an $(m - q + 1)$-frame $\theta$ of $E$ is given over $S$.
    Then for each integer $i \ge q$, the primary obstruction to extending the $(m - i + 1)$-frame $\theta^i$ consisting of the last $(m - i + 1)$ components of $\theta$ to $B$ is defined in $H^i(B, S; \{\pi_{i - 1}(V_{m, m - i + 1})\})$.
    Its mod $2$ reduction is called the {\it $i$-th Stiefel--Whitney class relative to $\theta$}, and denoted by
    \[w_i(E | \theta) \in H^i(B, S; \Z_2).\]
\end{definition}

\begin{definition}\label{definition:RCC-Chern}
    Let $E$ be a complex vector bundle over $B$ of rank $m$.
    Assume that an $(m - q + 1)$-frame $\theta$ of $E$ is given over $S$.
    Then for each integer $i \ge q$, the primary obstruction class to extending the $(m - i + 1)$-frame $\theta^i$ consisting of the last $(m - i + 1)$ components of $\theta$ to $B$ is called the {\it $i$-th Chern class relative to $\theta$}, and denoted by
    \[c_i(E | \theta) \in H^{2i}(B, S; \{\pi_{2i - 1}(W_{m, m - i + 1})\}) \cong H^{2i}(B, S; \Z).\]
\end{definition}

\begin{definition}\label{definition:RCC-Pontryagin}
    Let $E$ be a real vector bundle over $B$ of rank $m$.
    Assume that an $(m - 2q + 1)$-frame $\theta$ of $E$ is given over $S$. This induces an $(m - 2q + 1)$-frame $\theta_\C$ of the complexified bundle $E \otimes \C$.
    Then for each integer $i \ge q$, define the {\it $i$-th Pontryagin class relative to $\theta$} as
    \[p_i(E | \theta) = (-1)^i c_{2i}(E \otimes_\R \C | \theta_\C) \in H^{4i}(B, S; \Z).\]
\end{definition}

\begin{definition}\label{definition:RCN}
    When the base space $B$ is a compact $m$-manifold $M$ and the subspace $S$ is the boundary $\partial M$, we define the {\it relative characteristic number} as
    \[w_m[E | \theta] = \langle w_m(E | \theta), [M, \partial M] \rangle \in \Z_2.\]
    This is similar for other types as well.
\end{definition}

We will use the following properties.
Although we discuss only relative Stiefel--Whitney classes, the other cases are similar.
Hereafter, let $E$ be a real vector bundle over $B$ of rank $m$, and assume that an $(m - q + 1)$-frame $\theta$ of $E$ is given over $S$.

\begin{proposition}
{\it
    It holds that
    \[j^*w_i(E | \theta) = w_i(E)\]
    in $H^i(B; \Z_2)$, where $j \colon (B, \varnothing) \to (B, S)$ is the natural inclusion.
}
\end{proposition}

\begin{proposition}\label{theorem:naturality-rel}
{\it
    Let $f \colon (B', S') \to (B, S)$ be a simplicial map, and consider the induced frame $f^*\theta$ of $f^*E$ over $S'$.
    Then
    \[f^*w_i(E | \theta) = w_i(f^*E | f^*\theta)\]
    in $H^i(B', S'; \Z_2)$.
}
\end{proposition}

For another $(m - q + 1)$-frame $\rho$ of $E$ over $S$ and an integer $i \ge q$, let 
\[d_{w_i}(\theta, \rho) \in H^{i - 1}(S; \Z_2)\]
denote (the mod $2$ reduction of) the difference class from $\theta^i$ to $\rho^i$.
Let $\connhom \colon H^{i - 1}(S; \Z_2) \to H^i(B, S; \Z_2)$ denote the connecting homomorphism.

\begin{proposition}\label{theorem:diff-SW}
{\it
    Assume that another $(m - q + 1)$-frame $\rho$ of $E$ is given over $S$.
    Then
    \[w_i(E | \theta) - w_i(E | \rho) = \connhom d_{w_i}(\theta, \rho)\]
    in $H^i(B, S; \Z_2)$.
}
\end{proposition}

\begin{proposition}\label{theorem:naturality-diff}
{\it
    Assume that another $(m - q + 1)$-frame $\rho$ of $E$ is given over $S$.
    Furthermore, let $f \colon (B', S') \to (B, S)$ be a simplicial map, and consider the induced frames $f^*\theta$, $f^*\rho$ of $f^*E$ over $S'$.
    Then
    \[f^* d_{w_i}(\theta, \rho) = d_{w_i}(f^* \theta, f^* \rho)\]
    in $H^{i - 1}(S'; \Z_2)$.
}
\end{proposition}

The topic of the rest of this section was pointed out by L\'aszl\'o Feh\'er.
We consider the standard action of $\GL(m; \R)$ on $\Hom(\R^{m - q + 1}, \R^m)$.
Let $\Sigma \subset \Hom(\R^{m - q + 1}, \R^m)$ denote the subset consisting of linear maps of rank less than $m - q + 1$.
This forms a $\GL(m; \R)$-invariant and stratified cycle of codimension $q$.

\begin{proposition}\label{theorem:another-RCC-SW}
{\it
    Assume that an $(m - q + 1)$-frame $\theta$ of the tautological bundle $\gamma^m = \B_{\GL(m; \R)} \R^m$ is given over a subcomplex $S \subset \B\GL(m; \R)$.
    Then
    \[w_i(\gamma^m | \theta) = \Tp(\Sigma | \theta^i)\]
    in $H^i(\B\GL(m; \R), S; \Z_2) \cong H^i(\BO(m), S; \Z_2)$.
    In particular, 
    \[w_i = w_i(\gamma^m) = \Tp(\Sigma)\]
    in $H^i(\BO(m); \Z_2)$.
}
\end{proposition}

\begin{proof}
    It is immediate from \cref{theorem:another}.
\end{proof}

We also consider the singularity type $A_1 \subset J^1(m, m - q + 1) = \Hom(\R^m, \R^{m - q + 1})$.
Recall that this is codimension $q$ and endowed with an action of $\G \simeq J^1 \G = \GL(m; \R) \times \GL(m - q + 1; \R)$.
It is well-known that $\Tp(A_1)(w_i, 0) = w_q$, and hence, we have
\[\Tp(A_1)(w_i, 0) = \Tp(\Sigma).\]
This equality can be explained more precisely.
In fact, the subset $\Sigma$ coincides with $A_1$ under transposing linear maps; furthermore, forgetting the $\GL(m - q + 1; \R)$-action on $A_1$, the two subsets form the same refined fundamental class.
Its relative version will be discussed in \cref{subsection:case-A1-nega}.

\section{Structure of relative Thom polynomials}\label{section:main}

In this section, we show a structure theorem of relative Thom polynomials.

\subsection{Real \texorpdfstring{$C^\infty$}{C-infty} unoriented version}\label{subsection:main-BO}

Throughout this subsection, we assume that $m \le n$.

\subsubsection{Setup}\label{subsubsection:setup}

First, we prepare relative characteristic classes to be substituted into absolute Thom polynomials.
We introduce the following notion.

\begin{definition}\label{definition:framable-prescribed-data}
    A {\it framable prescribed datum} is a prescribed datum $\phi \colon M_0 \to N_0$ for the type $A_1$ (hence, $\phi$ is an immersion) such that
    \begin{itemize}
        \item the normal bundle of $\phi \colon M_0 \to N_0$ is trivial;
        \item $N_0$ is parallelizable, i.e., $TN_0$ is trivial.
    \end{itemize}
\end{definition}

For simplicity, we fix a trivialization of the normal bundle of $\phi$, namely, an immersion $\phi$ is assumed to be a framed immersion in usual sense.

We fix a framable prescribed datum $\phi \colon M_0 \imm N_0$ and a frame $\Theta$ on $N_0$.
Then they (and the trivialization of the normal bundle of $\phi$) naturally induce a frame $\theta$ of $\epsilon^{n - m} \oplus TM_0$ via the isomorphism
\[\phi^*TN_0 \cong \epsilon^{n-m} \oplus TM_0.\]
Furthermore, we regard the two frames $\theta$, $\Theta$ as subframes of some universal bundles as follows.
Consider an embedding $\kappa_1 \colon M_0 \into \BO(m)$ which is a classifying map of $TM_0$ (this is unique up to isotopy).
Let $\gamma^l$ denote the tautological bundle over $\BO(l)$.
Then the frame $\theta$ induces a frame of the direct sum $\epsilon^{n-m} \oplus \gamma^m$ over $\kappa_1(M_0)$.
In the same way, consider an embedding $\kappa_2 \colon N_0 \into \BO(n)$; the frame $\Theta$ induces a frame of $\gamma^n$ over $\kappa_2(N_0)$.
Therefore, two types of relative Stiefel--Whitney classes are defined:
\begin{align*}
    w_i(\epsilon^{n-m} \oplus \gamma^m | \theta) &\in H^i(\BO(m), \kappa_1(M_0); \Z_2) & (1 \le i \le m), \\
    w_j(\gamma^n | \Theta)  &\in H^j(\BO(n), \kappa_2(N_0); \Z_2) & (1 \le j \le n).
\end{align*}
Moreover, consider the embedding
\[\kappa \circ (\id, \phi) \colon M_0 \into M_0 \times N_0 \into \BO(m) \times \BO(n),\]
where $\kappa$ is the classifying map of $J^K(M_0, N_0)$.
Since $\kappa \simeq \kappa_1 \times \kappa_2$, we have projections
\begin{align*}
    \pr_1 &\colon (\BO(m) \times \BO(n), M_0) \to (\BO(m), \kappa_1(M_0)), \\
    \pr_2 &\colon (\BO(m) \times \BO(n), M_0) \to (\BO(n), \kappa_2(N_0)).
\end{align*}
Consequently, we have the classes
\begin{align*}
    (\pr_1)^* w_i(\epsilon^{n-m} \oplus \gamma^m | \theta) &\in H^i(\BO(m) \times \BO(n), M_0; \Z_2) & (1 \le i \le m), \\
    (\pr_2)^* w_j(\gamma^n | \Theta)  &\in H^j(\BO(m) \times \BO(n), M_0; \Z_2) & (1 \le j \le n).
\end{align*}
The following is immediate from \cref{theorem:naturality-rel}.

\begin{lemma}
{\it
    For any extension datum $f \colon M \to N$ of $\phi \colon M_0 \to N_0$,
    \begin{align*}
        (\kappa \circ (\id, f))^* (\pr_1)^* w_i(\epsilon^{n-m} \oplus \gamma^m | \theta) &= w_i(\epsilon^{n-m} \oplus TM | \theta) &\in H^i(M, M_0; \Z_2), \\
        (\kappa \circ (\id, f))^* (\pr_2)^* w_j(\gamma^n | \Theta) &= w_j(f^*TN | \Theta) &\in H^j(M, M_0; \Z_2), 
    \end{align*}
    where $\kappa \colon M \times N \to \BO(m) \times \BO(n)$ is the classifying map of $J^K(M, N)$.
}
\end{lemma}

\subsubsection{Main theorem and its proof}

This subsection is devoted to proving the following.

\begin{theorem}[$=$ \cref{maintheorem:BO-Z2}]\label{maintheorem:BO-Z2-restated}
{\it
    For any $\K$-singularity type $\eta \subset J^K(m, n)$ of codimension $q$ and any framable prescribed datum $\phi \colon M_0 \imm N_0$, there is a class $\alpha(\eta | \phi) \in H^{q-1}(M_0; \Z_2)$ satisfying the following: for any frame $\Theta$ on $N_0$, 
    \[\Tp(\eta | \phi) = \Tp(\eta) \Big( (\pr_1)^* w_i(\epsilon^{n-m} \oplus \gamma^m | \theta), (\pr_2)^* w_j(\gamma^n | \Theta) \Big) + \connhom \alpha(\eta | \phi)\]
    in $H^q(\BO(m) \times \BO(n), M_0; \Z_2)$, 
    where $\connhom$ is the connecting homomorphism.
    Furthermore, $\alpha(\eta | \phi)$ is unique modulo $\Ker \connhom$.
}
\end{theorem}

The proof is divided into showing the following two claims.
Let $\eta \subset J^K(m, n)$ be a $\K$-singularity type of codimension $q$ and $\phi \colon M_0 \imm N_0$ a framable prescribed datum.

\begin{claim}\label{claim:quasi-existence}
{\it
    For any frame $\Theta$ on $N_0$, there is a unique class $\alpha(\eta | \phi, \Theta) \in H^{q-1}(M_0; \Z_2)$ satisfying the following:
    \[\Tp(\eta | \phi) = \Tp(\eta) \Big( (\pr_1)^* w_i(\epsilon^{n-m} \oplus \gamma^m | \theta), (\pr_2)^* w_j(\gamma^n | \Theta) \Big) + \connhom \alpha(\eta | \phi, \Theta)\]
    in $H^q(\BO(m) \times \BO(n), M_0; \Z_2)$.
}
\end{claim}

\begin{proof}
    We focus on the cohomology long exact sequence of the pair $(\BO(m) \times \BO(n), M_0)$:
    \begin{align*}
        H^{q-1}(\BO(m) \times \BO(n); \Z_2) \xto{i^*} H^{q-1}(M_0; \Z_2) \ \\ \xto{\connhom} H^q(\BO(m) \times \BO(n), M_0; \Z_2) \xto{j^*} H^q(\BO(m) \times \BO(n); \Z_2) \xto{i^*} H^q(M_0; \Z_2).
    \end{align*}
    We choose a frame $\Theta$ on $N_0$.
    Then the map $j^*$ admits a section
    \[s_\Theta \colon H^q(\BO(m) \times \BO(n); \Z_2) \to H^q(\BO(m) \times \BO(n), M_0; \Z_2)\]
    defined by 
    \begin{align*}
        w_i = (\pr_1)^* w_i(\epsilon^{n-m} \oplus \gamma^m) &\mapsto (\pr_1)^* w_i(\epsilon^{n-m} \oplus \gamma^m | \theta), \\
        w'_j = (\pr_2)^* w_j(\gamma^n) &\mapsto (\pr_2)^* w_j(\gamma^n | \Theta).
    \end{align*}
    In particular, $j^*$ is a surjection, namely, $i^*$ is zero.
    Therefore, we obtain the splitting exact sequence
    \[0 \to H^{q-1}(M_0; \Z_2) \xto{\connhom} H^q(\BO(m) \times \BO(n), M_0; \Z_2) \xto{j^*} H^q(\BO(m) \times \BO(n); \Z_2) \to 0.\]
    Since $\Tp(\eta | \phi)$ and $s_\Theta(\Tp(\eta))$ are mapped to the same class $\Tp(\eta)$ by $j^*$, \cref{claim:quasi-existence} is proven.
\end{proof}

Notice that by this argument, we also saw the following.

\begin{proposition}\label{theorem:inj-conn-hom}
{\it
    For each integer $q$, the connecting homomorphism
    \[\connhom \colon H^{q-1}(M_0; \Z_2) \to H^q(\BO(m) \times \BO(n), M_0; \Z_2)\]
    is an injection.
}
\end{proposition}

\begin{claim}\label{claim:uniqueness}
{\it
    The class $\alpha(\eta | \phi, \Theta)$ is independent of the choice of $\Theta$.
}
\end{claim}

\begin{proof}
    We choose two frames $\Theta$ and $\Xi$ on $N_0$, and put 
    \[D = \connhom (\alpha(\eta | \phi, \Theta) - \alpha(\eta | \phi, \Xi)).\]
    We already have
    \begin{align*}
        \Tp(\eta | \phi) 
        &= \Tp(\eta) \Big((\pr_1)^* w_i(\epsilon^{n-m} \oplus \gamma^m | \theta), (\pr_2)^* w_j(\gamma^n | \Theta) \Big) +  \connhom \alpha(\eta | \phi, \Theta), \\
        &= \Tp(\eta) \Big((\pr_1)^* w_i(\epsilon^{n-m} \oplus \gamma^m | \xi), (\pr_2)^* w_j(\gamma^n | \Xi) \Big) + \connhom \alpha(\eta | \phi, \Xi),
    \end{align*}
    where $\xi$ is the frame of $\epsilon^{n-m} \oplus TM_0$ induced by $(\phi, \Xi)$.
    Hence,
    \begin{align*}
        D &= \Tp(\eta) \Big((\pr_1)^* w_i(\epsilon^{n-m} \oplus \gamma^m | \xi), (\pr_2)^* w_j(\gamma^n | \Xi) \Big) \\
            & \qquad - \Tp(\eta) \Big((\pr_1)^* w_i(\epsilon^{n-m} \oplus \gamma^m | \theta), (\pr_2)^* w_j(\gamma^n | \Theta) \Big).
    \end{align*}
    First, by \cref{theorem:naturality-diff}, 
    \begin{align*}
        w_i(\epsilon^{n-m} \oplus \gamma^m | \xi) &= w_i(\epsilon^{n-m} \oplus \gamma^m | \theta) + \connhom d_{w_i}(\xi, \theta), \\
        w'_j(\gamma^n | \Xi) &= w'_j(\gamma^n | \Theta) + \connhom d_{w_j}(\Xi, \Theta),
    \end{align*}
    where $d_{w_i}(\xi, \theta)$ is the difference class between two $(n - i + 1)$-frames which are defined by the last $(n - i + 1)$ components of $\xi$ and $\theta$, and $d_{w_j}(\Xi, \Theta)$ is defined similarly.
    Hence, we have
    \begin{align*}
        D 
        &= \Tp(\eta) \Big((\pr_1)^* w_i(\epsilon^{n-m} \oplus \gamma^m | \theta) + \connhom d_{w_i}(\xi, \theta), (\pr_2)^* w_j(\gamma^n | \Theta) + \connhom \phi^* d_{w_j}(\Xi, \Theta) \Big) \\
            & \qquad - \Tp(\eta) \Big((\pr_1)^* w_i(\epsilon^{n-m} \oplus \gamma^m | \theta), (\pr_2)^* w_j(\gamma^n | \Theta) \Big).
    \end{align*}
    Second, recall the following fact.

    \begin{lemma}\label{theorem:easy}
    {\it
        Let $(X, A)$ be a space pair and $R$ a ring.
        Then for any classes $u \in H^i(X, A; R)$ and $v \in H^{j-1}(A; R)$, 
        \[u \smile \connhom v = 0 \in H^{i + j}(X, A; R).\]
    }
    \end{lemma}

    \noindent{Using this, we have}
    \begin{align*}
        D &= a(\eta) \cdot \connhom d_{w_q}(\xi, \theta) + b(\eta) \cdot \connhom \phi^* d_{w_q}(\Xi, \Theta) \\
        &= \connhom \big( a(\eta) \cdot d_{w_q}(\xi, \theta) + b(\eta) \cdot \phi^* d_{w_q}(\Xi, \Theta) \big),
    \end{align*}
    where $a(\eta)$ (resp.~$b(\eta)$) is the coefficient of the monomial $w_q$ (resp.~$w'_q$) in $\Tp(\eta)(w_i, w'_j)$.
    Third, by the naturality of difference classes,
    \[d_{w_q}(\xi, \theta) = \phi^* d_{w_q}(\Xi, \Theta).\]
    Hence, we have
    \[D = (a(\eta) + b(\eta)) \cdot \connhom d_{w_q}(\xi, \theta).\]
    Fourth, by the $\K$-invariance of $\eta$ and \cref{theorem:ATP-for-K}, we have
    \[a(\eta) = -b(\eta).\]
    This shows 
    \[\alpha(\eta | \phi, \Theta) \equiv \alpha(\eta | \phi, \Xi)\]
    in $H^{q-1}(M_0; \Z_2)$ modulo $\Ker \connhom$.
    Finally, \cref{theorem:inj-conn-hom} proves \cref{claim:uniqueness}.
\end{proof}

Thus, \cref{maintheorem:BO-Z2-restated} is proven.

\subsubsection{Consequences}

The following is immediate from \cref{maintheorem:BO-Z2-restated}.

\begin{corollary}
{\it
    The first term of the right-hand side in \cref{maintheorem:BO-Z2-restated}, 
    \[\Tp(\eta) \Big( (\pr_1)^* w_i(\epsilon^{n-m} \oplus \gamma^m | \theta), (\pr_2)^* w_j(\gamma^n | \Theta) \Big),\] 
    is independent of the choice of a frame $\Theta$ on $N_0$.
}
\end{corollary}

\begin{remark}\label{remark:relative-char-class-for-K}
    It would be reasonable to define the $i$-th universal quotient Stiefel--Whitney class relative to the framable prescribed datum $\phi$ and the frame $\Theta$, 
    \[\tilde{w}_i(\spacesymbol | \phi, \Theta) \in H^i(\BO(m) \times \BO(n), M_0; \Z_2),\]
    as the $i$-th part of
    \[
        \frac{1 + (\pr_2)^* w_1(\gamma^n | \Theta) + (\pr_2)^* w_2(\gamma^n | \Theta) + \cdots}{1 + (\pr_1)^* w_1(\epsilon^{n-m} \oplus \gamma^m | \theta) + (\pr_1)^* w_2(\epsilon^{n-m} \oplus \gamma^m | \theta) + \cdots}.
    \]
    Then it holds that 
    \[\Tp(\eta) \Big( (\pr_1)^* w_i(\epsilon^{n-m} \oplus \gamma^m | \theta), (\pr_2)^* w_j(\gamma^n | \Theta) \Big) = \Tp_\K(\eta)(\tilde{w}_i(\spacesymbol | \phi, \Theta)),\]
    where $\Tp_\K(\eta)(\tilde{w}_i)$ is the expression of the Thom polynomial of $\eta$ so that $\tilde{w}_i$ is to be substituted by $w_i(f^* TN - TM)$.
\end{remark}

We also give the following notion.

\begin{definition}\label{definition:bdry-term}
    We call the class $\alpha(\eta | \phi) \in H^{q-1}(M_0; \Z_2)$ in \cref{maintheorem:BO-Z2-restated} the {\it correction term} of $\Tp(\eta | \phi)$.
\end{definition}

We see several properties of correction terms.
To the end of this subsection, we fix a singularity type $\eta \subset J^K(m, n)$ of codimension $q$ and a framable prescribed datum $\phi \colon M_0 \imm N_0$.

\begin{theorem}[universality]\label{theorem:exp-imm}
{\it 
    For any extension datum $f \colon M \to N$ of $\phi$ and any frame $\Theta$ on $N_0$, 
    \[ \left[ \ol{\eta(f)} \right] = \Tp(\eta) \Big(w_i(\epsilon^{n-m} \oplus TM | \theta), f^* w_j(TN | \Theta) \Big) + \connhom \alpha(\eta | \phi)\]
    in $H^q(M, M_0; \Z_2)$.
}
\end{theorem}

\begin{proof}
    Consider the diagram
    \[
    \begin{tikzcd}[column sep=2em]
    & {(J^K(M, N), j^K \phi(M_0))} \arrow[r] \arrow[d] & {(\B_\G J^K(m, n), j^K \phi(M_0))} \arrow[d] \\
    {(M, M_0)} \arrow[r, "{(\id, f)}"'] \arrow[ru, "J^K f"] & {(M \times N, (\id, \phi)(M_0))} \arrow[r] & {(\B \G, M_0),}
    \end{tikzcd}
    \]
    and induce the diagram
    \[
    \begin{tikzcd}[column sep=2em]
    & {H^q(J^K(M, N), j^K \phi(M_0); \Z_2)} \arrow[ld, "(J^K f)^*"'] & {H^q(\B_\G J^K(m, n), j^K \phi(M_0); \Z_2)} \arrow[l] \\
    H^q(M, M_0; \Z_2) & H^q(M \times N, (\id, \phi)(M_0); \Z_2) \arrow[u] \arrow[l, "{(\mathrm{id}, f)^*}"] & H^q(\B \G, M_0; \Z_2). \arrow[u] \arrow[l]
    \end{tikzcd}
    \]
    Pull back the identity in \cref{maintheorem:BO-Z2-restated} to $H^q(M, M_0; \Z_2)$ by the lower horizontal maps.
    Then the assertion follows from \cref{theorem:rel-locus,theorem:naturality-rel}.
\end{proof}

\begin{theorem}[comparison of two maps]\label{theorem:correction-compare}
{\it
    Let $\psi \colon M_0 \imm N_0$ be another framable prescribed datum.
    Let $\Theta$ be a frame on $N_0$, and let $\theta_\phi, \theta_\psi$ denote the frames of $\epsilon^{n-m} \oplus \gamma^m|_{M_0}$ corresponding to $\phi, \psi$, respectively.
    Then
    \[\alpha(\eta | \phi) -  \alpha(\eta | \psi) = d_\eta(j^K \phi, j^K \psi) - a(\eta) \cdot d_{w_q}(\theta_\phi, \theta_\psi)\]
    in $H^{q-1}(M_0; \Z_2)$, 
    where $a(\eta)$ is the coefficient of the monomial $w_q$ in $\Tp(\eta)(w_i, w'_j)$.
}
\end{theorem}

\begin{proof}
    By \cref{maintheorem:BO-Z2-restated}, we have
    \begin{align*}
        \Tp(\eta | \phi) &= \Tp(\eta) \Big( (\pr_1)^* w_i(\epsilon^{n-m} \oplus \gamma^m | \theta_\phi), (\pr_2)^* w_j(\gamma^n | \Theta) \Big) + \connhom \alpha(\eta | \phi), \\
        \Tp(\eta | \psi) &= \Tp(\eta) \Big( (\pr_1)^* w_i(\epsilon^{n-m} \oplus \gamma^m | \theta_\psi), (\pr_2)^* w_j(\gamma^n | \Theta) \Big) + \connhom \alpha(\eta | \psi).
    \end{align*}
    Subtracting these equations and using \cref{theorem:diff-RTP}, 
    \begin{align*}
        &\connhom d_\eta(j^K \phi, j^K \psi) \\
        & \quad = \Tp(\eta) \Big((\pr_1)^* w_i(\epsilon^{n-m} \oplus \gamma^m | \theta_\phi), (\pr_2)^* w_j(\gamma^n | \Theta) \Big) \\
        & \quad \qquad - \Tp(\eta) \Big((\pr_1)^* w_i(\epsilon^{n-m} \oplus \gamma^m | \theta_\psi), (\pr_2)^* w_j(\gamma^n | \Theta) \Big) \\
        & \quad \qquad \qquad + \connhom (\alpha(\eta | \phi) - \alpha(\eta | \psi)).
    \end{align*}
    Furthermore, using \cref{theorem:diff-SW} and \cref{theorem:easy}, we also have 
    \begin{align*}
        &\Tp(\eta) \Big((\pr_1)^* w_i(\epsilon^{n-m} \oplus \gamma^m | \theta_\phi), (\pr_2)^* w_j(\gamma^n | \Theta) \Big) \\
        & \quad - \Tp(\eta) \Big((\pr_1)^* w_i(\epsilon^{n-m} \oplus \gamma^m | \theta_\psi), (\pr_2)^* w_j(\gamma^n | \Theta) \Big) \\
        =&\Tp(\eta) \Big((\pr_1)^* w_i(\epsilon^{n-m} \oplus \gamma^m | \theta_\psi) + \connhom d_{w_i}(\theta_\phi, \theta_\psi), (\pr_2)^* w_j(\gamma^n | \Theta) \Big) \\
        & \quad - \Tp(\eta) \Big((\pr_1)^* w_i(\epsilon^{n-m} \oplus \gamma^m | \theta_\psi), (\pr_2)^* w_j(\gamma^n | \Theta) \Big) \\
        =& \connhom a(\eta) \cdot d_{w_q}(\theta_\phi, \theta_\psi).
    \end{align*}
    Finally, \cref{theorem:inj-conn-hom} proves the assertion.
\end{proof}

\begin{corollary}[regular homotopy invariance]\label{corollary:correction-reg-htpy-inv}
{\it 
    Let $\psi \colon M_0 \imm N_0$ be another framable prescribed datum.
    If $\phi$ and $\psi$ are mutually regularly homotopic, then
    \[\alpha(\eta | \phi) = \alpha(\eta | \psi)\]
    in $H^{q-1}(M_0; \Z_2)$.
    Therefore, the series $(\alpha(\eta | \phi))_\eta$, where $\eta$ runs over all singularity types in $J^K(m, n)$, forms a regular homotopy invariant of the immersion $\phi$.
}
\end{corollary}

\begin{proof}
    By the existence of an immersion $\Phi$, both the classes $\connhom d_\eta(j^K \phi, j^K \psi)$ and $\connhom d_{w_q}(\theta_\phi, \theta_\psi)$ vanish.
    Then the assertion follows from \cref{theorem:correction-compare}.
\end{proof}

\begin{corollary}[self-intersection formula]\label{theorem:self-intersection}
{\it
    It holds that
    \[\Tp(\eta | \phi)^2 = \Tp(\eta) \Big((\pr_1)^* w_i(\epsilon^{n-m} \oplus \gamma^m | \theta), (\pr_2)^* w_j(\gamma^n | \Theta) \Big)^2\]
    in $H^{2q}(\BO(m) \times \BO(n), M_0; \Z_2)$.
}
\end{corollary}

\begin{proof}
    This is immediate from \cref{maintheorem:BO-Z2-restated,theorem:easy}.
\end{proof}

\begin{corollary}\label{theorem:diff-class-for-sing}
{\it
    Let $\psi \colon M_0 \imm N_0$ be another framable prescribed datum.
    If $\alpha(\eta | \phi) = \alpha(\eta | \psi) = 0$, then
    \[d_\eta(j^K \phi, j^K \psi) = a(\eta) \cdot d_{w_q}(\theta_\phi, \theta_\psi)\]
    in $H^{q-1}(M_0; \Z_2)$, 
    where $a(\eta)$ is the coefficient of the monomial $w_q$ in $\Tp(\eta)(w_i, w'_j)$.
}
\end{corollary}

\begin{proof}
    This is immediate from \cref{theorem:correction-compare}.
\end{proof}

\subsection{Real \texorpdfstring{$C^\infty$}{C-infty} oriented and complex analytic versions}\label{subsection:main-other-versions}

All results shown in the previous subsection hold in the other categories.
For later use, we pick up versions of some of them.

\subsubsection{Real $C^\infty$ oriented version}

\begin{theorem}[$=$ \cref{maintheorem:BO-Z}]\label{maintheorem:BO-Z-restated}
{\it
    For any cooriented $\K$-singularity type $\eta \subset J^K(m, n)$ of codimension $q$ and any framable prescribed datum $\phi \colon M_0 \imm N_0$, there is a unique class $\alpha(\eta | \phi) \in H^{q-1}(M_0; \Z)$ modulo $2$-torsion satisfying the following: for any frame $\Theta$ on $N_0$,
    \[\Tp(\eta | \phi) \equiv \Tp(\eta) \Big( (\pr_1)^* p_i(\epsilon^{n-m} \oplus \gamma^m | \theta), (\pr_2)^* p_j(\gamma^n | \Theta) \Big) + \connhom \alpha(\eta | \phi)\]
    in $H^q(\BO(m) \times \BO(n), M_0; \Z)$ modulo $2$-torsion.
    Therefore, 
    for any extension datum $f \colon M \to N$ of $\phi$,
    \[ \left[ \ol{\eta(f)} \right] \equiv \Tp(\eta) \Big(p_i(\epsilon^{n-m} \oplus TM | \theta), f^* p_j(TN | \Theta) \Big) + \connhom \alpha(\eta | \phi)\]
    in $H^q(M, M_0; \Z)$ modulo $2$-torsion.
}
\end{theorem}

\begin{remark}\label{remark:2-torsion}
    The ambiguity on modulo $2$-torsion can be removed.
    In fact, in the proof of this version, we define a section $s_\Theta$ of
    \begin{align*}
        j^* \colon H^q(\BO(m) \times \BO(n), M_0; \Z) 
        &\to H^q(\BO(m) \times \BO(n); \Z) \\
        &\qquad \cong \Z \left[p_1, \dots, p_{\lfloor \frac{m}{2} \rfloor}, p'_1, \dots, p'_{\lfloor \frac{n}{2} \rfloor} \right] \oplus \Im \beta
    \end{align*}
    on the $2$-torsion part $\Im \beta$ by
    \begin{align*}
        \beta(w_i) &\mapsto (\pr_1)^* \beta \left( w_i(\epsilon^{n-m} \oplus \gamma^m | \theta) \right), \\
        \beta(w'_j) &\mapsto (\pr_2)^* \beta \left( w_j(\gamma^n | \Theta) \right).
    \end{align*}
    Therefore, there is a precisely unique class $\alpha(\eta | \phi) \in H^{q-1}(M_0; \Z)$ satisfying
    \[\Tp(\eta | \phi) = s_\Theta \Tp(\eta) + \connhom \alpha(\eta | \phi).\]
\end{remark}

\begin{corollary}\label{theorem:self-intersection-BO-Z}
{\it
    It holds that
    \[\Tp(\eta | \phi)^2 \equiv \Tp(\eta) \Big((\pr_1)^* p_i(\epsilon^{n-m} \oplus \gamma^m | \theta), (\pr_2)^* p_j(\gamma^n | \Theta) \Big)^2\]
    in $H^{2q}(\BO(m) \times \BO(n), M_0; \Z)$ modulo $2$-torsion.
}
\end{corollary}

\begin{remark}\label{remark:alt-proof-8-dim}
    This gives an alternative proof of the formula for $[\Sigma^2(f)]$ shown by Takase~\cite{Tak12}.
\end{remark}

\subsubsection{Complex analytic version}

Here, $J^K(m, n)$ denotes the jet space for holomorphic germs.

\begin{theorem}[$=$ \cref{maintheorem:BU-Z}]\label{maintheorem:BU-Z-restated}
{\it
    For any complex $\K$-singularity type $\eta \subset J^K(m, n)$ of codimension $q$ and any framable complex prescribed datum $\phi \colon M_0 \imm N_0$, there is a unique class $\alpha(\eta | \phi) \in H^{2q-1}(M_0; \Z)$ satisfying the following: for any complex frame $\Theta$ on $N_0$,
    \[\Tp(\eta | \phi) = \Tp(\eta) \Big( (\pr_1)^* c_i(\epsilon^{n-m} \oplus \gamma^m | \theta), (\pr_2)^* c_j(\gamma^n | \Theta) \Big) + \connhom \alpha(\eta | \phi)\]
    in $H^{2q}(\BU(m) \times \BU(n), M_0; \Z)$.
    Therefore, 
    for any complex extension datum $f \colon M \to N$ of $\phi$,
    \[ \left[ \ol{\eta(f)} \right] = \Tp(\eta) \Big(c_i(\epsilon^{n-m} \oplus TM | \theta), f^* c_j(TN | \Theta) \Big) + \connhom \alpha(\eta | \phi)\]
    in $H^{2q}(M, M_0; \Z)$.
}
\end{theorem}

\begin{theorem}\label{theorem:correction-compare-cpx}
{\it
    Let $\phi, \psi \colon M_0 \imm N_0$ be two framable complex prescribed data.
    Let $\Theta$ be a complex frame on $N_0$, and let $\theta_\phi, \theta_\psi$ denote the complex frames of $\epsilon^{n-m} \oplus TM_0$ corresponding to $\phi, \psi$.
    Then
    \[\alpha(\eta | \phi) -  \alpha(\eta | \psi) = d_\eta(j^K \phi, j^K \psi) - a(\eta) \cdot d_{c_q}(\theta_\phi, \theta_\psi)\]
    in $H^{2q-1}(M_0; \Z)$.
}
\end{theorem}

\begin{remark}\label{remark:almost-cpx-structure}
    Assume that $\eta = A_1$ and $K = 1$. Then
    \cref{theorem:correction-compare-cpx} remains valid for maps $\phi$, $\psi$ which are holomorphic with respect to different complex structures on $M_0$ if these complex structures are homotopic as almost complex structures.
    Indeed, it suffices to use $\Hom_\C(TM_0, TN_0)$ instead of $J^1(M_0, N_0)$, and the subbundles $A_1(M_0, N_0)$ with respect to homotopic almost complex structures are identified by the homotopy.
\end{remark}

In advance of the forthcoming sections, we give an application to the simplest case.
We consider the complex $\K$-singularity type $A_1 \subset J^1(1, 1)$.
Recall that this is of complex codimension $1$ and has the Thom polynomial
\[\Tp(A_1)(c_i, c'_j) = -c_1 + c'_1.\]

\begin{theorem}\label{formula:A1-cpx-m=1}
{\it
    Let $M_0$, $N_0$ be Riemann surfaces which are both diffeomorphic to $S^1 \times [0, \epsilon]$ as $C^\infty$ manifolds.
    Then, for any framable complex prescribed datum $\phi \colon M_0 \imm N_0$ such that $\phi^{-1}(S^1 \times \{0\}) = S^1 \times \{0\}$,
    \[\alpha(A_1 | \phi) = 0\]
    in $H^1(S^1; \Z)$.
}
\end{theorem}

\begin{proof}
    First, notice that both the complex structures on $M_0, N_0$ are homotopic to the standard one as complex structures, since $\SO(2)/\U(1)$ is one point.
    Furthermore, the correction term is invariant up to regular homotopy of $\phi$.
    Hence, we may assume that the map $\phi$ is the restriction of a holomorphic map $f \colon D^2 \to D^2$, a small holomorphic perturbation of the map 
    \[D^2 \to D^2; \quad z \mapsto z^{\deg \phi}.\]
    (Note that $\deg \phi > 0$ since $\phi$ is holomorphic.)
    Now, let $\Theta$ be the frame on $N_0$ obtained by the complexification of an outward normal vector field.
    Then, by \cref{maintheorem:BU-Z-restated}, we have
    \begin{align*}
        [A_1(f)]
        &= -c_1(TD^2 | \theta) + f^* c_1(TD^2 | \Theta) + \connhom \alpha(A_1 | \phi).
    \end{align*}
    Since $\theta$ and $\Theta$ are outward, by the relative version of the Poincar\'e--Hopf theorem, 
    \begin{align*}
        \# A_1(f) 
        &= -\chi(D^2) + \deg f \cdot \chi(D^2) + \langle \alpha(A_1 | \phi), [\partial D^2] \rangle.
    \end{align*}
    Since $\# A_1(f) = \deg \phi - 1$ and $\deg f = \deg \phi$, we have 
    $\alpha(A_1 | \phi) = 0$.
\end{proof}

\begin{corollary}[relative version of the Riemann--Hurwitz formula]
{\it
    Let $M, N$ be compact Riemann surfaces with boundary $\partial M = \partial N = S^1$.
    Let $f \colon (M, \partial M) \to (N, \partial N)$ be a generic holomorphic map which is an immersion around $\partial M$.
    Then
    \[[A_1(f)] = -c_1(TM | \theta) + f^* c_1(TN | \Theta)\]
    in $H^2(M, \partial M; \Z)$.
    Therefore,  
    \[\# A_1(f) = -\chi(M) + \deg f \cdot \chi(N).\]
}
\end{corollary}

\subsection{In dimensions \texorpdfstring{$m \ge n$}{m-ge-n}}

We try to trace the proof of \cref{maintheorem:BO-Z2-restated} for the case where $m \ge n$, although it is impossible in general.
Let us restart from the setup.

\begin{definition}\label{definition:framable-prescribed-data-subm}
    A {\it framable prescribed datum} is a prescribed datum $\phi \colon M_0 \to N_0$ for the type $A_1$ (then $\phi$ is a submersion) such that $N_0$ is parallelizable, i.e., $TN_0$ admits a global frame.
\end{definition}

We fix a framable prescribed datum $\phi \colon M_0 \to N_0$ and a global frame $\Theta$ on $N_0$.
Then an $n$-frame $\theta$ of $TM_0$ is induced via a homotopically unique choice of splitting of the exact sequence
\[0 \to \Ker d\phi \to TM_0 \xto{d\phi} \phi^*TN_0 \to 0.\]
In the same way as in \cref{subsection:main-BO}, we also have relative Stiefel--Whitney classes:
\begin{align*}
    w_i(\gamma^m | \theta) &\in H^i(\BO(m), \kappa_1(M_0); \Z_2) & (m - n + 1 \le i \le m), \\
    w_j(\gamma^n | \Theta) &\in H^j(\BO(n), \kappa_2(N_0); \Z_2) & (1 \le j \le n),
\end{align*}
where $\kappa_1 \colon M_0 \to \BO(m)$ and $\kappa_2 \colon N_0 \to \BO(n)$ are classifying maps of $TM_0$ and $TN_0$, respectively.
Now, recall that in the proof of \cref{maintheorem:BO-Z2-restated}, it was important to find a nice splitting of a short exact sequence.
However, in the present case, the range of $i$ is insufficient.
To overcome this difficulty, we will discard $w_j(\gamma^n | \Theta)$ and use only $w_i(\gamma^m | \theta)$ as follows.
We choose $\kappa_1$ to be an embedding, but $\kappa_2$ a constant map (this is possible since $N_0$ is parallelizable).
Then a classifying map of the jet bundle $J^K(M_0, N_0)$,
\[\kappa \colon M_0 \to \BO(m) \times \BO(n),\]
can be chosen to be an embedding and decomposed into $\kappa_1$ and the inclusion
\[\iota = (\id, *) \colon (\BO(m), \kappa_1(M_0)) \into (\BO(m) \times \BO(n), \kappa_1(M_0) \times \{*\}).\]
We write the embedded images $\kappa_1(M_0)$ and $\kappa_1(M_0) \times \{*\}$ as $M_0$ for simplicity.
Then we have the following.

\begin{theorem}\label{theorem:exp-subm}
{\it 
    For any singularity type $\eta \subset J^K(m, n)$ of codimension $q$ so that $\Tp(\eta)(w_i, 0)$ belongs to $\Z_2[w_{m - n + 1}, \ldots, w_m]$ and any framable prescribed datum $\phi \colon M_0 \to N_0$, there is a class $\alpha(\eta | \phi, \Theta) \in H^{q-1}(M_0; \Z_2)$ such that
    \[\iota^* \Tp(\eta | \phi) = \Tp(\eta) \big( \hat{w}_i(\gamma^m | \theta), 0 \big) + \connhom \alpha(\eta | \phi, \Theta)\]
    in $H^q(\BO(m), M_0; \Z_2)$, where
    \[\hat{w}_i(\gamma^m | \theta) =
    \begin{cases}
        w_i(\gamma^m | \theta) & (m - n + 1 \le i \le m), \\
        0 & (\text{otherwise}).
    \end{cases}
    \]
    Furthermore, $\alpha(\eta | \phi, \Theta)$ is unique modulo $\Ker \connhom$.
    Therefore, for any extension datum $f \colon M \to N_0$ of $\phi$ with the same target manifold, 
    \[ \left[ \ol{\eta(f)} \right] = \Tp(\eta) \big( \hat{w}_i(TM | \theta), 0 \big) + \connhom \alpha(\eta | \phi, \Theta)\]
    in $H^q(M, M_0; \Z_2)$, where
    \[\hat{w}_i(TM | \theta) =
    \begin{cases}
        w_i(TM | \theta) & (m - n + 1 \le i \le m), \\
        0 & (\text{otherwise}).
    \end{cases}
    \]
}
\end{theorem}

\begin{proof}
    We consider the exact sequence of the pair $(\BO(m), M_0)$:
    \[H^{q-1}(M_0; \Z_2) \xto{\connhom} H^q(\BO(m), M_0; \Z_2) \xto{j^*} H^q(\BO(m); \Z_2).\]
    By the assumption for $\Tp(\eta)$, we have
    \[j^* \iota^* \Tp(\eta | \phi) = \Tp(\eta) \big( \hat{w}_i(\gamma^m | \theta), 0 \big) \in H^q(\BO(m); \Z_2).\] 
    Then the assertion is immediate from the exactness of the sequence.
\end{proof}

\begin{definition}
    We also call the class $\alpha(\eta | \phi, \Theta)$ in \cref{theorem:exp-subm} or its image $\connhom \alpha(\eta | \phi, \Theta)$ the {\it correction term} of $\Tp(\eta | \phi)$ with respect to $\Theta$.
\end{definition}

\section{Smale invariants of immersions and variants}\label{section:classif-imm}

In this section, we briefly recall classification of sphere immersions up to regular homotopy, its complex analogue, and fix some conventions.

\subsection{Smale invariants}

For manifolds $V$ and $W$, let $\Imm[V, W]$ denote the set of regular homotopy classes of immersions $V \imm W$.

\begin{definition}[{\cite{Sma59}; see also \cite[p.246]{Hir59}}]
    For an immersion $\iota \colon S^r \imm \R^n$, define its {\it Smale invariant} as the difference class
    \[\Omega(\iota) = d_{A_1}(j^1 \iota, j^1 \iota_0) \in H^r(S^r; \pi_r(V_{n, r})) \cong \pi_r(V_{n, r}),\]
    where $\iota_0 \colon S^r \imm \R^n$ is the standard inclusion and $V_{n, r} \simeq \SO(n)/\SO(n - r)$ is the Stiefel manifold.
\end{definition}

\begin{theorem}[\cite{Sma59}]
{\it
    The correspondence
    \[\Omega \colon \Imm[S^r, \R^n] \to \pi_r(V_{n, r}); \quad [\iota] \mapsto \Omega(\iota)\]
    is a bijection.
    Furthermore, endowing $\Imm[S^r, \R^n]$ with a group structure by connected sum, $\Omega$ forms a group isomorphism.
}
\end{theorem}

\begin{remark}
    If $n \ge r + 2$, then removing one point from $S^n$ induces a group isomorphism
    \[\Imm[S^r, S^n] \cong \Imm[S^r, \R^n].\]
\end{remark}

We will focus on the following cases in this paper.
Other cases are also computed by Paechter (e.g.,~\cite{Pae56}; see also \cref{remark:high-dim-cpx-Smale-inv}).

\begin{example} 
    For any integer $k \ge 1$, 
    \begin{itemize}
        \item $\Imm[S^k, \R^{2k}] \cong \pi_k(V_{2k, k}) \cong \Z$ if $k$ is even or $k = 1$, otherwise $\Z_2$. The values correspond to the self-intersection numbers;
        \item $\Imm[S^{4k - 1}, \R^{4k + 1}] \cong \pi_{4k - 1}(V_{4k + 1, 4k - 1}) \cong \Z$;
    \end{itemize}
\end{example}



When $r = 4k - 1$ and $n = 4k + 1$, the following is well-known.

\begin{lemma}[{cf.~Kervaire--Milnor \cite[Lemma 2]{KeMi58}}]\label{theorem:KM-coeff}
{\it
    Let $\iota \colon S^{4k - 1} \imm \R^{4k + 1}$ be an immersion.
    Consider the classes
    \begin{align*}
        \Omega(\iota) = d_{A_1}(j^1 \iota, j^1 \iota_0) &\in H^{4k - 1}(S^{4k - 1}; \pi_{4k - 1}(V_{4k + 1, 4k - 1})), \\
        d_{p_k}(\theta_\iota, \theta_{\iota_0}) &\in H^{4k - 1}(S^{4k - 1}; \pi_{4k - 1}(W_{4k + 1, 2k+2})), \\
        p_k(\epsilon^2 \oplus TD^{4k} | \theta_\iota) &\in H^{4k}(D^{4k}, S^{4k - 1}; \pi_{4k - 1}(W_{4k + 1, 2k+2})).
    \end{align*}
    Then they are related by the formula
    \[\Omega(\iota) \xmapsto{(p \circ i)_*} \pm \frac{1}{a_k (2k-1)!} \cdot d_{p_k}(\theta_\iota, \theta_{\iota_0}) \xmapsto{\connhom} \pm \frac{1}{a_k (2k-1)!} \cdot p_k(\epsilon^2 \oplus TD^{4k} | \theta_\iota),\]
    where
    \[a_k = 
    \begin{cases}
    2 & (k \text{ is odd}), \\
    1 & (k \text{ is even}),
    \end{cases}\]
    and 
    \begin{align*}
        i &\colon V_{4k + 1, 4k - 1} \into \SO(4k + 1) \into \U(4k + 1), \\
        p &\colon \U(4k + 1) \to W_{4k + 1, 2k + 2}
    \end{align*}
    are the standard inclusion and projection, respectively.
}
\end{lemma}

The sign ambiguity is due to the choice of generators of the homotopy groups
which are all isomorphic to $\Z$.
In this paper, we choose `$-$' as the sign in \cref{theorem:KM-coeff}.
Indeed, this is compatible with the convention of Ekholm--Takase~\cite[p.254]{ET11} when $r = 3$ (see also \cref{formula:Sigma2-dim4}).
This is also reasonable from the proof of \cref{theorem:opposite-sign-diff-and-frame} below.

\subsection{Complex Smale invariants}\label{subsection:cpx-Smale-inv}

We will consider a complex analogue of Smale invariants.

\begin{definition}\label{definition:nearly-standard}
    Let $J_0$ denote the complex $m$-manifold structure on $S^{2m - 1} \times [0, \epsilon]$ as the standard collar neighborhood of the unit disk $D^{2m} \subset \C^m$.
    We call another complex $m$-manifold structure $J$ on $S^{2m - 1} \times [0, \epsilon]$ {\it nearly standard} if $J$ is homotopic to $J_0$ as an almost complex structure.
\end{definition}

We will see an example in the proof of \cref{theorem:NP15}.

The following notion is based on the work of N\'emethi--Pint\'er~\cite{NP15}.

\begin{definition}\label{definition:cpx-Smale-inv}
    Let $\phi \colon S^{2m - 1} \times [0, \epsilon] \imm \C^n$ be an immersion which is holomorphic with respect to a nearly standard complex structure.
    Then we define its {\it complex Smale invariant} as the difference class
    \[\Omega_\C(\phi) = d_{A_1}(j^1 \phi, j^1 \phi_0)\]
    in $H^{2m - 1}(S^{2m - 1}; \pi_{2m - 1}(W_{n, m})) \cong \pi_{2m - 1}(W_{n, m})$,
    where $\phi_0 \colon S^{2m - 1} \times [0, \epsilon] \into D^{2m} \into \C^n$ is the standard inclusion and two jet sections are regarded as sections under the same almost complex structure by a homotopy.
\end{definition}

We will focus on the case where $n = 2m - 1$ and prepare the following lemma. 

\begin{lemma}\label{theorem:opposite-sign-diff-and-frame}
{\it
    Let $\phi \colon S^{2m - 1} \times [0, \epsilon] \imm \C^{2m - 1}$ be an immersion which is holomorphic with respect to a nearly standard complex structure. 
    Assume that the normal bundle of $\phi$ is trivial as a complex vector bundle.
    Consider the classes
    \begin{align*}
        \Omega_\C(\phi) = d_{A_1}(j^1 \phi, j^1 \phi_0) &\in H^{2m - 1}(S^{2m - 1}; \pi_{2m - 1}(W_{2m - 1, m})), \\
        d_{c_m}(\theta_\phi, \theta_{\phi_0}) &\in H^{2m - 1}(S^{2m - 1}; \pi_{2m - 1}(W_{2m - 1, m})), \\
        c_m(\epsilon^{m - 1} \oplus TD^{2m} | \theta_\phi) &\in H^{2m}(D^{2m}, S^{2m - 1}; \pi_{2m - 1}(W_{2m - 1, m})).
    \end{align*}
    Then they are related by the formula
    \[\Omega_\C(\phi) = - d_{c_m}(\theta_\phi, \theta_{\phi_0}) \xmapsto{\connhom} - c_m(\epsilon^{m - 1} \oplus TD^{2m} | \theta_\phi).\]
}    
\end{lemma}


\begin{proof}
    The latter correspondence is obvious from $c_m(\epsilon^{m - 1} \oplus TD^{2m} | \theta_{\phi_0}) = 0$.
    We show the former one.
    Let $\mathfrak{t} = (t_1, \dots, t_m)$ and $\mathfrak{e} = (e_1, \dots, e_{2m - 1})$ denote the standard complex frame on $M_0$ and $\C^{2m - 1}$, respectively. 
    We also choose a normal frame $\mathfrak{n} = (n_1, \dots, n_{m - 1})$ of $\phi$.
    On the one hand, the jet extension $j^1 \phi$ is regarded as a map 
    \[A \colon M_0 \to W_{2m - 1, m}\]
    along $\mathfrak{t}$ and $\mathfrak{e}$.
    (Hence, $W_{2m - 1, m}$ appears as the moduli of bundle homomorphisms from $TM_0$ to $T\C^{2m - 1}$.)
    We lift this map to a map
    \[\tilde{A} = [B | A] \colon M_0 \to W_{2m - 1, 2m - 1} = \GL(2m - 1; \C)\]
    along $\mathfrak{n}$.
    On the other hand, the frame $\theta_\phi$ of $\epsilon^{m - 1} \oplus TM_0$ was defined by the frame $\mathfrak{e}$ and the isomorphism
    \[\phi^* T\C^{2m - 1} \xto{\cong} \epsilon^{m - 1} \oplus TM_0.\]
    Notice that this isomorphism is represented by $\tilde{A}^{-1}$.
    Hence, the frame $\theta_\phi$ is regarded as the map
    \[\tilde{A}^{-1} \colon M_0 \to W_{2m - 1, 2m - 1} = \GL(2m - 1; \C),\]
    that is the pointwise inverse of the map $\tilde{A}$.
    The class $d_{c_m}(\theta_\phi, \theta_{\phi_0})$ was defined using the last $m$ vectors of $\theta_\phi$ (and $\theta_{\phi_0}$), i.e., the last $m$ columns of $\tilde{A}^{-1}$.
    (Hence, $W_{2m - 1, m}$ appears as the moduli of bundle homomorphisms from the subbundle of $T\C^{2m - 1}$ generated by $(e_m, \dots, e_{2m - 1})$ to $\epsilon^{m - 1} \oplus TM_0$.)

    Now, we compare the two classes $d_{A_1}(j^1 \phi, j^1 \phi_0)$ and $d_{c_m}(\theta_\phi, \theta_{\phi_0})$.
    They are regarded as the homotopy classes of maps $M_0 \to W_{2m - 1, m}$ and admit the lifts $\tilde{A}, \tilde{A}^{-1} \colon M_0 \to \GL(2m - 1; \C)$, respectively.
    Here, recall that the group structure of the homotopy group of a Lie group is expressed by the pointwise multiplication of two maps.
    Hence, the lifts $\tilde{A}, \tilde{A}^{-1}$ are negatives of each other as elements of $\pi_{2m - 1}(\GL(2m - 1; \C))$.
    This relationship descends to that as elements of $\pi_{2m - 1}(W_{2m - 1, m})$, since the standard projection $\GL(2m - 1; \C) \to W_{2m - 1, m}$ induces the injection
    \[\pi_{2m - 1}(\GL(2m - 1; \C)) \to \pi_{2m - 1}(W_{2m - 1, m}),\]
    which is regarded as the injection
    \[\Z \to \Z; \quad a \mapsto \pm (m - 1)! \cdot a.\]
    (The choice of this sign does not affect the relationship between the two classes being negatives of each other.)
    This completes the proof.
\end{proof}

\begin{remark}\label{remark:high-dim-cpx-Smale-inv}
    Only when $m = 2$, the complex Smale invariant $\Omega_\C(\Phi)$ is related to the Smale invariant of the immersion $\Phi|_{\partial \mathfrak{D}}^{\partial \mathfrak{B}} \colon S^3 \to S^5$ by
    \[\Omega_\C(\Phi) = -\Omega(\Phi|_{\partial \mathfrak{D}}^{\partial \mathfrak{B}}),\]
    under explicit choices of generators of $\pi_3(W_{3,2}) \cong \Z$ and $\pi_3(V_{5,3}) \cong \Z$~\cite{NP15}.
    When $m > 2$, the corresponding group $\pi_{2m - 1}(W_{2m - 1,m})$ is always isomorphic to $\Z$, while $\pi_{2m - 1}(V_{4m-3, 2m - 1})$ is isomorphic to $\Z_2 \oplus \Z_2$ if $m$ is odd and $\Z_4$ if $m$ is even~\cite{Pae56}. 
\end{remark}

\section{Determining correction terms I: Type \texorpdfstring{$A_1$}{A1} in dimensions \texorpdfstring{$m \le n$}{m-le-n}}\label{section:case-A1}

From \cref{section:case-A1,section:case-general-posi,section:case-general-nega}, we determine correction terms of relative Thom polynomials for $\K$-singularity types listed in \cref{table:list-sing} (see also \cref{subsection:list-sing-types}).
We will always consider the case where a prescribed source manifold is $M_0 = V \times [0, \epsilon]$, a target manifold is $N_0 = N = \R^n$.
Then a frame $\Theta$ on $N$ exists homotopically uniquely and extends to a homotopically unique frame on $N$. These frames can be chosen to be standard ones.

In this section, we apply the results in \cref{subsection:main-BO} to the simplest type $A_1 \subset J^K(m, n)$, where $m \le n$ and give an application.

\subsection{Case (i): real \texorpdfstring{$C^\infty$}{C-infty} case}\label{subsection:posi}

Let $m, n$ be integers with $1 \le m \le n$.
Recall that $A_1 \subset J^K(m, n)$ is of codimension $n - m + 1$ and its closure $\ol{A_1} = \ol{\Sigma^1}$ consists of jets whose $1$-jet part has rank less than $m$.
It is known that for any generic map $f \colon M \to N$, 
\[\Tp(A_1)(f) = w_{n - m + 1}(f) = w_{n - m + 1}(f^*TN - TM).\]
Indeed, the locus $\ol{A_1(f)}$ is the singular point set of $f$, which is precisely the primary obstruction for the virtual bundle $f^*TN - TM$ to be the genuine bundle of rank $n - m$.
In particular, 
\begin{align*}
    \Tp(A_1)(w_i, 0) 
    &= \bar{w}_{n - m + 1} \\
    &= \text{the } (n - m + 1) \text{-th part of } (1 + w_1 + w_2 + \cdots)^{-1} \\
    &= w_{n - m + 1} + \text{sum of products of lower degree terms}.
\end{align*}

\begin{theorem}\label{formula:A1-real}
{\it
    Let $\phi \colon S^{m - 1} \times [0, \epsilon] \imm \R^n$ be an immersion with trivial normal bundle.
    Then
    \[\alpha(A_1 | \phi) = 0\]
    in $H^{n - m}(S^{m - 1}; \Z_2)$.
}
\end{theorem}

This assertion is obvious when $n \ne m, 2m - 1$ by dimensional reason.

\begin{proof}
    Let $\phi_0 \colon S^{m - 1} \times [0, \epsilon] \into D^m \into \R^n$ denote the standard inclusion.
    By \cref{theorem:correction-compare}, 
    \[\alpha(A_1 | \phi) - \alpha(A_1 | \phi_0) = d_{A_1}(j^1 \phi, j^1 \phi_0) + d_{w_m}(\theta_\phi, \theta_{\phi_0}).\]
    We see that $\alpha(A_1 | \phi_0) = 0$ as follows.
    Let $f_0 \colon D^m \to \R^n$ denote the standard inclusion.
    Since $f_0$ is a non-singular extension of $\phi_0$, by \cref{theorem:exp-imm},
    \begin{align*}
        \connhom \alpha(A_1 | \phi_0) 
        &= [A_1(f_0)] - \Tp(A_1)(w_i(\epsilon^{n - m} \oplus TD^m | \theta_0), 0) \\
        &= 0 -\Tp(A_1)(0, 0) \\
        &= 0.
    \end{align*}
    Then \cref{theorem:inj-conn-hom} yields that $\alpha(A_1 | \phi_0) = 0$.

    When $n = 2m - 1$, 
    the classes $d_{A_1}(j^1 \phi, j^1 \phi_0)$, $d_{w_m}(\theta_\phi, \theta_{\phi_0})$ are defined in the mod $2$ reduction of the groups
    \[H^{m - 1}(S^{m - 1}; \{\pi_{m - 1}(J^1(m, 2m - 1) - \ol{A_1})\}) \cong H^{m - 1}(S^{m - 1}; \{\pi_{m - 1}(V_{2m - 1, m})\})\]
    and it holds that
    \[d_{A_1}(j^1 \phi, j^1 \phi_0) = d_{w_m}(\theta_\phi, \theta_{\phi_0}).\]
    Also when $n = m$, 
    the classes $d_{A_1}(j^1 \phi, j^1 \phi_0)$, $d_{w_1}(\theta_\phi, \theta_{\phi_0})$ are defined in the mod $2$ reduction of the groups
    \[H^0(S^{m - 1}; \{\pi_0(J^1(m, m) - \ol{A_1})\}) \cong H^0(S^{m - 1}; \{\pi_0(\O(m))\})\]
    and it holds that
    \[d_{A_1}(j^1 \phi, j^1 \phi_0) = d_{w_1}(\theta_\phi, \theta_{\phi_0}).\]
    This completes the proof.
\end{proof}

We should note that how restrictive the assumption that the normal bundle of $\phi$ is trivial is.
See the following two remarks.

\begin{remark}
    By Hirsch's lemma, for any immersion $\phi \colon S^{m - 1} \times [0, \epsilon] \imm \R^{2m - 1}$ with trivial normal bundle, there is an immersion $\iota \colon S^{m - 1} \imm \R^{2m - 2}$ with trivial normal bundle such that $\phi$ is regularly homotopic to the suspension $\iota \times \id_{[0, \epsilon]} \colon S^{m - 1} \times [0, \epsilon] \imm \R^{2m - 2} \times [0, \epsilon]$.
\end{remark}

\begin{remark}\label{remark:normal-bundle-triviality-A1-real}
    Let $\iota \colon S^{m - 1} \imm \R^{2m - 2}$ be an immersion.
    If $m = 2$, $4$, or $8$, then the normal bundle of $\iota$ is always trivial. Indeed, every orientable vector bundle over $S^{m - 1}$ of rank $m - 1$ is trivial in these cases.
    Otherwise, the normal bundle of $\iota$ is trivial if and only if $\iota$ is regularly homotopic to the standard inclusion.
    This fact can be seen as follows.
    When $m$ is even and $m \neq 2, 4, 8$, there is an immersion $S^{m - 1} \imm \R^{2m - 2}$ whose normal bundle is isomorphic to the non-trivial bundle $TS^{m - 1}$ (e.g., the Whitney sphere, which is an immersion with only one transverse double point). Such an immersion forms only one regular homotopy class which is not regularly homotopic to the standard inclusion.
    When $m$ is odd, the normal Euler class of $\iota$ is an integer and coincides with twice the algebraic number of multiple points of $\iota$ (this fact follows from Herbert's multiple-point formula, see \cref{theorem:HR-abs}).
    Since the algebraic number of multiple points of $\iota$ is equal to the Smale invariant of $\iota$, the fact follows.
\end{remark}

For later use, we focus on the case where $n = 2m - 1$ and also show the following incomplete form.

\begin{proposition}\label{formula:A1-real-benri}
    {\it
    Let $\iota \colon S^{m - 1} \to \R^{2m - 2}$ be an immersion with trivial normal bundle.
    Then
    \[\alpha(A_1 | \iota \times \id_{[0, \epsilon]}) = \Omega(\iota) + d_{w_m}(\theta_\iota, \theta_{\iota_0})\]
    in $H^{m - 1}(S^{m - 1}; \Z_2)$,
    where $\iota_0$ is the standard inclusion $S^{m - 1} \into \R^m \into \R^{2m - 2}$, and where $\theta_\iota$ (resp.~$\theta_{\iota_0}$) is the frame of $\epsilon^{m - 1} \oplus TS^{m - 1}$ induced by $\iota$ (resp.~$\iota_0$).
}
\end{proposition}

\begin{proof}
    Put $\phi = \iota \times \id_{[0, \epsilon]}$ and $\phi_0 = \iota_0 \times \id_{[0, \epsilon]}$.
    Notice that their normal bundles are trivial by assumption.
    By \cref{theorem:correction-compare}, 
    \[\alpha(A_1 | \phi) - \alpha(A_1 | \phi_0) = d_{A_1}(j^1 \phi, j^1 \phi_0) + d_{w_m}(\theta_\phi, \theta_{\phi_0}).\]
    We have already seen that $\alpha(A_1 | \phi_0) = 0$.
    We see that
    \[d_{A_1}(j^1 \phi, j^1 \phi_0) = d_{A_1}(j^1 \iota, j^1 \iota_0) = \Omega(\iota)\]
    as follows.
    Recall that 
    \begin{align*}
        d_{A_1}(j^1 \phi, j^1 \phi_0) &\in H^{m - 1}(S^{m - 1} \times [0, \epsilon]; \pi_{m - 1}(J^1(m, 2m - 1) - \ol{A_1})), \\
        d_{A_1}(j^1 \iota, j^1 \iota_0) &\in H^{m - 1}(S^{m - 1} \times [0, \epsilon]; \pi_{m - 1}(V_{2m - 2,m - 1})).
    \end{align*}
    These classes were regarded as classes in the same cohomology via the isomorphisms
    \begin{align*}
        \pi_{m - 1}(J^1(m, 2m - 1) - \ol{A_1}) 
        &\cong \pi_{m - 1}(V_{2m - 1, m}) \\
        &\cong \pi_{m - 1}(V_{2m - 2,m - 1})
    \end{align*}
    Here, the second isomorphism is induced by the suspension
    \[\Imm[S^{m - 1}, \R^{2m - 2}] \to \Imm \left[ S^{m - 1} \times [0, \epsilon], \R^{2m - 2} \times [0, \epsilon] \right]; \quad \iota \mapsto \iota \times \id_{[0, \epsilon]}.\]
    It is clear that this isomorphism maps $d_{A_1}(j^1 \iota, j^1 \iota_0)$ to $d_{A_1}(j^1 \phi, j^1 \phi_0)$.
    In the same way, it is shown that $d_{w_m}(\theta_\phi, \theta_{\phi_0}) = d_{w_m}(\theta_\iota, \theta_{\iota_0})$.
    This completes the proof.
\end{proof}

\subsection{Case (ii): complex analytic case}\label{subsection:case-complex-1}

Again let $m, n$ be integers with $1 \le m \le n$, and consider the $\K$-singularity type $A_1 \subset J^K(m, n)$ in the complex jet space.
Recall that Thom polynomials of complex singularity types are polynomials in Chern classes.
It is known that for any generic holomorphic map $f \colon M \to N$, 
\[\Tp(A_1) = c_{n - m + 1}(f),\]
in particular,
\begin{align*}
    \Tp(A_1)(c_i, 0) 
    &= \bar{c}_{n - m + 1} \\
    &= -c_{n - m + 1} + \text{sum of products of lower degree terms}.
\end{align*}
We also consider the case where $n = 2m - 1$.

\begin{theorem}\label{formula:A1-cpx}
{\it
    Let $\phi \colon S^{2m - 1} \times [0, \epsilon] \imm \C^{2m - 1}$ be an immersion with trivial normal bundle which is holomorphic with respect to a nearly standard complex structure on $S^{2m - 1} \times [0, \epsilon]$.
    Then
    \[\alpha(A_1 | \phi) = 0\]
    in $H^{2m - 1}(S^{2m - 1}; \Z)$.
}
\end{theorem}

\begin{proof}
    We may assume that the degree of jets is $K = 1$ since the type $A_1$ is determined only by $1$-jet part.
    By the assumption that the source of $\phi$ is nearly standard and \cref{remark:almost-cpx-structure}, we can apply \cref{theorem:correction-compare-cpx} to $\phi$ and $\phi_0$.
    Then we have
    \begin{align*}
        \alpha(A_1 | \phi) - \alpha(A_1 | \phi_0)
        &= d_{A_1}(j^K \phi, j^K \phi_0) - (- d_{c_m}(\theta_\phi, \theta_{\phi_0})) \\
        &= \Omega_\C(\phi) + d_{c_m}(\theta_\phi, \theta_{\phi_0}) \\
        &= 0.
    \end{align*}
    The last equality is by \cref{theorem:opposite-sign-diff-and-frame}.

    On the other hand, let $f_0 \colon D^{2m} \into \C^m \into \C^{2m - 1}$ denote the standard inclusion.
    Then, by \cref{maintheorem:BU-Z-restated}, we have
    \begin{align*}
        \connhom \alpha(A_1 | \phi_0) 
        &= [A_1(f_0)] - \Tp(A_1)(c_i(\epsilon^{m - 1} \oplus TD^{2m} | \theta_0), 0) \\
        &= 0 - \Tp(A_1)(0, 0) \\
        &= 0.
    \end{align*}
    Hence, $\alpha(A_1 | \phi_0) = 0$ and this completes the proof.
\end{proof}

\begin{remark}\label{remark:normal-bundle-triviality-cpx}
    Let $\phi \colon M_0 = S^{2m - 1} \times [0, \epsilon] \imm \C^{2m - 1}$ be an immersion which is holomorphic with respect to a nearly standard complex structure.
    Then the following are equivalent:
    \begin{enumerate}
        \item the normal bundle of $\phi$ is trivial as a complex vector bundle;
        \item the differential $d\phi \colon M_0 \to W_{2m - 1, m}$ lifts to a map $M_0 \to \GL(2m - 1; \C)$;
        \item $\Omega_\C(\phi)$ is a multiple of $(m - 1)!$
    \end{enumerate}
    (see the proof of \cref{theorem:opposite-sign-diff-and-frame}).
\end{remark}

The following incomplete form will fit later use.

\begin{proposition}\label{formula:A1-cpx-benri}
{\it
    Let $\phi \colon S^{2m - 1} \times [0, \epsilon] \imm \C^{2m - 1}$ be an immersion with trivial normal bundle which is holomorphic with respect to a nearly standard complex structure on $S^{2m - 1} \times [0, \epsilon]$.
    Then
    \[\alpha(A_1 | \phi) = \Omega_\C(\phi) + d_{c_m}(\theta_\phi, \theta_{\phi_0})\]
    in $H^{2m - 1}(S^{2m - 1}; \Z)$.
}
\end{proposition}

\subsection{Relevance to N\'emethi--Pint\'er's formula}\label{subsection:relationship-to-NP15}

\cref{formula:A1-real,formula:A1-cpx} partially recover and generalize N\'emethi--Pint\'er's Smale invariant formulas for immersions associated with $\A$-finite map-germs.

Let $\Phi \colon (\R^m, 0) \to (\R^{2m - 1}, 0)$ be a map-germ which is singular only at the origin.
Take a sufficiently small ball $\mathfrak{B} \subset \R^{2m - 1}$ centered at the origin so that $\mathfrak{D} = \Phi^{-1}(\mathfrak{B})$ is diffeomorphic to the $m$-disk.
Then one has the immersion $\Phi|_{\partial \mathfrak{D}}^{\partial \mathfrak{B}} \colon \partial \mathfrak{D} = S^{m - 1} \imm \partial \mathfrak{B} = S^{2m - 2}$ and its Smale invariant 
\[\Omega(\Phi|_{\partial \mathfrak{D}}^{\partial \mathfrak{B}}) \in \pi_{m - 1}(V_{2m - 2, m - 1}) \cong 
\begin{cases}
    \Z_2 & (m > 2 \text{ even}), \\
    \Z & (m \text{ is odd or } m = 2).
\end{cases}
\]
We show the following.

\begin{theorem}\label{theorem:NP15-real}
{\it 
    Let $\Phi \colon (\R^m, 0) \to (\R^{2m - 1}, 0)$ be a map-germ which is $\A$-finite and singular only at the origin. 
    If the normal bundle of $\Phi|_{\partial \mathfrak{D}}^{\partial \mathfrak{B}}$ is trivial, then
    \[C(\Phi) = \Omega(\Phi|_{\partial \mathfrak{D}}^{\partial \mathfrak{B}}) \bmod 2 \in \Z_2,\]
    where $C(\Phi)$ is the number of $A_1$-points of a generic perturbation of $\Phi$.
}
\end{theorem}

\begin{remark}
    The number $C(\Phi)$ is finite by the assumption that $\Phi$ is $\A$-finite.
\end{remark}

\begin{remark}
    \cref{theorem:NP15-real} partially recovers N\'emethi--Pint\'er's formula \cite[Proposition 11.1.1]{NP15} (\cref{remark:normal-bundle-triviality-A1-real} explains in what sense this is partial).
    Their formula is valid without the assumption that the normal bundle of $\Phi|_{\partial \mathfrak{D}}^{\partial \mathfrak{B}}$ is trivial, and is also valid as an equality of integers when $m$ is odd or $m = 2$.
    Their proof focuses on the geometric relationship between $A_1$-points (Whitney umbrella) and self-intersections of a generic perturbation of $\Phi$, and uses the coincidence of the Smale invariant with the self-intersection number for $\Phi|_{\partial \mathfrak{D}}^{\partial \mathfrak{B}}$.
    On the other hand, our following proof is based on obstruction theoretic argument.
\end{remark}

\begin{proof}[Proof of \cref{theorem:NP15-real}]
    Take a sufficiently small ball $\mathfrak{B} \subset \R^{2m - 1}$ centered at the origin so that $\mathfrak{D} = \Phi^{-1}(\mathfrak{B})$ is diffeomorphic to the $m$-disk.
    After that, take a generic perturbation of $\Phi$ on $\mathfrak{D}$ which does not change the regular homotopy class of the boundary immersion.
    We also remove one point from $\partial \mathfrak{B}$.
    Then let $f \colon \mathfrak{D} \to \R^{2m - 1}$ denote the resulting map.
    Furthermore, put $\iota = \Phi|_{\partial \mathfrak{D}}^{\partial \mathfrak{B}}$ and $\phi = \Phi|_{\partial \mathfrak{D}}^{\partial \mathfrak{B}} \times \id_{[0, \epsilon]}$.
    Note that $f$ is an extension datum of the immersion $\phi$.
    
    By \cref{formula:A1-real-benri,theorem:exp-imm}, we have
    \begin{align*}
        [A_1(f)] 
        &= \Tp(A_1)(w_i(\epsilon^{m - 1} \oplus T\mathfrak{D} | \theta), 0) + \connhom \Omega(\iota) + \connhom d_{w_m}(\theta_\iota, \theta_{\iota_0})
    \end{align*}
    in $H^m(\mathfrak{D}, \partial \mathfrak{D}; \Z_2)$.
    We also have
    \begin{align*}
        \Tp(A_1)(w_i(\epsilon^{m - 1} \oplus T\mathfrak{D} | \theta), 0)
        &= w_m(\epsilon^{m - 1} \oplus T\mathfrak{D} | \theta) \\
        &= \connhom d_{w_m}(\theta_\phi, \theta_{\phi_0}) \\
        &= \connhom d_{w_m}(\theta_\iota, \theta_{\iota_0}).
    \end{align*}
    Therefore, 
    $[A_1(f)] = \connhom \Omega(\Phi|_{\partial \mathfrak{D}}^{\partial \mathfrak{B}}) \bmod 2$ and this completes the proof.
\end{proof}

Next, we consider the complex analytic case.
Let $\Phi \colon (\C^m, 0) \to (\C^{2m - 1}, 0)$ be a holomorphic map-germ which is singular only at the origin.
Take a sufficiently small ball $B \subset \C^{2m - 1}$ centered at the origin so that $\mathfrak{D} = \Phi^{-1}(B)$ is diffeomorphic to the unit $2m$-disk $D^{2m}$.
Since this diffeomorphism can be obtained by an isotopy in $\C^m$, a collar neighborhood $\nu(\partial \mathfrak{D})$ of $\partial \mathfrak{D}$ is nearly standard.
Then we have the complex Smale invariant of $\Phi|_{\nu(\partial \mathfrak{D})}$: 
\[\Omega_\C(\Phi|_{\nu(\partial \mathfrak{D})}) \in H^{2m - 1}(\partial \mathfrak{D}; \pi_{2m - 1}(W_{2m - 1, m})) \cong \Z.\]

\begin{theorem}\label{theorem:NP15}
{\it
    Let $\Phi \colon (\C^m, 0) \to (\C^{2m - 1}, 0)$ be a holomorphic map-germ which is singular only at the origin.
    If the normal bundle of $\Phi|_{\nu(\partial \mathfrak{D})}$ is trivial, then
    \[C(\Phi) = \Omega_\C(\Phi|_{\nu(\partial \mathfrak{D})}) \in \Z,\]
    where $C(\Phi)$ is the number of $A_1$-points of a generic holomorphic perturbation of $\Phi$.
}
\end{theorem}

\begin{remark}
    The number $C(\Phi)$ is finite since $\Phi$ is $\A$-finite by Mather--Gaffney's criterion (cf.~\cite[Theorem 4.5]{MN20}).
    This is also computed from $\Phi$ in a completely algebraic way~\cite{Mon95, NP15}.
\end{remark}

\begin{remark}
    \cref{theorem:NP15} recovers N\'emethi--Pint\'er's formula~\cite[Theorems 5.1.1]{NP15}, and hence~\cite[Theorem 1.2.2]{NP15} (see~\cref{remark:high-dim-cpx-Smale-inv}).
    Indeed, their formula is stated for the case where $m = 2$, and in this case, we can omit the assumption that the normal bundle of $\Phi|_{\nu(\partial \mathfrak{D})}$ is trivial (see~\cref{remark:normal-bundle-triviality-cpx}).
    The proof in~\cite{NP15} uses local representation of $A_1$-singularities.
    Our following proof uses a similar argument to \cref{theorem:NP15-real}.
\end{remark}

\begin{proof}[Proof of \cref{theorem:NP15}]
    Take a sufficiently small ball $\mathfrak{B} \subset \C^{2m - 1}$ centered at the origin so that $\mathfrak{D} = \Phi^{-1}(\mathfrak{B})$ is diffeomorphic to the $2m$-disk.
    After that, take a generic holomorphic perturbation of $\Phi$ on $\mathfrak{D}$ which does not change the regular homotopy class of the boundary immersion. Let $f \colon \mathfrak{D} \to \C^{2m - 1}$ denote the resulting map.

    By \cref{formula:A1-cpx-benri}, we have 
    \[[A_1(f)] = -c_m(\epsilon^{m - 1} \oplus T \mathfrak{D} | \theta) + \connhom \Omega_\C(\phi) + \connhom d_{c_m}(\theta_\phi, \theta_{\phi_0})\]
    in $H^{2m}(\mathfrak{D}, \partial \mathfrak{D}; \Z)$, 
    where 
    $\phi = \Phi|_{\nu(\partial \mathfrak{D})}$
    and where
    $\phi_0 \colon S^{2m - 1} \times [0, \epsilon] \to \C^{2m - 1}$ is the restriction of the standard inclusion.
    We also have 
    \begin{align*}
        c_m(\epsilon^{m - 1} \oplus T \mathfrak{D} | \theta) 
        &= c_m(\epsilon^{m - 1} \oplus T \mathfrak{D} | \theta_\phi) - c_m(\epsilon^{m - 1} \oplus T \mathfrak{D} | \theta_{\phi_0}) \\
        &= \connhom d_{c_m}(\theta_\phi, \theta_{\phi_0}).
    \end{align*}
    Therefore, 
    \begin{align*}
        \# A_1(f) 
        &= \langle [A_1(f)], [\mathfrak{D}, \partial \mathfrak{D}] \rangle \\
        &= \langle \connhom \Omega_\C(\phi), [\mathfrak{D}, \partial \mathfrak{D}] \rangle \\
        &= \Omega_\C(\phi).
    \end{align*}
    This completes the proof.
\end{proof}

\section{Determining correction terms \texorpdfstring{\II}{II}: Further types in dimensions \texorpdfstring{$m \le n$}{m-le-n}}\label{section:case-general-posi}

In this section, we deal with $\K$-singularity types in dimensions $m \le n$.
We reinterpret several earlier results from the viewpoint of relative Thom polynomials, and determine the corresponding correction terms.

\subsection{Cases (vi) and (vii): type \texorpdfstring{$A_2$}{A2} in dimensions \texorpdfstring{$(m, n) = (4k, 6k-1)$}{(m,n)=(4k,6k-1)}}\label{subsection:A2_m<n}

Recall that for an integer $k \ge 1$, the type $A_2 \subset J^K(4k, 6k-1)$ is a coorientable $\K$-singularity type of codimension $4k$ and has the Thom polynomial
\begin{align*}
    \Tp(A_2)(p_i, 0) 
    &= \bar{p}_k \\
    &= -p_k + \text{sum of products of lower degree terms}
\end{align*}
\cite[Lemma 3]{Szu00}.

\begin{theorem}\label{formula:A2-ESz03}
{\it 
    Let $\phi \colon S^{4k - 1} \times [0, \epsilon] \imm \R^{4k + 1}$ be an immersion, and let $j \colon \R^{4k + 1} \into \R^{6k-1}$ denote the standard inclusion.
    Then
    \[\alpha(A_2 | j \circ \phi) = 0\]
    in $H^{4k - 1}(S^{4k - 1}; \Z)$.
}
\end{theorem}

Notice that the normal bundle of $\phi$, and in turn that of $j \circ \phi$, is trivial.

\begin{proof}
    Choose an arbitrary extension datum $f \colon M \to \R^{6k-1}$ of $j \circ \phi$, where $M$ is an oriented null-cobordism of $S^{4k - 1}$.
    Then, by Ekholm--Sz\H{u}cs~\cite[Theorem 1.1(a)]{ESz03},
    \[\Omega(\phi|_{S^{4k - 1}}) = \frac{1}{a_k (2k-1)!} (-\bar{p}_k[\hat{M}] + \# A_2(f)),\]
    where $\hat{M}$ is a $4k$-manifold obtained by gluing $M$ and $D^{4k}$ by an orientation-reversing diffeomorphism between the boundaries.
    Let $\theta$ (resp.~$\theta_\phi$) denote the frame of $\epsilon^{2k-1} \oplus T(S^{4k - 1} \times [0, \epsilon])$ (resp.~$\epsilon \oplus T(S^{4k - 1} \times [0, \epsilon])$) induced by the standard frame on $\R^{6k-1}$ and $j \circ \phi$ (resp.~$\phi$).
    Then, by the additivity of characteristic numbers,
    \begin{align*}
        \bar{p}_k[\hat{M}] 
        &= \bar{p}_k[\epsilon^{2k-1} \oplus T \hat{M}] \\
        &= \bar{p}_k[\epsilon^{2k-1} \oplus TM | \theta] - \bar{p}_k[\epsilon^{2k-1} \oplus TD^{4k} | \theta].
    \end{align*}
    We also have
    \begin{align*}
        - \bar{p}_k[\epsilon^{2k-1} \oplus TD^{4k} | \theta]
        &= p_k[\epsilon^{2k-1} \oplus TD^{4k} | \theta] \\
        &= p_k[\epsilon \oplus TD^{4k} | \theta_\phi] \\
        &= - a_k (2k-1)! \cdot \Omega(\phi|_{S^{4k - 1}}),
    \end{align*}
    where the last equality is by the sign convention fixed right after~\cref{theorem:KM-coeff}.
    Combining these three equalities, we have
    \[[A_2(f)] = \Tp(A_2)(p_i(\epsilon^{2k-1} \oplus TM | \theta), 0).\]
    By \cref{maintheorem:BO-Z-restated} and the fact that $H^{4k - 1}(S^{4k - 1}; \Z)$ is torsion-free, we obtain the assertion.
\end{proof}

\begin{remark}
    Every immersion $\iota \colon S^{4k - 1} \imm \R^{4k + 1}$ has a trivial normal bundle and the choice of its trivialization is unique up to homotopy.
    Hence, $\iota$ is extended to a unique immersion $\phi \colon S^{4k - 1} \times [0, \epsilon] \imm \R^{4k + 1}$ up to regular homotopy.
\end{remark}

When $k = 1$, we also obtain a more general formula.
Let $V$ be a closed oriented $3$-manifold and $\iota \colon V \imm \R^5$ an immersion with trivial normal bundle.
Let $\theta$ denote the frame of $\epsilon^2 \oplus TV$ induced by $\iota$.
We recall the following integer invariants.

\begin{itemize}
    \item 
        For a frame $\theta'$ of $\epsilon^1 \oplus TV$, Hirzebruch's {\it signature defect} was defined to be 
        \[h(V, \theta') = p_1[TM | \theta'] - 3\sigma(M),\]
        where $M$ is an oriented null-cobordism of $V$.
        This is independent of the choice of $M$. See~\cite{KiMe99} for details.
        For any frame $\theta$ of $\epsilon^2 \oplus TV$, there is a frame $\theta'$ of $\epsilon^1 \oplus TV$ such that $1 \oplus \theta'$ is homotopic to $\theta$ as a frame of $\epsilon^2 \oplus TV$, by the Hirsch lemma; it holds that 
        $p_1[\epsilon \oplus TM | \theta] = p_1[TM | \theta']$.
        Thus, one can consider
        \[h(V, \theta) = p_1[\epsilon \oplus TM | \theta] - 3\sigma(M)\]
        and we will call this the {\it signature defect} of $(V, \theta)$.        
    \item 
        Put
        \[\tau(V) = \dim_{\Z_2} \Tor H_1(V; \Z) \otimes_\Z \Z_2.\]
        Saeki--Sz\H{u}cs--Takase constructed a regular homotopy invariant 
        \[i_a(\iota) = \frac{3}{2}(\sigma(M) - \tau(V)) + \frac{1}{2} \# A_2(f),\]
        where $M$ is an oriented null-cobordism of $V$ and $f \colon M \to \R^5$ is a generic extension of $\iota$ which is an immersion around $\partial M = V$.
        This is always an integer and independent of the choice of $M$ and $f$.
        Furthermore, the pair of $i_a$ and the so-called Wu invariant is a complete invariant for $\Imm[V, \R^5]$.
        See~\cite{SSzT02} for details.
\end{itemize}
Saeki--Sz\H{u}cs--Takase's invariant $i_a$ is a generalization of the Smale invariant based on Ekholm--Sz\H{u}cs's formula (see the proof of \cref{formula:A2-ESz03}). 
For this reason, this invariant is called the {\it Smale-type invariant} in~\cite{Tan25} or the {\it ESzTS invariant} in~\cite{GP24}.
By regarding these invariants as classes of $H^3(V; \Z) \cong \Z$, we obtain the expression of the correction term as follows.

\begin{theorem}\label{formula:A2-SSzT02}
{\it
    Let $V$ be a closed oriented $3$-manifold and $\phi \colon V \times [0, \epsilon] \imm \R^5$ an immersion with trivial normal bundle.
    Then
    \[\alpha(A_2 | \phi) = 2 i_a(\phi|_V) + 3\tau(V) + h(V, \theta)\]
    in $H^3(V; \Z)$.
}
\end{theorem}

\begin{proof}[Proof of \cref{formula:A2-SSzT02}]
    Choose an extension datum $f \colon M \to \R^5$ of $\phi$, where $M$ is an oriented null-cobordism of $V$.
    Then
    \begin{align*}
        \# A_2(f) 
        &= -3(\sigma(M) - \tau(V)) + 2 i_a(\phi|_V) \\
        &= -p_1[\epsilon \oplus TM | \theta] + h(V, \theta) + 3\tau(V) + 2 i_a(\phi|_V)
    \end{align*}
    by the definitions of the invariants.
    Since $H^3(V; \Z)$ is torsion-free, we obtain the assertion.
\end{proof}

\begin{remark}
    Every immersion $\iota \colon V \imm \R^5$ with trivial normal bundle is extended to an immersion $\phi \colon V \times [0, \epsilon] \imm \R^5$.
    The choice of an extension $\phi$ corresponds to that of a normal frame of $\iota$, and in turn, to that of class in $H^1(V; \Z)$.
    However, this choice does not appear in the formula.
\end{remark}

\begin{remark}
    Juh\'asz~\cite{Juh05} generalized Saeki--Sz\H{u}cs--Takase's invariant $i_a$ to the case where the normal bundle of an immersion $\iota \colon V \imm \R^5$ is not trivial.
    However, such an immersion cannot be extended to any immersion $\phi \colon V \times [0, \epsilon] \imm \R^5$.
\end{remark}

We have the following corollary to the above theorem.

\begin{theorem}\label{conj-formula:A2-SSzT02}
{\it
    Let $V$ be a closed oriented $3$-manifold and $\phi \colon V \times [0, \epsilon] \imm \R^5$ an immersion with trivial normal bundle. If $\alpha(A_2 | \phi) = 0$, then
    \[2 i_a(\phi|_V) = - 3 \tau(V) - h(V, \theta).\]
}
\end{theorem}

The author does not know if $\alpha(A_2 | \phi) = 0$. See also \cref{remark:conjecture-for-A2}.

\subsection{Case (viii): type \texorpdfstring{$\Sigma^2$}{Sigma2} in dimensions \texorpdfstring{$(m, n) = (4, 4)$}{(m,n)=(4,4)}}\label{subsection:Sigma2_m=n=4}

Recall that the type $\Sigma^2 \subset J^K(4, 4)$ is a coorientable $\K$-singularity type of codimension $4$ and has the Thom polynomial
\[\Tp(\Sigma^2)(p_i, 0) = -p_1\]
\cite{Ron71}.

\begin{theorem}\label{formula:Sigma2-dim4}
{\it 
    Let $V$ be a closed oriented $3$-manifold and $\phi \colon V \times [0, \epsilon] \imm \R^4$ an immersion.
    Then
    \[\alpha(\Sigma^2 | \phi) = 0\]
    in $H^3(V; \Z)$.
}
\end{theorem}

To show this, we will employ Takase's method which was used in \cite[Proposition 3.1]{Tak12} and corresponds to Case (ix) below.
Let $\SI(k, r)$ denote the set of all oriented cobordism classes of immersions of closed oriented $k$-manifolds into $\R^{k + r}$.
This is endowed with an abelian group structure by disjoint union (or connected sum).
Recall that $\SI(3, 1)$ is a finite group:
\[\SI(3, 1) \cong \pi_3^S \cong \Z_{24}\]
\cite{Rok52,Wel66}.

\begin{proof}[Proof of \cref{formula:Sigma2-dim4}]    
    First, $24 \cdot \phi|_V$ is oriented null-cobordant.
    Take its oriented null-cobordism $g \colon M' \imm \R^5_+$, where $M'$ is the oriented null-cobordism of $24 \cdot V$ and $g^{-1}(\partial \R^5_+) = \partial M'$.
    Put 
    \[f' = \pi \circ g \colon M' \imm \R^5_+ \to \partial \R^5_+ = \R^4.\]
    One can construct $g$ so that the restriction of $f'$ to a collar neighborhood of each component $V$ coincides with `$-\phi$', the extension of $\phi|_V$ into the opposite direction of $\phi$.
    On the other hand, take an extension datum $f \colon M \to \R^4$ of $\phi$, where $M$ is an oriented null-cobordism of $V$.
    We obtain a closed oriented $4$-manifold $M_1$ by gluing $24 \cdot M$ and $M'$ by an orientation-reversing diffeomorphism between the boundaries, and a generic map 
    \[F \colon M_1 \to \R^4\]
    gluing $24 \cdot f$ and $f'$. 
    By construction, the restriction of $F$ to a bicollar of $V$ is an immersion of the normal bundle of $\phi|_V$.
    Furthermore, the projection $\pi$ makes the rank of $g$ decrease at most $1$.
    Therefore, we have
    \begin{align*}
        24 \cdot \# \Sigma^2(f) 
        &= \# \Sigma^2(F) \\
        &= -p_1[M_1] \\
        &= -24 \cdot p_1[TM | \theta] + p_1[TM' | \theta].
    \end{align*}
    However, $p_1[TM' | \theta] = p_1[\epsilon \oplus TM' | \theta_g] = 0$, where $\theta_g$ is the frame of $(\epsilon \oplus TM')|_{\partial M'}$ induced by $g$.
    Thus, we obtain
    \[\# \Sigma^2(f) = -p_1[TM | \theta],\]
    which completes the proof.
\end{proof}

\begin{remark}\label{remark:conjecture-for-A2}
    Takase's method is applicable to any case where the following are satisfied:
    \begin{enumerate}
        \item the cobordism classification of immersions is isomorphic to a finite group;
        \item the singularity type is of corank at least $2$;
        \item the coefficient ring of the cohomology has no torsion. 
    \end{enumerate}
    Note that Case (vii) (\cref{formula:A2-ESz03}) does not satisfy the condition (2), while it satisfies the others. In particular, 
    \[\SI(3, 2) \cong \pi^S_5(\C P^\infty) \cong \Z_2\]
    \cite{Wel66,Pas84}.
\end{remark}

Now, we give two applications.
First, applying \cref{formula:Sigma2-dim4} to the case where $V = S^3$, we recover the following.

\begin{theorem}[{Ekholm--Takase~\cite[Proposition 2.7]{ET11}}]
{\it
    Let $\iota \colon S^3 \imm \R^4$ be an immersion, and let $\theta$ denote the framing of $\epsilon \oplus TS^3$ induced by $\iota$.
    Choose an oriented null-cobordism $M$ of $S^3$ and a generic extension $f \colon M \to \R^4$ of $\iota$ which is an immersion around $\partial M = S^3$.
    Then
    \[h(S^3, \theta) = - 3 \sigma(M) - \# \Sigma^2(f).\]
}    
\end{theorem}

\begin{proof}
    We obtained
    \[\# \Sigma^2(f) = - p_1[TM | \theta].\]
    Then the assertion follows by the definition of the signature defect.
\end{proof}

Second, we discover a relationship among the signature defect and the following invariants.
Let $V$ be a closed oriented $3$-manifold and $\iota \colon V \imm \R^4$ an immersion.
\begin{itemize}
    \item 
        We consider the {\it $\mu$-invariant}
        \[\mu(V, \theta) = \sigma(M) \bmod{16},\]
        where $M$ is a spin null-cobordism of $V$ with respect to the spin structure obtained by restricting $\theta$ to the $2$-skeleton of $V$.
        This is independent of the choice of $M$.
    \item 
        Takase constructed an oriented cobordism invariant\footnote{The sign of $\# \Sigma^2(f)$ is opposite to Takase's. Indeed, a coorientation of $\Sigma^2$ is chosen so that $\Tp(\Sigma^2)(p_i, 0) = +p_1$ in \cite[\S2.2]{Tak07a}.}\footnote{Takase's invariant is $-T(\iota)$. We adopt $T(\iota)$ just for compatibility to the other formulas.}
        \[T(\iota) = \frac{3}{2}(\sigma(M) - \mu(V, \theta)) + \frac{1}{2}\# \Sigma^2(f) \bmod 24,\]
        where $M$ is an oriented null-cobordism of $V$ and $f \colon M \to \R^4$ is a generic extension of $\iota$ which is an immersion around $\partial M = V$.
        This is independent of the choice of $M$ and $f$, and is also a complete invariant for $\SI(3, 1) \cong \Z_{24}$.
        See~\cite{Tak07a} for details.
\end{itemize}

Takase's invariant $T$ is a generalization of the Smale invariant $\Omega(j \circ \iota) \bmod{24}$ (also compare to Ekholm--Sz\H{u}cs' formula and Saeki--Sz\H{u}cs--Takase's invariant).
We obtain the following.

\begin{theorem}
{\it
    Let $V$ be a closed oriented $3$-manifold and $\iota \colon V \imm \R^4$ an immersion. Then
    \[2 T(\iota) \equiv - 3 \mu(V, \theta) - h(V, \theta) \pmod{48}.\]
}
\end{theorem}

\begin{proof}
    By the definition of $T$, for any oriented null-cobordism $M$ of $V$ and any generic extension $f \colon M \to \R^4$ of $\iota$ which is an immersion around $\partial M = V$,
    \begin{align*}
        \# \Sigma^2(f) 
        &\equiv -3(\sigma(M) - \mu(V, \theta)) + 2T(\iota) \\
        &\equiv -p_1[TM | \theta] + h(V, \theta) + 3\mu(V, \theta) + 2T(\iota)
    \end{align*}
    modulo $48$.
    Hence, 
    \[\alpha(\Sigma^2 | \phi) \equiv 2 T(\phi|_V) + 3 \mu(V, \theta) + h(V, \theta).\]
    Then the assertion follows from \cref{formula:Sigma2-dim4}.
\end{proof}

\begin{remark}
    Ekholm~\cite{Ekh01b} constructed a complete invariant
    $(\lambda, \beta) \colon \SI(3, 1) \to \Z_3 \oplus \Z_8$
    using self-intersections of immersions.
    This coincides with $T$ up to automorphisms of $\Z_{24}$.
    Hence, the above formula gives an equality also for $(\lambda, \beta)$.
\end{remark}

\subsection{Case (ix): type \texorpdfstring{$\Sigma_\FR$}{Sigma-FR} in dimensions \texorpdfstring{$(m, n) = (8, 8)$}{(m,n)=(8,8)}}\label{subsection:SigmaFR_m=n=8}

Recall that the type $\Sigma_\FR \subset J^K(8, 8)$ is a coorientable $\K$-singularity type of codimension $8$ and has the Thom polynomial
\[\Tp(\Sigma_\FR)(p_i, 0) = 12 p_1^2 - 9 p_2\]
\cite{FR02}.
The following is immediate from the result of Takase~\cite[Proposition 3.1]{Tak12}.

\begin{theorem}\label{formula:8-dim}
{\it
    Let $V$ be a closed oriented $7$-manifold and $\phi \colon V \times [0, \epsilon] \imm \R^8$ an immersion.
    Then
    \[\alpha(\Sigma_\FR | \phi) = 0\]
    in $H^7(V; \Z)$.
}
\end{theorem}

We note that this was proven using the finiteness of the group $\SI(7, 1) \cong \Z_{240}$ and the fact that $\Sigma_\FR$ is of corank $2$.

\section{Determining correction terms \texorpdfstring{\III}{III}: Further types in dimensions \texorpdfstring{$m \ge n$}{m-ge-n}}\label{section:case-general-nega}

In this section, we deal with singularity types in dimensions $m \ge n$.

\subsection{Cases (iii) and (iv): type \texorpdfstring{$A_1$}{A1} in dimensions \texorpdfstring{$m \ge n$}{m-ge-n}}\label{subsection:case-A1-nega}

Recall that the type $A_1 \subset J^1(m, n)$ is of codimension $m - n + 1$ and has the Thom polynomial
\[\Tp(A_1)(w_i, 0) = w_{m - n + 1}.\]
We saw it at the end of \cref{section:RCC}.
We obtain its relative version as follows.

\begin{proposition}\label{theorem:RCC-is-RTP}
{\it
    Let $\phi \colon M_0 \to N$ be a framable prescribed datum and $\Theta$ a global frame on $N$.
    Then 
    \[\iota^* \Tp(A_1 | \phi) = w_{m - n + 1}(\gamma^m | \theta)\]
    in $H^{m - n + 1}(\BO(m), M_0; \Z_2)$.
} 
\end{proposition}

\begin{proof}
    As we saw at the end of \cref{section:RCC}, forgetting the $\GL(n; \R)$-action on $A_1$, the subsets $A_1 \subset J^1(m, n) = \Hom(\R^m, \R^n)$ and $\Sigma \subset \Hom(\R^n, \R^m)$ form the same $\GL(m; \R)$-equivariant refined fundamental class.
    Furthermore, prescribing $j^1 \phi$ in the bundle $\B_\G J^1(m, n)$ over $M_0$ is the same as prescribing the induced frame $\theta$ in $\Hom(\R^n, \R^m)$.
\end{proof}

\begin{theorem}\label{formula:A1-nega}
{\it
    For any framable prescribed datum $\phi \colon M_0 \to N$ and any global frame $\Theta$ on $N$,
    \[\alpha(A_1 | \phi, \Theta) \equiv 0\]
    in $H^{m - n}(M_0; \Z_2)$ modulo $\Ker (\connhom \colon H^{m - n}(M_0; \Z_2) \to H^{m - n + 1}(\BO(m), M_0; \Z_2))$.
}
\end{theorem}

\begin{proof}
    Combine \cref{theorem:RCC-is-RTP} and \cref{theorem:exp-subm}.
\end{proof}

We note that Saeki applied this type of observation to a generic extendability problem of a prescribed map in the case $(m, n) = (3, 2)$.

\begin{theorem}[{\cite[Theorem 2.3]{Sae20}}]
{\it
    Let $M$ be a closed oriented $3$-manifold, $L \subset M$ an oriented null-cobordant link with normal frame $\theta_L$, and $J \subset M$ an unoriented link disjoint from $L$.
    Then the following conditions are equivalent: 
    \begin{enumerate}
        \item there is a generic map $f \colon M \to \R^2$ and a regular value $y \in \R^2$ such that $f^{-1}(y)$ coincides with $L$ as an oriented framed link and that $\ol{A_1(f)} = J$;
        \item $[J] = w_2(TM | \theta_L) \in H^2(M, L; \Z_2)$.
    \end{enumerate} 
}
\end{theorem}

The implication $(1) \Rightarrow (2)$ was proven in \cite[Lemma 2.4]{Sae20} and is a special case of \cref{formula:A1-nega}.
In this case, the corresponding framed prescribed datum is $(\pr_2 \colon L \times D^2 \to \R^2, \Theta)$, where $\Theta$ is a homotopically unique frame on $\R^2$. It is clear that the frame $\theta_L$ is induced by $(\pr_2, \Theta)$.

\vspace{10pt}

In the case where $m \ge n = 1$ and manifolds are oriented, we also have an integral coefficient version.
Here, we consider the action of orientation-preserving diffeomorphisms 
$\mathcal{R}^+ = \Diff^+(\R^m, 0)$
on the space of function-germs $f \colon (\R^m, 0) \to (\R, 0)$ given by
$\sigma.f = f \circ \sigma^{-1}$ ($\sigma \in \mathcal{R}^+$).
Notice that this action descends to $J^K(m, 1)$ and that
\[\mathcal{R}^+ \simeq \SO(m).\]
One can show that $A_1 \subset J^K(m, 1)$, the type consisting of Morse singularities, forms a coorientable $\mathcal{R}^+$-singularity type of codimension $m$ (cf.~\cite{Vas81}, \cite[Chapter 4, \S2.4]{AGLV93}).
Then the Poincar\'e--Hopf theorem yields that
\[\Tp(A_1)(p_i, e) = e\]
in $H^m(\BSO(m); \Z)$ modulo $2$-torsion ($e = 0$ if $m$ is odd).
Also in this situation, \cref{theorem:exp-subm} is available in appropriate form, and the following holds.

\begin{theorem}\label{formula:A1-fxn}
{\it
    Let $V$ be a closed oriented $(m - 1)$-manifold which is oriented null-cobordant. Let $\Theta$ denote the standard frame on $\R$.
    Then, for any framable prescribed datum $\phi \colon V \times [0, \epsilon] \to \R$ such that $\phi^{-1}(0) = V$ and $\phi(V \times [0, \epsilon]) \subset (-\infty, 0]$, 
    \[\alpha(A_1 | \phi, \Theta) \equiv 0\]
    in $H^{m - 1}(V; \Z)$ modulo $\Ker (\connhom \colon H^{m - 1}(V; \Z) \to H^m(\BSO(m), V; \Z))$.
}
\end{theorem}

\begin{proof}
    Choose an oriented null-cobordism $M$ of $V$.
    We also extend $\phi$ to a Morse function $f \colon M \to \R$ so that $f^{-1}(0) = V$.
    By the relative version of the Poincar\'e--Hopf theorem, we have
    \[\# A_1(f) = \chi(M), \quad \textrm{i.e.,} \quad [A_1(f)] = e(TM | \theta),\]
    where $\theta$ is an outward normal vector field defined around $\partial M$.
    The frame $\theta$ coincides with the gradient vector field of $\phi$, and in turn, with the frame induced by $\Theta$ and $\phi$.
    This completes the proof.
\end{proof}

\subsection{Case (v): type \texorpdfstring{$A_2$}{A2} in dimensions \texorpdfstring{$(m, n) = (2k, 2)$}{(m,n)=(2k,2)}}\label{subsection:A2_m>=n=2}

Recall that for $m \ge 2$, the type $A_2 \subset J^K(m, 2)$ is of codimension $m$ and has the Thom polynomial 
\[\Tp(A_2)(w_i, 0) = w_m\]
\cite{Tho55}.
We consider the case where $m = 2k$, and will use the result of H.~Levine~\cite{Lev95}. 
For a generic map $f \colon M \to \R^2$ on an oriented manifold $M$ and an oriented simple closed curve $C \subset M$, one can define the {\it projectivized rotation number} $\prot(f, C)$ to be the mapping degree of
\[C \to \R P^1; \quad x \mapsto \Im d(f|_C)_x \subset T_{f(x)} \R^2 \cong \R^2.\]

\begin{theorem}[\cite{Lev95}; see also~\cite{Lev66}]\label{theorem:Levine}
{\it
    Let $M$ be a compact oriented $2k$-manifold with boundary and $f \colon M \to \R^2$ a generic map which is a submersion around $\partial M$.
    Let $\{S_i\}$ denote the set of components of $\Sigma(f)$.
    Let $\{B_j\}$ denote the set of components of $\partial M$ if $k = 1$, and of $\Sigma(f|_{\partial M})$ if $k > 1$.
    Then there is a unique tuple of orientations of $S_i$ and $B_j$ such that 
    \[\chi(M) = \sum_i \prot(f, S_i) + \frac{1}{2} \sum_j \prot(f, B_j).\]
}
\end{theorem}

\begin{theorem}\label{formula:A2-subm}
{\it
    Let $k \ge 1$ be an integer, $V$ an oriented null-cobordant $(2k-1)$-manifold, and $\phi \colon V \times [0, \epsilon] \to \R^2$ a submersion.
    Then, for a homotopically unique frame $\Theta$ on $\R^2$, 
    \[\alpha(A_2 | \phi, \Theta) \equiv 0\]
    in $H^{2k-1}(V; \Z_2)$ modulo $\Ker (\connhom \colon H^{2k-1}(V; \Z_2) \to H^{2k}(\BO(2k) \times \BO(2), V; \Z_2))$.
}
\end{theorem}

\begin{proof}
    Take an extension datum $f \colon M \to \R^2$ of $\phi$, where $M$ is an oriented null-cobordism of $V$.
    First, by \cref{theorem:Levine}, we have
    \[\chi(M) = \sum_i \prot(f, S_i) + \frac{1}{2} \sum_j \prot(f, B_j).\]
    Second, let $\theta_0$ be an outward vector field of $V = V \times \{0\} \subset V \times [0, \epsilon]$
    and $\theta_1$ a vector field on $V$ tangent to $M$ such that $d\phi(\theta_1)$ is nowhere-vanishing.
    Then, by the relative version of the Poincar\'e--Hopf theorem, we have 
    \begin{align*}
        \chi(M) 
        &= e[TM | \theta_0] \\
        &= e[TM | \theta_1] + \langle \connhom d_e(\theta_0, \theta_1), [M, V] \rangle \\
        &= e[TM | \theta_1] + \langle d_e(\theta_0, \theta_1), [V] \rangle.
    \end{align*}
    Third, by geometric observation, the term $\prot(f, S_i)$ has the same parity as the number of cusps of the plane curve $f|_{S_i}$. Hence, we have
    \[\sum_i \prot(f, S_i) \equiv \# A_2(f) \pmod 2.\]
    Fourth, it is clear that for each $j$,
    \[\prot(f, B_j) = \prot(\phi, B_j).\]
    Combining these four equalities, we have
    \[\# A_2(f) \equiv e[TM | \theta_1] + \langle d_e(\theta_0, \theta_1), [V] \rangle + \frac{1}{2} \sum_j \prot(\phi, B_j) \pmod 2.\]
    Furthermore, comparing it to \cref{theorem:exp-subm} and the equality $e[TM | \theta_1] \equiv w_{2k}[TM | \theta_1]$ yields that it suffices to show that
    \[\langle d_e(\theta_0, \theta_1), [V] \rangle \equiv \frac{1}{2} \sum_j \prot(\phi, B_j) \pmod 2.\]

    When $k = 1$, since the restriction $\phi|_{B_j}$ is an immersion, 
    the usual rotation number $\rot(\phi, B_j)$ is properly defined.
    Furthermore, by geometric observation, 
    \begin{align*}
        \frac{1}{2} \sum_j \prot(\phi, B_j) 
        &= \sum_j \rot(\phi, B_j) \\
        &= \sum_j \langle d_e(\theta_0, \theta_1), [B_j] \rangle \\
        &= \pm \langle d_e(\theta_0, \theta_1), [V] \rangle.
    \end{align*}
    The last sign ambiguity is due to the possibility that the orientations of $B_j$ chosen in \cref{theorem:Levine} differ from the orientation of $V$.
    This completes the proof for the case where $k = 1$. 

    Assume that $k > 1$. We use the following invariant.

    \begin{definition}
        Let $E$ be an oriented vector bundle over $V$ of rank $2k-1$.
        For a generic nowhere-vanishing section $\sigma = (a, \sigma') \colon V \to \epsilon^1 \oplus E$, we consider the integer
        \[I(\sigma) = \sum_{\sigma'(x) = 0} \sgn(a(x)) \cdot \ind(\sigma', x),\]
        where $\ind(\sigma', x)$ is the local degree of $\sigma'$ around $x$ with respect to the orientation of $V$ and $E$.
    \end{definition}

    It is elementary that the value $I(\sigma)$ depends only on the homotopy class of $\sigma$ through nowhere-vanishing sections.
    We will consider the bundle $E = \epsilon^1 \oplus (\Ker d\phi)|_V$, more precisely, 
    \begin{align*}
        T(V \times [0, \epsilon])|_V 
        &= \R \theta_1 \oplus \left(\R \theta_2 \oplus (\Ker d\phi)|_V \right) \\
        &= \epsilon^1 \oplus E,
    \end{align*}
    and the following two sections.
    One is the outward vector field $\theta_0$.
    On the other hand, the projection
    \[\nu \colon \epsilon^1 \oplus E = T(V \times [0, \epsilon])|_V \to \R \theta_0\]
    forms a section of $(\epsilon^1 \oplus E)^* = \epsilon^1 \oplus E^*$. 
    We choose a Riemannian metric $g$ which makes the decomposition orthogonal and is standard on $\R \theta_i$ for each $i = 1, 2$.
    Then the other section which we consider is the dual of $\nu$ under $g$:
    \[\nu^\sharp = (\nu(\theta_1), (\nu|_E)^\sharp) \colon V \to \R \theta_0 \oplus E = \epsilon^1 \oplus E.\]
    It is obvious that the two sections $\theta_0$ and $\nu^\sharp$ are homotopic through nowhere-vanishing sections of $\epsilon^1 \oplus E$.
    Therefore, we have 
    \[I(\theta_0) = I(\nu^\sharp)\]
    and compute both sides.

\begin{claim}
{\it
    $I(\theta_0) = \pm 2 \langle d_e(\theta_0, \theta_1), [V] \rangle$.
}
\end{claim}

\begin{proof}
    Write $\theta_0 = (a, v)$, where $a$, $v$ are the $\theta_1$- and $E$-component of $\theta_0$, respectively.
    Since $E$ admits the nowhere-vanishing section $\theta_2$,
    \[\sum_{v(x) = 0} \ind(v, x) = \langle e(E), [V]\rangle = 0.\]
    Hence,
    \begin{align*}
        I(\theta_0)
        &= \sum_{\substack{v(x) = 0 \\ a(x) > 0}} \ind(v, x)
        - \sum_{\substack{v(x) = 0 \\ a(x) < 0}} \ind(v, x) \\
        &= -2 \sum_{\substack{v(x) = 0 \\ a(x) < 0}} \ind(v, x).
    \end{align*} 
    Now, notice that if $v(x) = 0$, then $\theta_0(x) = a(x) \theta_1(x)$.
    Hence, for $x \in V$,
    \[v(x) = 0 \text{ and } a(x) < 0 \iff \theta_0(x) \text{ is a negative multiple of } \theta_1(x).\]
    Since $d_e(\theta_0, \theta_1)$ is the precise obstruction to homotoping $\theta_1$ to $\theta_0$, it counts, with signs, all such points $x \in V$. Namely,
    \[\sum_{\substack{v(x) = 0 \\ a(x) < 0}} \ind(v, x) = \pm \langle d_e(\theta_0, \theta_1), [V]\rangle.\]
    This shows the claim.
\end{proof}

\begin{claim}
{\it
    $\displaystyle I(\nu^\sharp) = \sum_j \epsilon_j \cdot \prot(\phi, B_j)$ for some $\epsilon_j \in \{\pm 1\}$.
}
\end{claim}

\begin{proof}
    Recall that 
    \[I(\nu^\sharp) = \sum_{(\nu|_E)^\sharp(x) = 0} \sgn(\nu(\theta_1(x))) \cdot \ind((\nu|_E)^\sharp, x).\]
    Since the metric $g$ makes $E = \R \theta_2 \oplus (\Ker d\phi)|_V$ orthogonal, for $x \in V$, 
    \begin{align*}
        (\nu|_E)^\sharp(x) = 0 
        &\iff \nu(\theta_2(x)) = 0 \text{ and } \nu|_{\Ker d\phi_x} = 0 \\
        &\iff \nu(\theta_2(x)) = 0 \text{ and } \Ker d\phi_x \subset T_x V \\
        &\iff \nu(\theta_2(x)) = 0 \text{ and } x \in \Sigma(\phi|_V) \\
        &\iff \nu(\theta_2(x)) = 0 \text{ and } x \in B_j \text{ for some } j
    \end{align*}
    Hence, 
    \[I(\nu^\sharp) = \sum_j \sum_{\substack{x \in B_j \\ \nu(\theta_2(x)) = 0}} \sgn(\nu(\theta_1(x))) \cdot \ind((\nu|_E)^\sharp, x).\]
    On the other hand, the integer $\prot(\phi, B_j)$ was the mapping degree of the map
    \[\gamma_j \colon B_j \to \R P^1; \quad x \mapsto \Im d(\phi|_V)_x\]
    (the curve $B_j$ was oriented when we apply \cref{theorem:Levine}).
    Then it suffices to show that for any $j$, there is a sign $\epsilon_j \in \{\pm 1\}$ such that
    \[\sum_{\substack{x \in B_j \\ \gamma_j(x) = \R e_2}} \sgn d(\gamma_j)_x = \epsilon_j \sum_{\substack{x \in B_j \\ \nu(\theta_2(x)) = 0}} \sgn(\nu(\theta_1(x))) \cdot \ind((\nu|_E)^\sharp, x).\]

    Fix an index $j$. In fact, for $x \in B_j$, 
    \[\gamma_j(x) = \R e_2 \iff \nu(\theta_2(x)) = 0.\] 
    The direction `$\Leftarrow$' is obvious since $x \in \Sigma(\phi|_V)$ and $\theta_2(x) \in T_x V$.
    We show `$\Rightarrow$'.
    Suppose that $\Im d(\phi|_V)_x = \R e_2$.
    Then there is a vector $v \in T_x V$ such that
    \[d\phi_x(v) = e_2 = d\phi_x(\theta_2(x)).\]
    Hence,
    \[v - \theta_2(x) \in \Ker d\phi_x.\]
    Since $\Ker d\phi_x \subset T_x V$ and $v \in T_xV$, we have $\theta_2(x) \in T_x V$, i.e., $\nu(\theta_2(x)) = 0$.

    Finally, we compare the signs.
    Let $x \in B_j$ be a point such that $\gamma_j(x) = \R e_2$.
    Since $\nu(\theta_2(x)) = 0$, the map $\gamma_j$ is represented by the function $\nu(\theta_2) / \nu(\theta_1)$ around $x$.
    Hence, 
    \begin{align*}
        \sgn d(\gamma_j)_x 
        &= \sgn \left( \frac{1}{\nu(\theta_1(x))}d(\nu(\theta_2)|_{B_j})_x \right) \\
        &= \sgn \nu(\theta_1(x)) \cdot \sgn d(\nu(\theta_2)|_{B_j})_x.
    \end{align*}
    On the other hand, 
    \[\ind((\nu|_E)^\sharp, x) = \sgn d(\nu(\theta_2)|_{B_j})_x \cdot \sgn \det d\left(\nu'|_{N_x}\right)_x,\]
    where $\nu' = (\nu|_{(\Ker d\phi)|_V})^\sharp$, $N$ is the normal bundle to $B_j$ (which is identified with a tubular neighborhood), and $N_x$ is the fiber at $x$.
    Notice that an orientation of $N$ is chosen to be compatible with those of $B_j$ and $V$, and that of $(\Ker d\phi)|_V$ is also chosen to be compatible with those of $\R \theta_2$ and $E = \R \theta_2 \oplus (\Ker d\phi)|_V$.
    The correspondence $x \mapsto \sgn \det d\left(\nu'|_{N_x}\right)_x$ is constant on $B_j$. 
    Let $\epsilon_j$ denote the constant sign. Then
    \[\sgn d(\gamma_j)_x = \epsilon_j \cdot \sgn \nu(\theta_1(x)) \cdot \ind((\nu|_E)^\sharp, x).\]
    This shows the claim.
\end{proof}

Thus, \cref{formula:A2-subm} is proven also for the case where $k > 1$.
\end{proof}

\section{Suggestive evidence of relative Thom polynomials for multi-singularities}\label{section:multi-case}

In this section, we give the precise statement and proof of the relative version of Herbert's multiple-point formula.
We also apply it to Ekholm's linking invariant.
This section can be read independently of other sections.

\subsection{Herbert's formula}\label{section:abs-multi}

Let $M$ be a compact $m$-manifold, $N$ an $n$-manifold, and $f \colon M \imm N$ a generic (i.e., self-transverse) immersion.

For an integer $k \ge 1$, we define the {\it multiple-point loci} of $f$ by
\begin{alignat*}{2}
    N_k &= N_k(f) = \{ y \in N \mid \# f^{-1}(y) = k \} &{}& \subset N, \\
    M_k &= M_k(f) = f^{-1}(N_k) &{}& \subset M.
\end{alignat*}
Then their closures
$\overline{N_k} \subset N$, $\overline{M_k} \subset M$
are locally unions of some normally crossing linear subspaces.
Hence, they admit the Whitney stratifications
\[\overline{N_k} = \bigcup_{i = k}^\infty N_i, \quad \overline{M_k} = \bigcup_{i = k}^\infty M_i\]
and form Whitney stratified cycles (both the unions stop at finite multiplicity).
Note that their normal bundles consist of copies of the fiber of the normal bundle $\nu_f$ of $f$.

Then we have two types of cohomology classes as follows. Here, the coefficient ring is chosen to be $\Z$ if $\nu_f$ is oriented and $r$ is even, or $\Z_2$ otherwise.
Put $r = n - m$.
\begin{itemize}
    \item 
        The {\it multiple-point loci classes}:
        \[n_k =  \left[ \ol{N_k} \right] \in H^{kr}(N), \quad m_k =  \left[ \ol{M_k} \right] \in H^{(k-1)r}(M).\]
    \item   
        The {\it normal Euler class} of $f$:
        \[e(\nu_f) = \left[ s^{-1}(M) \right] \in H^r(M),\]
        where $s \colon M \to \nu_f$ is a section which is transverse to the zero section. This is independent of the choice of $s$.
\end{itemize}

In terms of these classes, Herbert's formula is stated as follows.

\begin{theorem}[\cite{Her75, Ron80}]\label{theorem:HR-abs}
{\it
    Let $k \ge 1$ be an integer.
    Then 
    \[m_{k + 1} = f^* n_k - e(\nu_f) \smile m_k\]
    in $H^{kr}(M)$.
}
\end{theorem}

The following is immediate by solving the above recursive equation.

\begin{corollary}\label{theorem:HR-abs-cor}
{\it
    Let $k \ge 1$ be an integer.
    Then 
    \[m_{k + 1} = f^* n_k + (-1)^k e(\nu_f)^k\]
    in $H^{kr}(M)$.
}
\end{corollary}

\subsection{Relative version of Herbert's formula}\label{section:rel-multi}

Again, let $M$ be a compact $m$-manifold, $N$ an $n$-manifold, and $f \colon M \imm N$ a generic (i.e., self-transverse) immersion. Let $r = n - m$ denote the codimension of $f$.
Under additional assumptions on the boundary, the basic cohomology classes in the previous subsection can be upgraded into relative versions as follows.

First, we assume that $f$ satisfies that $f^{-1}(\partial N) = \partial M$, and put
\[f_\partial = f|_{\partial M}^{\partial N} \colon \partial M \imm \partial N.\]
We also assume that $f$ is transverse to $\partial N$, i.e., the restriction of $f$ to a collar neighborhood of $\partial M$ is of the form $f_\partial \times \id_{[0, \epsilon]}$.
Then we obtain the following.
\begin{itemize}
    \item 
        If $M_{k + 1}(f_\partial) = \varnothing$, then the multiple-point loci of $f$ are disjoint from the boundary.
        Therefore, we obtain relative classes:
        \[n_k^\rel =  \left[ \ol{N_k} \right] \in H^{kr}(N, \partial N), \quad m_k^\rel =  \left[ \ol{M_k} \right] \in H^{(k-1)r}(M, \partial M).\]
    \item   
        If $f_\partial$ admits a nowhere-vanishing normal vector field $\theta$, then the normal Euler class of $f$ {\it relative to} $\theta$ is defined. Here, we express it by
        \[e(\nu_f | \theta) = \left[ s^{-1}(M) \right] \in H^r(M, \partial M),\]
        where $s \colon M \to \nu_f$ is a section which is transverse to the zero section and $s|_{\partial M} = \theta$. This is independent of the choice of $s$.
\end{itemize}

Then a relative version of Theorem~\ref{theorem:HR-abs} is stated as follows.

\begin{theorem}\label{theorem:HR-rel}
{\it
    Let $k \ge 2$ be an integer.
    Assume that $M_k(f_\partial) = \varnothing$ and that $f_\partial$ admits a nowhere-vanishing normal vector field $\theta$.
    Then
    \[m_{k + 1}^\rel = f^* n_k^\rel - e(\nu_f | \theta) \smile m_k^\rel\]
    in $H^{kr}(M, \partial M)$.
}
\end{theorem}

We also have the following modified version.

\begin{theorem}\label{theorem:HR-rel-modified}
{\it
    Let $k \ge 1$ be an integer.
    Assume that $M_{k + 1}(f_\partial) = \varnothing$ and that $f_\partial$ admits a nowhere-vanishing normal vector field $\theta$.
    Then
    \[m_{k + 1}^\rel = \left[ g^{-1}(\overline{N_k}) \right] - e(\nu_f | \theta) \smile m_k\]
    in $H^{kr}(M, \partial M)$, 
    where $g \colon (M, \partial M) \imm (N, \partial N)$ is a perturbation of $f$ which is obtained by pushing $f$ into the $\theta$-direction around $\partial M$.
}
\end{theorem}

Note that although $m_k$ is an absolute cohomology class, the product $e(\nu_f | \theta) \smile m_k$ is a relative one.
We also note that Theorems~\ref{theorem:HR-rel} and~\ref{theorem:HR-rel-modified} hold even in the case relative to a submanifold $S \subset M$ which does not meet $\partial M$.

Before the proofs of these theorems, we see an immediate consequence.
The following is the relative version of Corollary~\ref{theorem:HR-abs-cor}.

\begin{corollary}\label{corollary:HR-rel-cor}
{\it
    Under the same setup in Theorem~\ref{theorem:HR-rel-modified}, 
    \[m_{k + 1}^\rel = \left[ g^{-1}(\overline{N_k}) \right] - e(\nu_f | \theta) \smile f^* n_k + (-1)^k e(\nu_f | \theta)^k\]
    in $H^{kr}(M, \partial M)$.
}
\end{corollary}

\begin{proof}
    By Theorem~\ref{theorem:HR-rel-modified}, we have
    \[m_{k + 1}^\rel = \left[ g^{-1}(\overline{N_k}) \right] - e(\nu_f | \theta) \smile m_k.\]
    On the other hand, by Theorem~\ref{theorem:HR-abs}, we have
    \[m_k = f^* n_k + (-1)^{k-1} e(\nu_f)^{k-1}.\]
    Therefore, 
    \begin{align*}
        e(\nu_f | \theta) \smile m_k
        &= e(\nu_f | \theta) \smile f^* n_k + e(\nu_f | \theta) \smile (-1)^{k-1} e(\nu_f)^{k-1} \\
        &= e(\nu_f | \theta) \smile f^* n_k + e(\nu_f | \theta) \smile (-1)^{k-1} e(\nu_f | \theta)^{k-1} \\
        &= e(\nu_f | \theta) \smile f^* n_k + (-1)^{k-1} e(\nu_f | \theta)^k.
    \end{align*}
    This completes the proof.
\end{proof}

We give the proof of Theorems~\ref{theorem:HR-rel} and~\ref{theorem:HR-rel-modified}.
Again, note that these are due to Herbert's original proof and Ekholm--Sz\H{u}cs' argument~\cite{Her75, ESz03}.

\subsubsection{Proof of Theorem~\ref{theorem:HR-rel}}

    Fix a Riemannian metric $d_2$ on $N$ and induce $d_1$ on $M$ by the immersion $f$.
    Let $B_1(x; r)$ denote the open ball in $M$ centered at $x$ with radius $r > 0$ with respect to $d_1$. The notation is similar for $d_2$.

    Since $M$ is compact and $f$ is an immersion, 
    there is a positive number $\eta > 0$ such that for any $x \in M$, $f$ is an embedding on the open ball $B_1(x; 3 \eta)$.
    Furthermore, there is another positive number $\delta > 0$ such that for any $x, x' \in M$, if $d_2(f(x), f(x')) < \delta$, then either $d_1(x, x') < \eta$ or $d_1(x, x') > 2 \eta$.
    Hereafter, we fix such numbers.

    We choose the normal bundle $\nu_f$ to $f \colon M \imm N$ so that it is compatible with the immersion $f_{\partial} = f|_{\partial M}^{\partial N} \colon \partial M \imm \partial N$.
    Let $s_0$ denote the zero-section of $\nu_f$ and $\bar{f} \colon \nu_f \to TN$ the natural bundle homomorphism.
    We take a tubular neighborhood $T$ of $s_0(M)$ in $E(\nu_f)$ so that the map $\exp = \exp_N \circ \bar{f}|_T \colon T \to N$ is an immersion which is an embedding on $T \cap E(\nu_f|_{B_1(p; \eta)})$ for each $p \in M$.
    We also choose a section $s \colon M \to T$ so that
    \begin{enumerate}
        \item $s$ is transverse to $s_0(M)$;
        \item $s|_{\partial M}$ coincides with the given normal vector field $\theta$ up to scaling;
        \item the map $g = \exp \circ s \colon M \to N$ is an immersion and transverse to $N_i$ for any $1 \le i \le k$, and satisfies that $d_2(f(x), g(x)) < \delta$ for any $x \in M$.
    \end{enumerate}

    In the following, we will prove the formula showing
    \[g^* n_k^\rel = e(\nu_f | \theta) \smile m_k^\rel + m_{k + 1}^\rel.\]
    First, set 
    \[M_{k + 1}^g = g^{-1}(N_k) = \{ x \in M \mid g(x) = f(x_1) \text{ for some } x_1 \in M_k \}.\]
    Since $\overline{N_k}$ does not meet $\partial N$ and $g$ is transverse to any stratum of $\overline{N_k}$, the closure $\overline{M_{k + 1}^g} = g^{-1}(\overline{N_k})$ is also a stratified cycle and carries the class $g^* n_k^\rel \in H^*(M, \partial M)$ as its dual.
    We also parameterize $\overline{M_{k + 1}^g}$ by the manifold
    \[\tilde{M}_{k + 1}^g = \{ (x, [x_1, \dots, x_k]) \in M \times \tilde{N}_k \mid g(x) = f(x_1) \}\]
    and the projection $\pi \colon M \times \tilde{N}_k \to M$.
    It is clear that $\pi(\tilde{M}_{k + 1}^g) = \overline{M_{k + 1}^g}$.

    Now, we consider subsets
    \begin{align*}
        O_\near &= \{(x, [x_1, \dots, x_k]) \in M \times \tilde{N}_k \mid d_1(x, x_i) < \eta \text{ for some } i\}, \\
        O_\far &= \{(x, [x_1, \dots, x_k]) \in M \times \tilde{N}_k \mid d_1(x, x_i) > 2 \eta \text{ for any } i\}
    \end{align*}
    of $M \times \tilde{M}_k$. These are open and disjoint.
    Furthermore, since $g$ was a $\delta$-approximation to $f$ and by the choice of $\delta$, we have
    \[\tilde{M}_{k + 1}^g \subset O_\near \cup O_\far.\]
    Putting
    \[X_\near = \pi(\tilde{M}_{k + 1}^g \cap O_\near), \quad X_\far = \pi(\tilde{M}_{k + 1}^g \cap O_\far),\]
    which are clopen subsets of $\overline{M_{k + 1}^g}$, we have the decomposition 
    \[\overline{M_{k + 1}^g} = X_\near \cup X_\far.\]
    Then we will show that the duals of $X_\near$ and $X_\far$ are $e(\nu_f) \smile m_k^\rel$ and $m_{k + 1}^\rel$, respectively.

    \begin{claim}
    {\it
        One has $X_\near = s^{-1}(s_0(M)) \cap \overline{M_k}$.
        Moreover, each intersection $s^{-1}(s_0(M)) \cap M_i$ ($i \ge k$) is transverse.
    }
    \end{claim}

    \begin{proof}
        Let $x \in X_\near$. There is $[x_1, \dots, x_k] \in \tilde{M}_k$ such that $(x, [x_1, \dots, x_k]) \in \tilde{M}_{k + 1}^g \cap O_\near$.
        Without loss of generality, assume that $d_1(x, x_1) < \eta$.
        Since $g(x) = f(x_1)$ and $\exp$ was an embedding $T \cap E(\nu_f|_{B_1(p; \eta)})$, we have that $s(x) = s_0(x_1)$ and hence $x = x_1$.
        This means that
        \[X_\near = \{ x \in M \mid g(x) = f(x), x \in \overline{M_k} \},\]
        which completes the proof of the former assertion.

        We show the latter assertion.
        Fix an integer $i \ge k$.
        Let $x \in s^{-1}(s_0(M)) \cap M_i$ and put $y = f(x) \in N_i$.
        Then we have distinct points $x, x_2, \dots, x_i \in f^{-1}(y)$.
        Since $g$ is transverse to $N_i = f(M_i)$, we have
        \[dg_x(T_x M) + df_x(T_x M_i) = T_y N.\]
        Restricting this equality to $df_x(T_x M)$,
        \[df_x(T_x s^{-1}(s_0(M))) + df_x(T_x M_i) = df_x(T_x M).\]
        However, since $df_x$ was injective, 
        \[T_x s^{-1}(s_0(M)) + T_x M_i = T_x M.\]
        This means that the intersection $s^{-1}(s_0(M)) \cap M_i$ is transverse.
    \end{proof}
    
    By the condition (2) for $s$, the set $s^{-1}(s_0(M))$ carries the class $\bar{e}$ as its dual.
    Furthermore, $\overline{M_k}$ carries the class $m_k^\rel$ as its dual.
    Therefore, the set $X_\near$ carries the class $e(\nu_f) \smile m_k^\rel \in H^{kr}(M, \partial M)$ as its dual for the $\Z_2$-coefficient case.
    Let us consider the integer coefficient case, i.e., $\nu_f$ is cooriented and $r$ is even.

    \begin{claim}
    {\it 
        The set $X_\near$ carries the class $e(\nu_f) \smile m_k^\rel \in H^{kr}(M, \partial M)$ as its dual without sign ambiguity.
    }
    \end{claim}

    \begin{proof}
        We show the coorientation of $X_\near \subset \overline{M_{k + 1}^g} = g^{-1}(\overline{N_k})$ in $M$ as transverse pull-back coincides with that of $s^{-1}(s_0(M)) \cap \overline{M_k}$ in $M$ as transverse intersection.

        For a cooriented subspace $V$ of a vector space $W$, let $\Co(V \subset W)$ denote its coorientation. Then for $x \in M_{k + 1}^g \cap X_\near$,
        \begin{alignat*}{7}
            & \Co( T_x M_{k + 1}^g &{}& \subset T_x M) \\
            = & \Co( T_{g(x)} N_k &{}& \subset T_{g(x)} N) \\
            = & \Co( dg_x (T_x M) \cap T_{g(x)} N_k &{}& \subset dg_x (T_x M)) \\
            = & \Co( df_x (T_x M) \cap dg_x (T_x M) \cap T_{g(x)} N_k &{}& \subset dg_x (T_x M)) \\
            = & \Co( df_x (T_x M) \cap dg_x (T_x M) \cap T_{g(x)} N_k &{}& \subset df_x (T_x M)) \\
            = & \Co( df_x (T_x s^{-1}(s_0(M))) \cap df_x (T_x M_k) &{}& \subset df_x (T_x M)) \\
            = & \Co( T_x s^{-1}(s_0(M)) \cap T_x M_k &{}& \subset T_x M).
        \end{alignat*}
    \end{proof}

    We also observe $X_\far$.

    \begin{claim}
    {\it
        The set $X_\far$ is homologous to $\overline{M_{k + 1}}$ as a stratified cycle of $\Int M$.
    }
    \end{claim}

    \begin{proof}
        Let $x \in X_\far$. Then there is $[x_1, \dots, x_k] \in \tilde{M}_k$ such that $(x, [x_1, \dots, x_k]) \in \tilde{M}_{k + 1}^g \cap O_\far$. 
        Since $d_1(x, x_i) > 2 \eta$, the points $x, x_1, \dots, x_k$ are distinct.
        This means that
        \[X_\far = \left\{ x \in M \ \middle| \ 
        \begin{array}{l}
            \text{there are } x_1, \dots, x_k \in M \text{ such that } \\
            g(x) = f(x_1) = \dots = f(x_k) \\
            \text{and } x, x_1, \dots, x_k \text{ are distinct}
        \end{array} \right\}.\]
        We define the homotopy
        \[H \colon M \times I \to N; \quad H(x, t) = h_t(x) = \exp (\rho(t) \cdot s(x)),\]
        between $h_0 = f$ and $h_1 = g$, where $I = [0, 1]$ and $\rho \colon I \to I$ is a function such that 
        \[\rho([0, 1/3]) = {0}, \quad \rho([2/3, 1]) = {1}.\]
        Then consider the set
        \[C = \bigcup_{t \in I} \pi(\tilde{M}_{k + 1}^{h_t} \cap O_\far) \times \{t\} \subset M \times I,\]
        which is a clopen subset of $H^{-1}(\overline{N_k})$.
        Now, we deform the stratification 
        $f(M) = \bigcup_{i = 1}^\infty N_i$
        so that $H$ is transverse to it.
        Then $C$ forms a cobordism between $X_\far$ and $\overline{M_{k + 1}}$.
    \end{proof}

    Since $\overline{M_{k + 1}}$ does not meet $\partial M$ by assumption, the cycle $X_\far$ carries the class $m_{k + 1}^\rel \in H^{kr}(M, \partial M)$ as its dual.
    This completes the proof of Theorem~\ref{theorem:HR-rel}.

\subsubsection{Proof of Theorem~\ref{theorem:HR-rel-modified}}

    We just mimic the above proof.
    Notice that the section $s$ (the perturbation $g$ of $f$) can be chosen sufficiently small so that the set $M_{k + 1}^g$ does not meet $\partial M$ by dimensional reason.
    This means that the absolute class $f^* n_k \in H^{kr}(M)$ admits the lift $\left[ g^{-1}(\ol{N_k}) \right] \in H^{kr}(M, \partial M)$.

\subsection{Application to Ekholm's linking invariant}\label{section:linking}

We briefly recall Ekholm's work. See \cite[\S4]{Ekh01a} for details.
Let $k \ge 1, r \ge 2$ be integers, $V$ a closed $(kr - 1)$-manifold, and $\iota \colon V \imm \R^{(k + 1)r - 1}$ a generic immersion.
Let $R$ denote the ring $\Z$ if $V$ is oriented and $r$ is even, or $\Z_2$ otherwise. 
First, the following holds by dimensional reason.

\begin{lemma}
{\it
    The $k$-fold loci $M_k(\iota) \subset V$ and $N_k(\iota) \subset \R^{(k + 1)r - 1}$ form closed submanifolds.
    Furthermore, the restriction of $\iota$ to $M_k(\iota)$ admits a nowhere-vanishing normal vector field.
}
\end{lemma}

\begin{definition}[{\cite[\S4.5]{Ekh01a}}]\label{definition:linking-theta}
    Let $\theta$ be a nowhere-vanishing normal vector field of the restriction of $\iota$ to $M_k(\iota)$.
    Define a normal vector field $\Theta$ on $N_k(\iota) \subset \R^{(k + 1)r - 1}$ by
    \[\Theta(y) = \theta(x_1) + \cdots + \theta(x_k) \quad (y \in N_k(\iota), \iota^{-1}(y) = \{x_1, \cdots, x_k\}),\]
    where the normal vectors are regarded as vectors in $\R^{(k + 1)r - 1}$ based at $y$.
    Push $N_k(\iota)$ into the $\Theta$-direction and let $N_k(\iota)^\theta$ denote its resulting submanifold.
    Then we define the {\it linking invariant of $\iota$ with respect to $\theta$} to be the linking number of $\iota(V)$ and $N_k(\iota)^\theta$ in $\R^{(k + 1)r - 1}$:
    \[L_k(\iota)_\theta = \lk_{\R^{(k + 1)r - 1}} (\iota(V), N_k(\iota)^\theta) \in R.\]
\end{definition}

Hereafter, we additionally assume that $H_{r - 1}(V; R) = H_r(V; R) = 0$ so we have natural isomorphisms
\[H_{r - 1}(\partial E(\nu_\iota); R) \cong H_{r - 1}(S^{r - 1}; R) \cong R.\]

\begin{definition}[{\cite[\S4.6]{Ekh01a}}]\label{definition:linking}
    Consider the class $[\theta(M_k(\iota))] \in H_{r - 1}(\partial E(\nu_\iota); R)$. Let it also denote the corresponding number in $R$ via the above isomorphisms.
    Then we define the {\it linking invariant} of $\iota$ to be
    \[L_k(\iota) = L_k(\iota)_\theta - [\theta(M_k(\iota))] \in R.\]
\end{definition}

\begin{lemma}[{\cite[Lemmata 4.15 and 4.17]{Ekh01a}}]
{\it
    The quantity $L_k(\iota)$ is independent of the choice of $\theta$ and invariant under regular homotopy through generic immersions.
}
\end{lemma}

Now, we show the following.

\begin{theorem}\label{theorem:gen-ESz03}
{\it
    Let $M$ be a compact $kr$-manifold with boundary such that
    \[H_{r - 1}(\partial M; R) = H_r(\partial M; R) = 0.\]
    Let $f \colon (M, \partial M) \imm (\R^{(k + 1)r}_+, \R^{(k + 1)r - 1})$ be a generic immersion.
    Assume that $\iota = f_\partial$ admits a nowhere-vanishing normal vector field $\theta$ on the whole of $\partial M$.
    Then
    \[L_k(\iota)_\theta = (k + 1) \# N_{k + 1}(f) - (-1)^k \left\langle e(\nu_f | \theta)^k, [M, \partial M] \right\rangle,\]
    where $\# N_{k + 1}(f)$ is the algebraic number of $N_{k + 1}(f)$.
}
\end{theorem}

\begin{proof}
    First, the $k$-fold locus of $f$ does not meet the boundary by dimensional reason.
    Applying Corollary~\ref{corollary:HR-rel-cor} to $f$ with normal frame $\theta$, we have
    \[m_{k + 1} = \left[ g^{-1}(\overline{N_k}) \right] + (-1)^k e(\nu_f | \theta)^k,\]
    where $g \colon (M, \partial M) \imm (\R^{(k + 1)r}_+, \R^{(k + 1)r - 1})$ is a perturbation of $f$ into the $\theta$-direction around $\partial M$.
    Evaluating both sides by the fundamental class $[M, \partial M]$, 
    \[(k + 1) \cdot \# N_{k + 1} = \left\langle \left[ g^{-1}(\overline{N_k})  \right], [M, \partial M] \right\rangle + (-1)^k \left\langle e(\nu_f | \theta)^k, [M, \partial M] \right\rangle.\]
    Therefore, it suffices to show that
    \[\left\langle \left[ g^{-1}(\overline{N_k}) \right], [M, \partial M] \right\rangle = L_k(\iota)_\theta.\]
    One can see that
    \begin{align*}
        \left\langle \left[ g^{-1}(\overline{N_k}) \right], [M, \partial M] \right\rangle
        &= \# g^{-1}(\ol{N_k(f)}) \\
        &= \I_{\R^{(k + 1)r}_+}(g(M), \ol{N_k(f)}) \\
        &= \lk_{\R^{(k + 1)r - 1}}(g(\partial M), N_k(\iota)),
    \end{align*}
    where $\I_X(Y, Z)$ denotes the algebraic number of the intersection of chains $Y$ and $Z$ in $X$.
    Now, we push $g(\partial M)$ back to $f(\partial M) = \iota(\partial M)$ and also push $N_k(\iota)$ forward to the $(-\theta)$-direction.
    These pushings can be simultaneously realized by an ambient isotopy of $\R^{(k + 1)r - 1}$.
    Hence, 
    \begin{align*}
        \left\langle \left[ g^{-1}(\overline{N_k}) \right], [M, \partial M] \right\rangle 
        &= \lk_{\R^{(k + 1)r - 1}}(\iota(\partial M), N_k(\iota)^{-\theta}) \\
        &= L_k(\iota)_{-\theta}.
    \end{align*}
    To complete the proof, we see that
    \[L_k(\iota)_{-\theta} = L_k(\iota)_{\theta}.\]
    By the assumption $H_{r - 1}(\partial M; R) = H_r(\partial M; R) = 0$, we have
    \[L_k(\iota)_{-\theta} = L_k(\iota)_{\theta} - ([\theta(M_k(\iota))] - [-\theta(M_k(\iota))])\]
    in $H_{r - 1}(S^{r - 1}; R)$.
    Furthermore, the term
    \[[\theta(M_k(\iota))] - [-\theta(M_k(\iota))] = [\theta(M_k(\iota))] - (-1)^r [\theta(M_k(\iota))]\]
    vanishes. This is obvious in the unoriented case. Furthermore, in the oriented case, the codimension $r$ is assumed to be even. 
    This completes the proof.
\end{proof}

When $k = 2l$ and $r = 2$, Theorem \ref{theorem:HR-rel-modified} directly implies Ekholm--Sz\H{u}cs' formula.
Let $V$ be a closed oriented $(4l-1)$-manifold and $\iota \colon V \imm \R^{4l+1}$ a generic immersion.

\begin{lemma}\label{lemma:ESz03-indep}
{\it
    If $V$ is $2$-connected, then there always exists a nowhere-vanishing normal vector field $\theta$ of $\iota$ and it is homotopically unique.
    Therefore, $L_{2l}(\iota)_\theta$ is independent of the choice of $\theta$.
}
\end{lemma}

Hence, $L_{2l}(\iota)_\theta$ will be denoted by $L_{2l}(\iota)$ below.
Furthermore, one can show the following.

\begin{lemma}[{\cite[\S8.1, Claim (a) in the proof of Theorem 1.3]{ESz03}}]\label{lemma:ESz03-nullcob}
{\it
    There is a positive integer $d$ such that the disjoint union of $d$ copies of $\iota$ is null-cobordant as an oriented immersion. That is, $d \cdot \iota \colon d \cdot V \imm \R^{4l+1}$ bounds a generic immersion $f \colon M \imm \R^{4l+2}_+$ of a connected compact oriented $4l$-manifold $M$.
}
\end{lemma}

We have the following, combining the above and Theorem~\ref{theorem:gen-ESz03}.

\begin{corollary}[{\cite[Theorem 1.3]{ESz03}}]\label{corollary:ESz03-linking}
{\it
    Let $V$ be a $2$-connected closed oriented $(4l-1)$-manifold and $\iota \colon V \imm \R^{4l+1}$ a generic immersion.
    Choose a positive integer $d$, a connected compact oriented $4l$-manifold $M$, and a generic immersion $f \colon M \imm \R^{4l+2}_+$ such that $\partial M = d \cdot V$ and $f$ is bounded by $d \cdot \iota \colon d \cdot V \imm \R^{4l+1}$.
    Then
    \[L_{2l}(\iota) = \frac{1}{d} \cdot \left\{ (2l + 1) \# N_{2l+1}(f) - \left\langle e(\nu_f | \theta)^{2l}, [M, \partial M] \right\rangle \right\},\]
    where $\theta$ is an arbitrary nowhere-vanishing normal vector field of $\iota$.
}
\end{corollary}

\begin{remark}
    Let $L^{\mathrm{ESz}}_{2l}$ denote the linking invariant with the sign convention of~\cite{ESz03}.
    With our convention, we have
    \[L^{\mathrm{ESz}}_{2l}(\iota) = - L_{2l}(\iota).\]
    In other words, \cite[Theorem 2]{ESz03} shows that
    \[-L^{\mathrm{ESz}}_{2l}(\iota) = \frac{1}{d} \cdot \left\{ (2l + 1) \# N_{2l+1}(f) - \left\langle e(\nu_f | \theta)^{2l}, [M, \partial M] \right\rangle \right\}.\]
\end{remark}

\begin{remark}
    In~\cite{ESz03}, the formula corresponding to Theorem~\ref{theorem:HR-rel-modified} was used to show the invariance of the quantity 
    \[(2l + 1) \# N_{2l+1}(f) - \left\langle e(\nu_f | \theta)^{2l}, [M, \partial M] \right\rangle + d \cdot L^{\mathrm{ESz}}_{2l}(\iota)\]
    up to cobordism.
    Then it was shown that this quantity is zero using a fact about the corresponding cobordism group, and in turn, Corollary~\ref{corollary:ESz03-linking} was proven.
\end{remark}

\end{document}